\documentclass[11pt]{article}
\usepackage[utf8]{inputenc}
\usepackage{amssymb}
\usepackage{amsmath}
\usepackage{amsthm}
\usepackage{enumitem}
\usepackage{mathtools}
\usepackage[colorlinks=true,allcolors=black]{hyperref}
\usepackage{url}
\usepackage{float}
\usepackage{tikz}
\tikzset{every picture/.style={line width=0.75pt}}
\usepackage[labelfont=bf]{caption}
\usepackage{setspace}

\allowdisplaybreaks

\newtheorem{thm}{Theorem}
\newtheorem{prop}{Proposition}
\newtheorem{definition}{Definition}
\newtheorem{rem}{Remark}
\newtheorem{lem}{Lemma}

\newcommand{\esssup}{\mathrm{ess\, sup}}
\newcommand{\essinf}{\mathrm{ess\, inf}}
\newcommand{\Addresses}{{
  \bigskip
  \footnotesize

  CB: \textsc{Department of Mathematics, University College London, 25 Gordon St, London, WC1H 0AY}\par\nopagebreak
  \textit{Email address:} \texttt{c.bellettini@ucl.ac.uk}

  \medskip

  KMS: \textsc{Department of Mathematics, Johns Hopkins University, 404 Krieger Hall, 3400 N. Charles Street, Baltimore, MD 21218, USA}\par\nopagebreak
  \textit{Email address:} \texttt{kmarsh34@jh.edu}
}}

\begin{document}

\title{\textbf{ON ISOLATED SINGULARITIES AND GENERIC REGULARITY OF MIN-MAX CMC HYPERSURFACES}}
\author{COSTANTE BELLETTINI \& KOBE MARSHALL-STEVENS}
\date{\vspace{-5ex}}
\maketitle

\begin{abstract}
\noindent In compact Riemannian manifolds of dimension 3 or higher with positive Ricci curvature, we prove that every constant mean curvature hypersurface produced by the Allen--Cahn min-max procedure in \cite{bellettini-wickramasekera} (with constant prescribing function) is a local minimiser of the natural area-type functional around each isolated singularity. In particular, every tangent cone at each isolated singularity of the resulting hypersurface is area-minimising. As a consequence, for any real $\lambda$ we show, through a surgery procedure, that for a generic 8-dimensional compact Riemannian manifold with positive
Ricci curvature there exists a closed embedded smooth hypersurface of constant mean curvature $\lambda$; the minimal case ($\lambda = 0$) of this result was obtained in \cite{C-L-S}. 
\end{abstract}

\tableofcontents

\section{Introduction}\label{sec: introduction}

The Allen--Cahn min-max procedure in \cite{bellettini-wickramasekera}, with constant prescribing function, shows that in a compact $8$-dimensional Riemannian manifold there exists a quasi-embedded hypersurface of constant mean curvature, with a singular set consisting of finitely many points (see Subsection \ref{subsec: Allen--Cahn minmax} below for a precise description). One may thus conjecture the existence of a smoothly embedded constant mean curvature hypersurface in all $8$-dimensional Riemannian manifolds under some assumption on the metric; for example a genericity assumption. As a first step, we resolve this for manifolds with positive Ricci curvature:

\begin{thm}\label{thm: generic regularity in dimension 8 in positive ricci} Let $N$ be a smooth compact $8$-dimensional manifold and $\lambda \in \mathbb{R}$. There is an open and dense subset, $\mathcal{G}$, of the smooth metrics with positive Ricci curvature such that for each $g \in \mathcal{G}$, there exists a closed embedded smooth hypersurface of constant mean curvature $\lambda$ in $(N,g)$.
\end{thm}

We actually prove more general results valid in higher dimensions, showing the generic existence of a closed embedded hypersurface of constant mean curvature, with singular set of codimension $7$, containing no isolated singularities with regular tangent cones. Indeed, Theorem \ref{thm: generic regularity in dimension 8 in positive ricci} is a consequence of the following:

\begin{thm}\label{thm: generic removability of isolated singularities with regular cones in positive ricci}
Let $N$ be a smooth compact manifold of dimension $n+1 \geq 3$ and $\lambda \in \mathbb{R}$.  There is a dense subset, $\mathcal{G}$, of the smooth metrics with positive Ricci curvature such that for each $g \in \mathcal{G}$, there exists a closed embedded hypersurface of constant mean curvature $\lambda$, smooth away from a closed singular set of Hausdorff dimension at most $n-7$, containing no isolated singularities with regular tangent cones in $(N,g)$.
\end{thm}

\begin{rem}
Let $\mathrm{Met}_{\mathrm{Ric}_g > 0}^{k,\alpha}(N)$ for each $k \geq 1$ and $\alpha \in (0,1)$  denote the open subset of Riemannian metrics of regularity $C^{k,\alpha}$ on $N$ with positive Ricci curvature. In the proof of Theorem \ref{thm: generic removability of isolated singularities with regular cones in positive ricci} we in fact establish that there exists a dense set, $\mathcal{G}_k \subset \mathrm{Met}^{k,\alpha}_{\mathrm{Ric}_g > 0}(N)$, such that for each $g \in \mathcal{G}_k$, the same existence conclusion of Theorem \ref{thm: generic removability of isolated singularities with regular cones in positive ricci} holds.
\end{rem}

The focus of the present work concerns the generic regularity of constant mean curvature hypersurfaces in ambient dimensions $8$ or higher and is related to previous work concerning minimal hypersurfaces. We now briefly summarise the works on generic regularity for minimal hypersurfaces in ambient dimension $8$:
\begin{itemize}
    \item The existence of the Simons cone, introduced in \cite{Simons68}, showed that stable minimal hypersurfaces may admit isolated singularities in dimension $8$. 
    
    \item The generic regularity of area-minimisers in each non-zero homology class was established in \cite{Smale}, using the fundamental result of \cite{hardt-simon}.

    \item In \cite{C-L-S}, using the Almgren--Pitts min-max procedure and \cite{hardt-simon} (in particular the extension in \cite{liu}), the existence of smooth minimal hypersurfaces was established in manifolds equipped with a generic metric of positive Ricci curvature.

    \item In \cite{li-wang21} it was shown that every $8$-dimensional closed manifold equipped with a generic metric (with no curvature assumption) admits a smooth minimal hypersurface. See also \cite{li-wang22}, where it is shown that for a generic metric, every embedded locally stable minimal hypersurface is smooth in dimension $8$.
\end{itemize}

Both \cite{Smale} and \cite{C-L-S} exploit local foliations by area-minimising hypersurfaces, provided by \cite{hardt-simon}, allowing for a surgery procedure to be established in order to perturb the metric locally. Recently an analogous foliation was established in \cite{Kostas23} (see also \cite{BL24}) for constant mean curvature hypersurfaces that locally minimise a prescribed mean curvature functional to at least one side. Such a foliation, to one side of the hypersurface, provides a natural way to perturb away an isolated singularity via a surgery procedure.

\bigskip

In order to produce such a foliation, one needs to establish that the tangent cones to isolated singularities of a candidate hypersurface are area-minimising; we establish this for all hypersurfaces of constant mean curvature arising from the Allen--Cahn min-max procedure in \cite{bellettini-wickramasekera} (with constant prescribing function) in manifolds with positive Ricci curvature.

\bigskip

The Allen--Cahn min-max procedure in \cite{bellettini-wickramasekera} produces in the first instance a quasi-embedded hypersurface of constant mean curvature with a possibly non-empty singular set of codimension at least $7$. The results of 
\cite{Bellettini-Workman} then establish that in manifolds with positive Ricci curvature the above hypersurface of constant mean curvature in fact attains the min-max value (in a manner made precise below) and is embedded, with the above dimension bound on the singular set. We thus work with this hypersurface as a candidate to perturb away isolated singularities via a surgery procedure.

\bigskip

For a compact Riemannian manifold $(N,g)$ and $\lambda \in \mathbb{R}$ we define the $\mathcal{F}_\lambda$ functional on a Caccioppoli set, $F \subset N$, by
\begin{equation*}
    \mathcal{F}_\lambda(F) = \mathrm{Per}_g(F) - \lambda \mathrm{Vol}_g(F).
\end{equation*}
Recall that, as shown in \cite[Proposition B.1]{bellettini-wickramasekera-cmc-reg}, smooth constant mean curvature hypersurfaces are locally $\mathcal{F}_\lambda$-minimising. Our main technical result shows that this $\mathcal{F}_\lambda$-minimisation also holds in sufficiently small balls around isolated singularities for constant mean curvature hypersurfaces produced by the Allen--Cahn min-max in manifolds with positive Ricci curvature:

\begin{thm}\label{thm: local minimisation in positive ricci} Let $(N,g)$ be a smooth compact Riemannian manifold of dimension $n+1 \geq 3$, with positive Ricci
curvature, and $\lambda \in \mathbb{R}$. The one-parameter Allen–Cahn min-max in \cite{bellettini-wickramasekera}, with constant prescribing function $\lambda$, produces a closed embedded hypersurface of constant mean curvature $\lambda$, smooth away from a closed singular set of Hausdorff dimension at most $n - 7$, locally $\mathcal{F}_\lambda$-minimising in balls around each isolated singularity. Precisely, this hypersurface arises as the boundary of a Caccioppoli set, $E \subset N$, and for each isolated singularity, $p \in \overline{\partial^* E} \setminus \partial^* E$, there exists an $r > 0$ such that 
\begin{equation*}
    \mathcal{F}_\lambda(E) = \inf_{G \in \mathcal{C}(N)} \{ \mathcal{F}_\lambda(G) \: | \: G \setminus B_r(p) = E \setminus B_r(p) \}, 
\end{equation*}
where $\mathcal{C}(N)$ is the set of Caccioppoli sets in $N$. Consequently, the hypersurface has area-minimising tangent cones at each isolated singularity. 
\end{thm}

\begin{rem}\label{rem: CLS comparison}
In the case that $\lambda = 0$, Theorem \ref{thm: local minimisation in positive ricci} shows that minimal hypersurfaces produced by the Allen--Cahn min-max procedure in \cite{guaraco} in manifolds with positive Ricci curvature are in fact locally area-minimising (to both sides), as opposed to just one-sided homotopy minimising as obtained via the Almgren--Pitts min-max procedure in the results of \cite{C-L-S}. These minimising properties immediately pass to tangent cones. Stable regular minimal cones that do not minimise area are known to exist, for example the Simons cone
    \begin{equation*}
        C^{1,5} = \left\{(x,y) \in \mathbb{R}^2 \times \mathbb{R}^6 \: | \: 5|x|^2 = |y|^2 \right\},
    \end{equation*}
is stable and one-sided area-minimising, but is not area-minimising to the other side (see \cite{Lawlor91}). Such a cone is not explicitly ruled out in \cite{C-L-S} from arising as a tangent cone to a min-max minimal hypersurface at an isolated singularity. Theorem \ref{thm: local minimisation in positive ricci} precludes such tangent cones. We also note that in the Allen--Cahn framework, obtaining an absolute area-minimisation property, as opposed to a homotopic minimisation property, appears to be natural; indeed the relevant space where the min-max is carried out is $W^{1,2}(N)$, which is contractible.
\end{rem}

The present work fits into the broader context of generic regularity for solutions to geometric variational problems. In addition to the works on generic regularity for minimal hypersurfaces in dimension $8$ as discussed above, we provide a non-exhaustive summary of work on generic regularity for geometric variational problems:
\begin{itemize}
    \item In \cite{C-M-S} and \cite{C-M-S-2} the generic regularity of area-minimising minimal hypersurfaces is established, in various settings, up to ambient dimension $10$.

    \item In \cite{White1985} and \cite{White2019} it is shown that for a generic ambient metric, every $2$-dimensional surface (integral current or flat chain mod $2$) without boundary that minimises area in its homology class has support equal to a smoothly embedded minimal surface.

    \item In \cite{Moore2006} and \cite{Moore2017} it is shown that for a generic ambient metric, parameterised $2$-dimensional minimal surfaces are free of branch points.

    \item In \cite{CCMS2020} and \cite{CCMS2022} an analogy was established between mean
    curvature flow with generic initial data and the generic regularity of area-minimising hypersurfaces.

    \item In \cite{F-R-S} the generic regularity of free boundaries for the obstacle problem is established up to ambient dimension $4$.
\end{itemize}

\subsection{Notation}\label{subsec: notation}

We now collect some notation and definitions that will be used throughout the paper:

\begin{itemize}
    \item Unless otherwise stated, throughout the paper we let $(N^{n+1},g)$ be a compact (with empty boundary), Riemannian manifold of dimension $n+1  \geq 3$, with positive Ricci curvature, $\mathrm{Ric}_g > 0$. We will implicitly assume $N$ is connected.
    
    \item We let $M \subset N$ be a non-empty, smooth, two-sided, separating, embedded hypersurface of constant mean curvature $\lambda \in \mathbb{R}$, with closed singular set, $\mathrm{Sing}(M) = \overline{M} \setminus M$, of Hausdorff dimension at most $n-7$. As $M$ is a separating hypersurface, we may write $N \setminus \overline{M}$ as two disjoint open sets, $E$ and $ N \setminus E$, with common boundary $\overline{M}$.

    \item We say that $p \in \mathrm{Sing}(M)$ is an \textit{isolated singularity} of $M$ if there exists some $R_p > 0$ such that  $\mathrm{Sing}(M) \cap \overline{B_{R_p}}(p) = \{p\}$, i.e.~such that $M \cap \overline{B_{R_p}}(p)$ is smooth. Moreover, we say that a multiplicity one tangent cone, $C_p$, to $M$ at an isolated singularity, $p \in \mathrm{Sing}(M)$, is a \textit{regular tangent cone} if $\mathrm{Sing}(C_p) = \{0\}$, where $0$ here denotes the origin in $\mathbb{R}^{n+1}$. We note that the tangent cone to $M$ at an isolated singularity with regular tangent cone is necessarily unique by the work of \cite{simon}.

    \item A measurable set $F \subset N$ is a \textit{Caccioppoli set} if
    \begin{equation*}
    \mathrm{Per}_g(F) = \sup \left\{\int_E div_g \varphi \: \bigg| \: \varphi \in \Gamma(TN), ||\varphi||_\infty \leq 1 \right\} < \infty,
    \end{equation*}
    where $div_g$ is the divergence with respect to the metric $g$, $\Gamma(TN)$ is the set of vector fields on $N$ and $||\cdot ||_\infty$ denotes the supremum norm. We denote by $\mathcal{C}(N)$ the set of Caccioppoli sets in $N$. 
    
    \item By De Giorgi’s Structure Theorem we have
    that the distributional derivative $D_g \chi_F$ (a Radon measure) of a Caccioppoli set $F$ is given by $D_g \chi_F = -\nu_F \mathcal{H}^n \llcorner \partial^{*} F$, where $\partial^{*} F$ is the reduced boundary of $F$ (an
    $n$-rectifiable set), $\mathcal{H}^n$ is the $n$-dimensional Hausdorff measure, $\nu_F$ is the unit normal to $\partial^{*}F$ pointing inside $F$ defined $\mathcal{H}^n$-a.e.~and $\chi_F$ is the indicator function of the set $F$. Our main reference for Caccioppoli sets will be \cite{maggi_2012}. 

    \item We define the following prescribed mean curvature functional on measurable subsets of $N$: for a measurable set $F \subset N$ we let
    \begin{equation*}
    \mathcal{F}_\lambda(F) = \mathrm{Per}_g(F) - \lambda \mathrm{Vol}_g(F) + \frac{\lambda}{2}\mathrm{Vol}_g(N),
    \end{equation*}
    where $\mathrm{Vol}_g$ denotes the $\mathcal{H}^{n+1}$ measure with respect to the metric $g$. This definition differs from that of Section \ref{sec: introduction} by the addition of the constant $\frac{\lambda}{2}\mathrm{Vol}_g(N)$. Note however this does not affect the set of critical points of the functional $\mathcal{F}_\lambda$ and is made purely for convenience of notation in forthcoming computations.

    \item With the above two definitions in mind, we take throughout $\nu$ to be the global unit normal to $M$ pointing into $E$ and write $M = \partial^* E$; by viewing $M = \partial^* E$ as the reduced boundary of the Caccioppoli set $E$ we have that $E$ is a critical point for $\mathcal{F}_\lambda$.

    \item We will use the notions of integral currents and varifolds throughout the paper; a reference for the notation and definitions used here may be found in \cite{simongmt}.

    \item Let $\mathrm{dist}_N$ denote the Riemannian distance on $N$ and define the function $d_{\overline{M}}$ on $N$ by setting $d_{\overline{M}} (\cdot) = \mathrm{dist}_N (\cdot, \overline{M})$; we then have that $d_{\overline{M}}$ is Lipschitz on $N$ (with Lipschitz constant equal to $1$) and, as $N$ is complete, for each $x \in N$ there exists a geodesic realising the value $d_{\overline{M}} (x)$. Let $d(N) = \sup_{x,y \in N} d_N(x,y)$ be the diameter of $N$, which is finite as $N$ is closed.

    \item We fix $R_l > 0$ such that for every $R \in (0,R_l)$ and each point $p \in N$ we have that the ball $B_R(p) \subset N$ of radius $R$ centred at a point $p \in N$ is $2$-bi-Lipschitz diffeomorphic, via a geodeisic normal coordinate chart, to the Euclidean ball, $B^{\mathbb{R}^{n+1}}_R(0) \subset \mathbb{R}^{n+1}$ of radius $R$ centred at the origin in $\mathbb{R}^{n+1}$.

    \item For $\varepsilon \in (0,1)$ we denote the \textit{Allen--Cahn energy} of a function $u \in W^{1,2}(N)$ by 
    \begin{equation*}
    \mathcal{E}_\varepsilon(u) =  \frac{1}{2\sigma} \int_N e_\varepsilon(u) = \frac{1}{2\sigma} \int_N \frac{\varepsilon}{2}|\nabla u|^2 + \frac{W(u)}{\varepsilon}, 
    \end{equation*} 
    where $W$ is a $C^2$ double well potential with non-degenerate minima at $\pm 1$, $c_W \leq W''(t) \leq C_W$ for constants $c_W, C_W > 0$ for all $t \in \mathbb{R} \setminus [-2,2]$ and $\sigma = \int_{-1}^1 \sqrt{W(t)/2} \, dt$. We then consider the following functional, which we shall refer to simply as the \textit{energy}, defined on functions $u \in W^{1,2}(N)$ by 
    \begin{equation*}
    \mathcal{F}_{\varepsilon,\lambda}(u) = \mathcal{E}_\varepsilon(u) - \frac{\lambda}{2}\int_N u.
    \end{equation*}

\end{itemize}

\subsection{Allen--Cahn min-max}\label{subsec: Allen--Cahn minmax}

Let $(N^{n+1},g)$ a closed, connected Riemannian manifold of dimension $n+1 \geq 3$. The Allen--Cahn min-max procedure in \cite{bellettini-wickramasekera} produces a hypersurface with mean curvature prescribed by an arbitrary non-negative Lipschitz function, and provides sharp dimension bounds on the singular set.

\bigskip

We recall this procedure in the case relevant to this work, in which the metric is assumed to have positive Ricci curvature, and the prescribing function is a non-negative constant $\lambda$; for producing a candidate hypersurface of constant mean curvature $\lambda < 0$ one can consider $-\lambda$ in the results below (this amounts to a change in the choice of unit normal to $M$). Thus, without loss of generality, for the rest of the paper we assume that $\lambda \geq 0$. The constant mean curvature hypersurfaces produced will, after proving Theorem \ref{thm: local minimisation in positive ricci}, be candidates for our surgery procedure established in Section \ref{sec: cmc surgery}.

\bigskip

For $\varepsilon \in (0,1)$ there exist two constant functions, $a_\varepsilon$ and $b_\varepsilon$, on $N$, stable critical points of $\mathcal{F}_{\varepsilon,\lambda}$ with $-1 < a_\varepsilon < -1 + c\varepsilon$, $1 < b_\varepsilon < 1 + c\varepsilon$ and $a_\varepsilon \rightarrow -1$, $b_\varepsilon \rightarrow 1$ on $N$ as $\varepsilon \rightarrow 0$, where $c > 0$ depends on $W$ and $\lambda$. These functions are constructed in \cite[Section 5]{bellettini-wickramasekera} by means of the negative gradient flow, through constant functions, of $\mathcal{F}_{\varepsilon,\lambda}$ starting at $\pm 1$. In particular, there are continuous paths of functions in $W^{1,2}(N)$ connecting $a_\varepsilon$ to $-1$ and $b_\varepsilon$ to $1$, provided by the negative gradient flow of $\mathcal{F}_{\varepsilon,\lambda}$, with energy along these paths bounded from above by $\mathcal{F}_{\varepsilon,\lambda}(-1)$ and $\mathcal{F}_{\varepsilon,\lambda}(1)$ respectively.

\bigskip

For every $\varepsilon \in (0,1)$ a min-max critical point, $u_{\varepsilon} \in W^{1,2}(N)$,
of $\mathcal{F}_{\varepsilon,\lambda}$ may be constructed, with $\sup_N |u_{\varepsilon}|$ uniformly bounded and $\mathcal{E}_{\varepsilon}(u_{\varepsilon})$ uniformly bounded from above and below by positive constants. This is done by applying a mountain pass lemma, for paths between the two stable critical points $a_\varepsilon$ and $b_\varepsilon$, based
on the fact that $\mathcal{F}_{\varepsilon,\lambda}$ satisfies a Palais–Smale condition. The Morse index of the $u_{\varepsilon}$
will then automatically be equal $1$ by virtue of the fact that if $\mathrm{Ric}_g > 0$ then, as noted in \cite[Remark 6.7]{bellettini-wickramasekera}, $a_\varepsilon$ and $b_\varepsilon$ are the only stable critical points of $\mathcal{F}_{\varepsilon,\lambda}$ (this is a general fact for semi-linear elliptic equations on compact
manifolds of nonnegative Ricci curvature, e.g.~see \cite{Jimbo1984}).

\bigskip

By general principles, the bounds above imply
that there exist a sequence $\varepsilon_j \rightarrow 0$, a non-zero Radon measure, $\mu$, on $N$ and a function, $u_\infty \in BV(N)$
with $u_\infty = \pm 1$ for a.e. $x \in N$, such that for the min-max critical points, $\{u_{\varepsilon_j}\}_{j = 1}^\infty$, we have as $\varepsilon_j \rightarrow 0$ that $e_{\varepsilon_j}(u_{\varepsilon_j}) \rightarrow \mu$
weakly as measures and $u_{\varepsilon_j} \rightarrow u_{\infty}$ strongly in $L^1(N)$. Defining $E = \{ u_\infty = 1 \}$, we note that $E$ is a Caccioppoli set with its reduced boundary $\partial^* E \subset \mathrm{Spt}\mu$; moreover, as $\mathrm{Ric}_g > 0$ we have that $E \neq \emptyset$ by \cite[Remark 6.7]{bellettini-wickramasekera}.

\bigskip

In \cite{bellettini-wickramasekera}, relying on the combined works of \cite{Hutchinson-Tonegawa} and \cite{Roger-Tonegawa}, it is then established
that $\mu$ is the weight measure of an integral $n$-varifold $V$ with the following properties:
\begin{itemize}
    \item $V = V_0 + V_\lambda$.

    \item $V_0$ is a (possibly zero) stationary integral $n$-varifold on $N$ with $\mathrm{Sing}(V_0)$ empty if $2 \leq n \leq 6$, $\mathrm{Sing}(V_0)$ discrete if $n = 7$ and $\mathrm{Sing}(V_0)$ of Hausdorff dimension $\leq n - 7$ when $n \geq 8$.

    \item $V_\lambda = |\partial^* E| \neq 0$ (by \cite[Remark 6.7]{bellettini-wickramasekera} as $\mathrm{Ric}_g > 0$), the multiplicity one $n$-varifold associated with the reduced boundary $\partial^* E$. Then $\mathrm{Spt}(V_\lambda) \setminus \mathrm{Sing}(V_\lambda) = \overline{\partial^* E} \setminus \mathrm{Sing}(V_\lambda)$ is a quasi-embedded hypersurface of constant mean curvature $\lambda$ with respect to the unit normal pointing into $E$; moreover, $\mathrm{Sing}(V_\lambda)$ is empty if $2 \leq n \leq 6$, $\mathrm{Sing}(V_\lambda)$ is discrete if $n = 7$ and $\mathrm{Sing}(V_\lambda)$ is of Hausdorff dimension $\leq n - 7$ when $n \geq 8$.
\end{itemize}
In the above, quasi-embedded means that near every non-embedded point of $\overline{\partial^* E} \setminus \mathrm{Sing}(V_\lambda)$, $\overline{\partial^* E}$ is the union of two embedded $C^{2,\alpha}$ disks intersecting tangentially, with each disk lying on one side of the other. Furthermore, for a varifold $W$ (for example $V_0$ and $V_\lambda$ above) by letting $\text{gen-reg}(W)$ be the set of quasi-embedded points of $\mathrm{\mathrm{Spt}}||W||$ we have defined $\mathrm{Sing}(W) = \mathrm{\mathrm{Spt}}||W|| \setminus \text{gen-reg}(W)$ (thus for an embedded hypersurface this agrees with the definition in Subsection \ref{subsec: notation} above). For a more detailed description of the definitions and results above we refer to \cite[Sections 3 and 4]{bellettini-wickramasekera}.

\bigskip

As in \cite[Theorem 2]{Bellettini-Workman}, the path we exhibit, for all $\varepsilon > 0$ sufficiently small, in Subsection \ref{subsec: sliding the one-dimensional profile on M} with the upper energy bounds provided by Lemma \ref{lem: sliding path upper energy bounds} (depicted by the dashed lines in Figure \ref{fig: path}) between $1$ and $-1$, along with short paths of constant functions connecting $1$ to $b_\varepsilon$ and $-1$ to $a_\varepsilon$, proves that $V_0 = 0$ (i.e.~that the min-max procedure produces no minimal piece in manifolds with positive Ricci curvature). Using this we then note that as we have $e_{\varepsilon_j} \rightarrow \mu$ as $\varepsilon_j \rightarrow 0$ and $E = \{ u_\infty = 1 \}$ we have
\begin{equation}\label{eqn: min-max critical points converge to surface energy}
    \mathcal{F}_{\varepsilon_j,\lambda}(u_{\varepsilon_j}) \rightarrow \mathcal{F}_{\lambda}(E) \text{ as } \varepsilon_j \rightarrow 0,
\end{equation}
i.e.~that the constant mean curvature hypersurface attains the min-max value. In the proof of Theorem \ref{thm: local minimisation in positive ricci}, under the contradiction assumption that our candidate hypersurface produced by the above procedure does not satisfy a local minimisation property, we will exploit (\ref{eqn: min-max critical points converge to surface energy}) by constructing continuous paths of functions in $W^{1,2}(N)$ for all $\varepsilon > 0$ sufficiently small, admissible in the min-max construction above, with energy along the paths bounded above by a value strictly below $\mathcal{F}_\lambda(E)$ (independently of $\varepsilon$); thus violating the min-max characterisation of $E$.

\bigskip

In fact, by \cite[Theorem 4]{Bellettini-Workman}, for $\lambda \neq 0$ we have that $\partial^*E$ is embedded (rather than quasi-embedded). Moreover, $\partial^* E$ is connected, has index $1$ and is  separating in the sense that $N \setminus \partial E$ may be written as the union of two disjoint open sets whose common boundary is $\partial E$. The same results hold in the case that $\lambda = 0$ by combining \cite[Theorem A]{guaraco} with \cite[Theorem 1.8]{bellettini}.

\bigskip

To summarise, we know that for $(N^{n+1},g)$ a closed, connected Riemannian manifold of dimension $n+1  \geq 3$, with positive Ricci curvature, the properties of $M$ as stated in Subsection \ref{subsec: notation} hold for any constant mean curvature hypersurface produced from the Allen--Cahn min-max construction of \cite{bellettini-wickramasekera} with prescribing function taken to be a constant.

\subsection{Strategy}\label{subsec: strategy}

The main strategy of the present work can be split into three distinct steps:
\begin{enumerate}
    \item \textbf{Surgery procedure:} We show how to perturb constant mean curvature hypersurfaces that are locally $\mathcal{F}_\lambda$-minimising around isolated singularities with regular tangent cones, resulting in a smooth hypersurface with constant mean curvature.

    \item \textbf{Functions to geometry:} We relate the local geometric behaviour of hypersurfaces produced by the Allen--Cahn min-max procedure to the $\varepsilon \rightarrow 0$ energy properties of specific $W^{1,2}(N)$ functions.
    
    \item \textbf{Paths of functions:} By exhibiting an admissible min-max path, we establish the energy properties for the functions from Step $2$ hold, which enables us to conclude that hypersurfaces generated through the Allen-Cahn min-max procedure in positive Ricci curvature are locally $\mathcal{F}_\lambda$-minimising around isolated singularities.
\end{enumerate}

We now sketch these steps in more detail: 

\bigskip

1.~\textbf{Surgery procedure:}
As mentioned in Section \ref{sec: introduction}, a local foliation around a hypersurface provides a natural way to perturb away an isolated singularity via a surgery procedure. Using the work of \cite{hardt-simon}, in both \cite{Smale} and \cite{C-L-S}, isolated singularities with regular tangent cones to locally area-minimising hypersurfaces are perturbed away by a “cut-and-paste" gluing along with a conformal change of the initial metric. We take a similar approach here. Near an isolated singularity with regular tangent cone of a locally $\mathcal{F}_\lambda$-minimising hypersurface of constant mean curvature $\lambda$, in \cite{Kostas23} (see also \cite{BL24}) it is shown that there is a foliation around this hypersurface (to either side) by smooth hypersurfaces of constant mean curvature $\lambda$. In Section \ref{sec: cmc surgery}, using this foliation, we establish a surgery procedure to perturb away isolated singularities with regular tangent cones of locally $\mathcal{F}_\lambda$-minimising constant mean curvature hypersurfaces.

\bigskip

This is achieved by first constructing, via a “cut-and-paste" gluing, a smooth hypersurface close in Hausdorff distance to the original one inside a chosen ball (which may be taken arbitrarily small), with constant mean curvature $\lambda$ outside of an annulus and inside which the original hypersurface is assumed to be smooth; this construction is depicted in Figure \ref{fig: surgery}.

\bigskip

It then remains to perturb the metric inside of this annulus so that the newly constructed hypersurface has constant mean curvature $\lambda$ everywhere. This is achieved by an appropriate choice of function for conformal change of the original metric; with the resulting metric arbitrarily close, in the $C^{k,\alpha}$-norm, to the original. The result is then a smooth hypersurface of constant mean curvature $\lambda$ with respect to the new metric, agreeing with the original hypersurface outside of the chosen ball.

\bigskip

2.~\textbf{Functions to geometry:} 
We aim to show that, when the ambient metric is assumed to have positive Ricci curvature, the hypersurfaces of constant mean curvature $\lambda$ produced by the Allen--Cahn min-max procedure in \cite{bellettini-wickramasekera} are in fact locally $\mathcal{F}_\lambda$-minimising. In order to do this we first relate the local $\mathcal{F}_\lambda$-minimisation we desire to a property about the energy of specific $W^{1,2}(N)$ functions defined from such a hypersurface as the parameter $\varepsilon \rightarrow 0$.

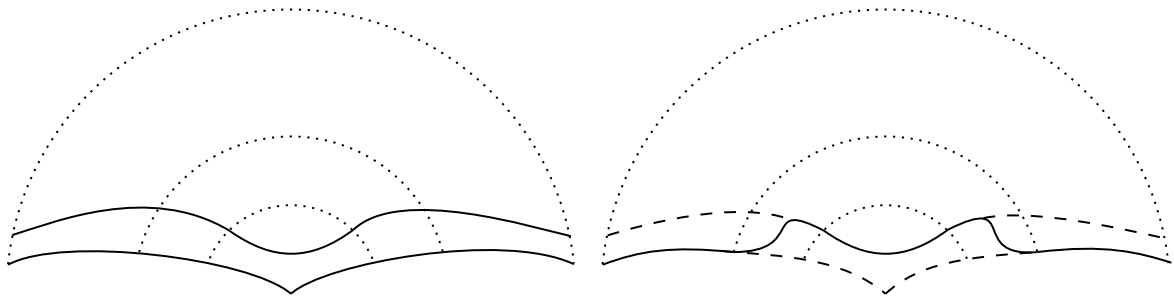
\begin{figure}[H]
    \centering
\captionsetup{justification=justified,margin=1cm}
\begin{tikzpicture}[scale=-0.02, xscale=-1]

\draw    (288.61,125.31) .. controls (317.94,145.98) and (342.61,145.98) .. (371.27,124.64) ;
\draw  [draw opacity=0][dash pattern={on 0.84pt off 2.51pt}] (275.68,143.71) .. controls (284.78,122.93) and (305.53,108.42) .. (329.67,108.42) .. controls (354.96,108.42) and (376.53,124.35) .. (384.88,146.73) -- (329.67,167.33) -- cycle ; \draw  [dash pattern={on 0.84pt off 2.51pt}] (275.68,143.71) .. controls (284.78,122.93) and (305.53,108.42) .. (329.67,108.42) .. controls (354.96,108.42) and (376.53,124.35) .. (384.88,146.73) ;  
\draw  [draw opacity=0][dash pattern={on 0.84pt off 2.51pt}] (228.76,139.6) .. controls (240.91,95.27) and (281.49,62.71) .. (329.67,62.71) .. controls (377.81,62.71) and (418.36,95.22) .. (430.55,139.49) -- (329.67,167.33) -- cycle ; \draw  [dash pattern={on 0.84pt off 2.51pt}] (228.76,139.6) .. controls (240.91,95.27) and (281.49,62.71) .. (329.67,62.71) .. controls (377.81,62.71) and (418.36,95.22) .. (430.55,139.49) ;  
\draw    (144.33,128.33) .. controls (171.67,120) and (235.4,93.4) .. (288.61,125.31) ;
\draw    (141.67,148) .. controls (177.5,128.5) and (301.75,143.5) .. (329.67,167.33) ;
\draw    (329.67,167.33) .. controls (351.75,147) and (470.37,125.37) .. (517.67,147.67) ;
\draw    (371.27,124.64) .. controls (402.2,95.4) and (491.67,123.67) .. (515.67,129) ;
\draw  [draw opacity=0][dash pattern={on 0.84pt off 2.51pt}] (141.61,147.67) .. controls (151.45,52.49) and (231.89,-21.73) .. (329.67,-21.73) .. controls (427.21,-21.73) and (507.5,52.13) .. (517.65,146.98) -- (329.67,167.33) -- cycle ; \draw  [dash pattern={on 0.84pt off 2.51pt}] (141.61,147.67) .. controls (151.45,52.49) and (231.89,-21.73) .. (329.67,-21.73) .. controls (427.21,-21.73) and (507.5,52.13) .. (517.65,146.98) ;  
\end{tikzpicture}
\hspace{1mm}
    \begin{tikzpicture}[scale=-0.02, xscale=-1]

\draw    (288.61,125.31) .. controls (317.94,145.98) and (342.61,145.98) .. (371.27,124.64) ;
\draw  [draw opacity=0][dash pattern={on 0.84pt off 2.51pt}] (275.68,143.71) .. controls (284.78,122.93) and (305.53,108.42) .. (329.67,108.42) .. controls (354.06,108.42) and (374.99,123.24) .. (383.94,144.37) -- (329.67,167.33) -- cycle ; \draw  [dash pattern={on 0.84pt off 2.51pt}] (275.68,143.71) .. controls (284.78,122.93) and (305.53,108.42) .. (329.67,108.42) .. controls (354.06,108.42) and (374.99,123.24) .. (383.94,144.37) ;  
\draw  [draw opacity=0][dash pattern={on 0.84pt off 2.51pt}] (228.76,139.6) .. controls (240.91,95.27) and (281.49,62.71) .. (329.67,62.71) .. controls (377.85,62.71) and (418.42,95.27) .. (430.58,139.6) -- (329.67,167.33) -- cycle ; \draw  [dash pattern={on 0.84pt off 2.51pt}] (228.76,139.6) .. controls (240.91,95.27) and (281.49,62.71) .. (329.67,62.71) .. controls (377.85,62.71) and (418.42,95.27) .. (430.58,139.6) ;  
\draw    (228.76,139.6) .. controls (277.75,139.5) and (246.25,103.5) .. (288.61,125.31) ;
\draw  [dash pattern={on 4.5pt off 4.5pt}]  (145,128.5) .. controls (172.33,120.17) and (230.75,105) .. (266.75,118) ;
\draw  [dash pattern={on 4.5pt off 4.5pt}]  (221.75,139) .. controls (260.25,142) and (301.75,143.5) .. (329.67,167.33) ;
\draw    (141.67,148) .. controls (167,138.5) and (187.5,135.5) .. (228.76,139.6) ;
\draw    (371.27,124.64) .. controls (419.75,100.5) and (383.75,143.5) .. (430.58,139.6) ;
\draw  [dash pattern={on 4.5pt off 4.5pt}]  (329.67,167.33) .. controls (351.75,147) and (381.75,143) .. (430.58,139.6) ;
\draw    (430.58,139.6) .. controls (476.25,134) and (498.25,140) .. (519.67,146.67) ;
\draw  [dash pattern={on 4.5pt off 4.5pt}]  (394.75,117) .. controls (427.75,110) and (491.5,125.17) .. (515.5,130.5) ;
\draw  [draw opacity=0][dash pattern={on 0.84pt off 2.51pt}] (141.61,147.67) .. controls (151.45,52.49) and (231.89,-21.73) .. (329.67,-21.73) .. controls (427.21,-21.73) and (507.5,52.13) .. (517.65,146.98) -- (329.67,167.33) -- cycle ; \draw  [dash pattern={on 0.84pt off 2.51pt}] (141.61,147.67) .. controls (151.45,52.49) and (231.89,-21.73) .. (329.67,-21.73) .. controls (427.21,-21.73) and (507.5,52.13) .. (517.65,146.98) ;  
\end{tikzpicture}
    \caption{In both graphics the innermost two thinly dotted curves depict an annulus around an isolated singularity of a constant mean curvature hypersurface and the outermost thinly dotted curves depict the boundary of the ball in which the foliations will be defined. In the left-hand graphic the lower solid curve depicts an isolated singularity (assumed to have regular tangent cone) of a constant mean curvature hypersurface and the upper solid curve depicts, under the assumption that the lower singular hypersurface is locally $\mathcal{F}_\lambda$-minimising, a smooth constant mean curvature hypersurface in the one-sided foliation provided by \cite{Kostas23}. The solid curve in the right-hand graphic depicts the smooth hypersurface constructed by gluing both of the hypersurfaces in the left-hand graphic. This gluing is done in such a way that the hypersurface outside of the larger ball in the annulus agrees with the singular one and inside of the smaller ball in the annulus agrees with the hypersurface provided by the foliation; with the thick dashed lines depicting the pieces of the hypersurfaces in the left-hand graphic not included in the construction in the right. The resulting construction in the right-hand graphic is then, after a suitable metric perturbation, the desired smooth constant mean curvature hypersurface.}
    \label{fig: surgery}
\end{figure}

Rather than work directly with the min-max critical points of $\mathcal{F}_{\varepsilon,\lambda}$ produced in \cite{bellettini-wickramasekera}, we instead introduce, in Subsection \ref{subsec: one dimensional profile}, a function $v_\varepsilon \in W^{1,2}(N)$, which we call the \textit{one-dimensional profile}. This function, as in other related works, is constructed by placing a truncated version of the one-dimensional solution to the Allen--Cahn equation, $\overline{\mathbb{H}}^\varepsilon$, in the normal direction to an underlying hypersurface, $M$, of constant mean curvature $\lambda$ as in Subsection \ref{subsec: notation}; i.e.~$v_\varepsilon = \overline{\mathbb{H}}^\varepsilon \circ d^\pm_{\overline{M}}$, where here $d^\pm_{\overline{M}}$ is the Lipschitz \textit{signed distance function} to $\overline{M}$, taking positive values in $E$ and negative values in $N \setminus E$. The function $v_\varepsilon$ is then shown to act as an approximation of the hypersurface $M$ in the sense that, analogously to (\ref{eqn: min-max critical points converge to surface energy}), we have
\begin{equation}\label{eqn: 1d profile approx cmc energy}
    \mathcal{F}_{\varepsilon,\lambda}(v_\varepsilon) \rightarrow \mathcal{F}_\lambda(E) \text{ as } \varepsilon \rightarrow 0.
\end{equation}
In Subsection \ref{subsec: function minimisation}, for an isolated singularity, $p \in \mathrm{Sing}(M)$, $\varepsilon > 0$ sufficiently small and radius $\rho > 0$, we minimise $\mathcal{F}_{\varepsilon,\lambda}$ over a class of functions $\mathcal{A}_{\varepsilon,\rho}(p)$. The set $\mathcal{A}_{\varepsilon,\rho}(p)$ is, roughly speaking, all $W^{1,2}(N)$ functions agreeing with $v_\varepsilon$ outside of the ball of radius $\rho$ centred at $p$. The minimiser of this problem is thus a function, $g_\varepsilon \in W^{1,2}(N)$, that agrees with $v_\varepsilon$ outside of the ball $B_\rho(p)$ and is such that
\begin{equation}\label{eqn: minimiser function}
    \mathcal{F}_{\varepsilon,\lambda}(g_\varepsilon) = \inf_{u \in \mathcal{A}_{\varepsilon,\rho}(p)} \mathcal{F}_{\varepsilon,\lambda}(u).
\end{equation}
Note that the notation used for $g_\varepsilon$ suppresses the dependence on $p \in \mathrm{Sing}(M)$ and $\rho > 0$ used in the construction; in each instance that the functions $g_\varepsilon$ are utilised the choice of isolated singularity and radius in question will be made explicit. We then produce, in Subsection \ref{subsec: recovery functions}, a sequence of “recovery functions", admissible in the minimisation problem that produced $g_\varepsilon$ above, for any local $\mathcal{F}_\lambda$-minimiser. Precisely, in the vein of \cite{kohn_sternberg_1989}, for each local $\mathcal{F}_\lambda$-minimiser, $F \in \mathcal{C}(N)$, agreeing with $E$ outside of $B_{\frac{\rho}{2}}(p)$, we show that there exists a sequence of functions, $f_\varepsilon \in \mathcal{A}_{\varepsilon,\rho}(p)$, for all $\varepsilon > 0$ sufficiently small such that
\begin{equation}\label{eqn: recovery functions approx min}
    \mathcal{F}_{\varepsilon,\lambda}(f_\varepsilon) \rightarrow \mathcal{F}_\lambda(F) \text{ as } \varepsilon \rightarrow 0.
\end{equation}
As $f_\varepsilon \in \mathcal{A}_{\varepsilon,\rho}(p)$ we conclude that by (\ref{eqn: minimiser function}) we have $\mathcal{F}_{\varepsilon,\lambda}(g_\varepsilon) \leq \mathcal{F}_{\varepsilon,\lambda}(f_\varepsilon)$. In particular, by (\ref{eqn: 1d profile approx cmc energy}), (\ref{eqn: minimiser function}) and (\ref{eqn: recovery functions approx min}), if it holds that
\begin{equation}\label{eqn: function to geometry property}
     \mathcal{F}_{\varepsilon,\lambda}(v_\varepsilon) \leq \mathcal{F}_{\varepsilon,\lambda}(g_\varepsilon) + \tau_\varepsilon \text{ for some sequence } \tau_\varepsilon \rightarrow 0 \text{ as } \varepsilon \rightarrow 0,
\end{equation}
then
\begin{equation*}
    \mathcal{F}_\lambda(E) \leq \mathcal{F}_\lambda(F),
\end{equation*}
so that $E$ is $\mathcal{F}_\lambda$-minimising in $B_{\frac{\rho}{2}}(p)$. In this manner we have related, via (\ref{eqn: function to geometry property}), the $\varepsilon \rightarrow 0$ energy behaviour of specific $W^{1,2}(N)$ functions, namely $v_\varepsilon$ and $g_\varepsilon$, defined from a hypersurface as produced by the Allen--Cahn min-max procedure, to the geometric behaviour of the underlying hypersurface; precisely, if (\ref{eqn: function to geometry property}) holds then $E$ is locally $\mathcal{F}_\lambda$-minimising. In order to prove Theorem \ref{thm: local minimisation in positive ricci} we will turn our attention, in Step 3 below, to establishing that (\ref{eqn: function to geometry property}) holds for all constant mean curvature hypersurfaces produced by the Allen--Cahn min-max procedure in manifolds with positive Ricci curvature.

\bigskip

In order to produce the “recovery functions" above we first establish a local smoothing procedure for Caccioppoli sets that are smooth in an annular region. This is done by another “cut-and-paste" argument using Sard's Theorem on the level sets of mollified indicator functions for the Caccioppoli set, full details of which can be found in Subsection \ref{subsec: local smoothing of caccioppoli sets}. 

\bigskip

We emphasise that this local smoothing procedure we exhibit for Caccioppoli sets is not, in and of itself, sufficient to establish Theorem \ref{thm: generic regularity in dimension 8 in positive ricci}. The reason for this is that the local smoothing we produce records no information about the mean curvature near the isolated singularity, and this lack of information on the mean curvature makes it difficult to perturb the metric. The foliation of \cite{Kostas23} however ensures there exists a sequence of smooth hypersurfaces of constant mean curvature $\lambda$ converging to the singular hypersurface, allowing for the perturbation to be constructed in the manner as described in Step 1.

\bigskip

3.~\textbf{Paths of functions:} Similarly to the strategy employed in previous works on hypersurfaces produced by the Allen--Cahn min-max procedure, for example in \cite{bellettini}, \cite{bellettini-wickramasekera}, and \cite{Bellettini-Workman}, the proof of Theorem \ref{thm: local minimisation in positive ricci} is achieved by exhibiting a suitable continuous path in $W^{1,2}(N)$. 

\bigskip

Under the assumption that a hypersurface produced by the min-max procedure violates a desired property, one basic idea is to exploit its min-max characterisation as follows. If a path admissible in the min-max procedure may be produced, with energy along this path bounded above by a constant strictly less than the min-max value, then one contradicts the assumption that such a hypersurface arose from the min-max. Thus, the desired property must hold for all hypersurfaces produced by the min-max procedure.

\bigskip

We first emphasise that the paths we construct in $W^{1,2}(N)$ reflect the underlying geometry imposed by the 
assumption of positive Ricci curvature. We denote the super-level sets and level sets of the signed distance function, for each $s \in \mathbb{R}$, by
\begin{equation*}
L(s) = \{ x \in N \: | \: d^\pm_{\overline{M}}(x) > s \} \text{ and } \Gamma(s) = \{ x \in N \: | \: d^\pm_{\overline{M}}(x) = s \},
\end{equation*}
respectively; so that $E(0) = E$ and $\Gamma(0) = \overline{M}$. Formally computing we have, for almost every $s \in \mathbb{R}$, that 
\begin{equation}\label{eqn: formal geometry derivatives}
    \frac{d}{ds}\mathcal{H}^n(\Gamma(s)) = - \int_{\Gamma(s)} H(x,s) \, d\mathcal{H}^n(x) \text{ and } \frac{d}{ds}\mathrm{Vol}_g(L(s)) = \mathcal{H}^n(\Gamma(s)),
\end{equation}
where $H(x,s)$ denotes the mean curvature of the level set $\Gamma(s)$ at a point $x$. By denoting $m = \min_N \mathrm{Ric}_g > 0$, the assumption of positive Ricci curvature implies the following relation between the mean curvature of $M$ and the level sets $\Gamma(s)$:
\begin{equation}\label{eqn: mean curavture intro relation}
    \begin{cases}
    H(x,s) \geq \lambda + ms \text{ for } s > 0\\
    H(x,0) = \lambda\\
    H(x,s) \leq \lambda + ms \text{ for } s < 0
    \end{cases},
\end{equation}
see Subsection \ref{subsec: level sets of distance function}. Therefore, by (\ref{eqn: formal geometry derivatives}) and (\ref{eqn: mean curavture intro relation}), for each $t \in \mathbb{R} \setminus \{0\}$ we compute that
\begin{align*}
    \mathcal{F}_\lambda(E(t)) - \mathcal{F}_\lambda(E) &= \int_0^t \frac{d}{ds}\mathcal{H}^n(\Gamma(s)) - \lambda \frac{d}{ds}\mathrm{Vol}_g(L(s)) \, ds\\
    &= \int_0^t \int_{\Gamma(s)} \lambda - H(x,s) \, d\mathcal{H}^n(x) \, ds < 0,
\end{align*}
and so from the assumption of positive Ricci curvature we conclude that for each $t \in \mathbb{R} \setminus \{0\}$ we have
\begin{equation}\label{eqn: geometry term}
    \mathcal{F}_\lambda(E(t)) < \mathcal{F}_\lambda(E).
\end{equation}
We then consider the continuous path of \textit{sliding functions}, $v^t_\varepsilon \in W^{1,2}(N)$, produced by sliding the zero level set of $v_\varepsilon$ from $\overline{M}$ to $\Gamma(t)$ for each $t \in \mathbb{R}$; that is
\begin{equation*}
    v^t_\varepsilon = \overline{\mathbb{H}}^\varepsilon \circ (d^\pm_{\overline{M}} - t),
\end{equation*}
so that for each $t \in \mathbb{R}$ we have $\{ v^t_\varepsilon = 0 \} = \Gamma(t)$, $v^0_\varepsilon = v_\varepsilon$, $v^{2d(N)}_\varepsilon = -1$ and $v^{-2d(N)} = +1$ (whenever $\varepsilon > 0$ is sufficiently small). The geometric relation (\ref{eqn: geometry term}), induced by the positive Ricci curvature assumption, then translates to the level of functions and allows us to compute, in Lemma \ref{lem: sliding path upper energy bounds}, that the $v^t_\varepsilon$ have the following property:
\begin{equation*}
    \mathcal{F}_{\varepsilon,\lambda}(v^t_\varepsilon) \leq  \mathcal{F}_\lambda(v_\varepsilon) + E(\varepsilon) \text{ where } E(\varepsilon) \rightarrow 0 \text{ as } \varepsilon \rightarrow 0.
\end{equation*}
As the nodal sets of the $v^t_\varepsilon$ are the level sets of the signed distance function to $\overline{M}$, the functions $v^t_\varepsilon$ may be seen as a path of functions analogous to a sweep-out of $N$ by the level sets of the signed distance function to $\overline{M}$. 

\bigskip

Thus, the path provided by the sliding functions, along with the energy reducing paths from $-1$ and $+1$ provided by negative gradient flow of the energy to $a_\varepsilon$ and $b_\varepsilon$ respectively, provides a “recovery path" for the value $\mathcal{F}_{\varepsilon,\lambda}(E)$; this path connects $a_\varepsilon$ to $b_\varepsilon$, passing through $v_\varepsilon$, with the maximum value of the energy along this path approximately $\mathcal{F}_\lambda(E)$ (by virtue of (\ref{eqn: 1d profile approx cmc energy})). Approximate upper energy bounds along this path are depicted by the thick dashed lines in Figure \ref{fig: path} below. In this manner, as mentioned in Subsection \ref{subsec: Allen--Cahn minmax}, such a path establishes that the Allen--Cahn min-max procedure in positive Ricci curvature produces no minimal piece.
\begin{figure}[H]
\centering
\captionsetup{justification=justified,margin=1cm}
\begin{tikzpicture}[x=0.75pt,y=0.75pt,yscale=-1,xscale=1]

\draw    (121.67,269) -- (121.67,32.67) ;
\draw [shift={(121.67,29.67)}, rotate = 90] [fill={rgb, 255:red, 0; green, 0; blue, 0 }  ][line width=0.08]  [draw opacity=0] (8.93,-4.29) -- (0,0) -- (8.93,4.29) -- cycle    ;
\draw    (7.67,239.67) -- (588,239.34) ;
\draw [shift={(591,239.33)}, rotate = 179.97] [fill={rgb, 255:red, 0; green, 0; blue, 0 }  ][line width=0.08]  [draw opacity=0] (8.93,-4.29) -- (0,0) -- (8.93,4.29) -- cycle    ;
\draw  [dash pattern={on 0.84pt off 2.51pt}]  (122,59.67) -- (579.67,59) ;
\draw    (500.67,188.33) -- (559.67,239) ;
\draw    (202.67,187.67) -- (141,239) ;
\draw  [dash pattern={on 0.84pt off 2.51pt}]  (122.33,138.4) -- (580.33,136.67) ;
\draw  [dash pattern={on 4.5pt off 4.5pt}]  (350.83,59.33) -- (500.67,188.33) ;
\draw  [dash pattern={on 4.5pt off 4.5pt}]  (350.83,59.33) -- (202.67,187.67) ;
\draw  [dash pattern={on 0.84pt off 2.51pt}]  (122,187.33) -- (201.67,188.33) ;
\draw    (257.33,188.33) .. controls (287.67,59.67) and (286.33,215.67) .. (309,219.67) .. controls (331.67,223.67) and (328.33,108.33) .. (351.23,102.67) ;
\draw    (351.23,102.67) .. controls (372.33,109) and (371,225.67) .. (392.33,219.67) .. controls (413.67,213.67) and (412.33,58.33) .. (445,188.33) ;
\draw [line width=0.75]  [dash pattern={on 0.84pt off 2.51pt}]  (351.23,102.67) -- (350.33,239.67) ;
\draw    (201.67,188.33) -- (257.33,188.33) ;
\draw  [dash pattern={on 0.84pt off 2.51pt}]  (122.4,103) -- (580.07,102.33) ;
\draw  [dash pattern={on 0.84pt off 2.51pt}]  (202.67,187.67) -- (203,237.67) ;
\draw  [dash pattern={on 0.84pt off 2.51pt}]  (257.33,188.33) -- (257.67,238.33) ;
\draw  [dash pattern={on 0.84pt off 2.51pt}]  (122,220.33) -- (579.67,219.67) ;
\draw  [dash pattern={on 0.84pt off 2.51pt}]  (500.67,188.33) -- (500.33,239) ;
\draw  [dash pattern={on 0.84pt off 2.51pt}]  (445,188.33) -- (445,239.67) ;
\draw  [dash pattern={on 0.84pt off 2.51pt}]  (391.33,219.67) -- (390.67,239) ;
\draw  [dash pattern={on 0.84pt off 2.51pt}]  (310,219.67) -- (310,229) -- (310,241) ;
\draw    (445,188.33) -- (500.67,188.33) ;
\draw  [dash pattern={on 0.84pt off 2.51pt}]  (497.67,188.33) -- (581.75,188.5) ;
\draw  [dash pattern={on 0.84pt off 2.51pt}]  (257.33,188.33) -- (445,188.33) ;

\draw (61.67,51.07) node [anchor=north west][inner sep=0.75pt]    {$\mathcal{F}_{\varepsilon ,\lambda }( v_{\varepsilon })$};
\draw (136,253.07) node [anchor=north west][inner sep=0.75pt]    {$a_{\varepsilon }$};
\draw (554,251.73) node [anchor=north west][inner sep=0.75pt]    {$b_{\varepsilon }$};
\draw (34.67,130.4) node [anchor=north west][inner sep=0.75pt]    {$\mathcal{F}_{\varepsilon ,\lambda }( v_{\varepsilon }) -\frac{\eta }{2}$};
\draw (192.67,247.07) node [anchor=north west][inner sep=0.75pt]    {$v_{\varepsilon }^{t_{0}}$};
\draw (35,175.07) node [anchor=north west][inner sep=0.75pt]    {$\mathcal{F}_{\varepsilon ,\lambda }( v_{\varepsilon }) -\eta $};
\draw (343.67,254.13) node [anchor=north west][inner sep=0.75pt]    {$g_{\varepsilon }$};
\draw (487.33,248.4) node [anchor=north west][inner sep=0.75pt]    {$v_{\varepsilon }^{-t_{0}}$};
\draw (244.33,246.07) node [anchor=north west][inner sep=0.75pt]    {$v_{\varepsilon }^{t_{0} ,0}$};
\draw (432.67,247.73) node [anchor=north west][inner sep=0.75pt]    {$v_{\varepsilon }^{-t_{0} ,0}$};
\draw (35.4,91.4) node [anchor=north west][inner sep=0.75pt]    {$\mathcal{F}_{\varepsilon ,\lambda }( v_{\varepsilon }) -\tau $};
\draw (9,211.07) node [anchor=north west][inner sep=0.75pt]    {$\mathcal{F}_{\varepsilon ,\lambda }( v_{\varepsilon }) -\eta -\tau $};
\draw (301.33,245.9) node [anchor=north west][inner sep=0.75pt]    {$g_{\varepsilon }^{t_{0} ,2}$};
\draw (379.67,246.73) node [anchor=north west][inner sep=0.75pt]    {$g_{\varepsilon }^{-t_{0} ,2}$};

\end{tikzpicture}
\caption{The solid curve depicts approximate (i.e.~up to the addition of a term that converges to zero as $\varepsilon \rightarrow 0$) upper energy bounds along the path taken from $a_\varepsilon$ to $b_\varepsilon$ constructed for the proof of Theorem \ref{thm: local minimisation in positive ricci} in Subsection \ref{subsec: proof of theorem 3}. The horizontal axis identifies some specific functions in the path and the vertical axis depicts the approximate upper bound on the energy of the path between each identified function. The thick dashed lines depict an approximate upper energy bound along the path of functions in Lemma \ref{lem: sliding path upper energy bounds}.}
\label{fig: path}
\end{figure}
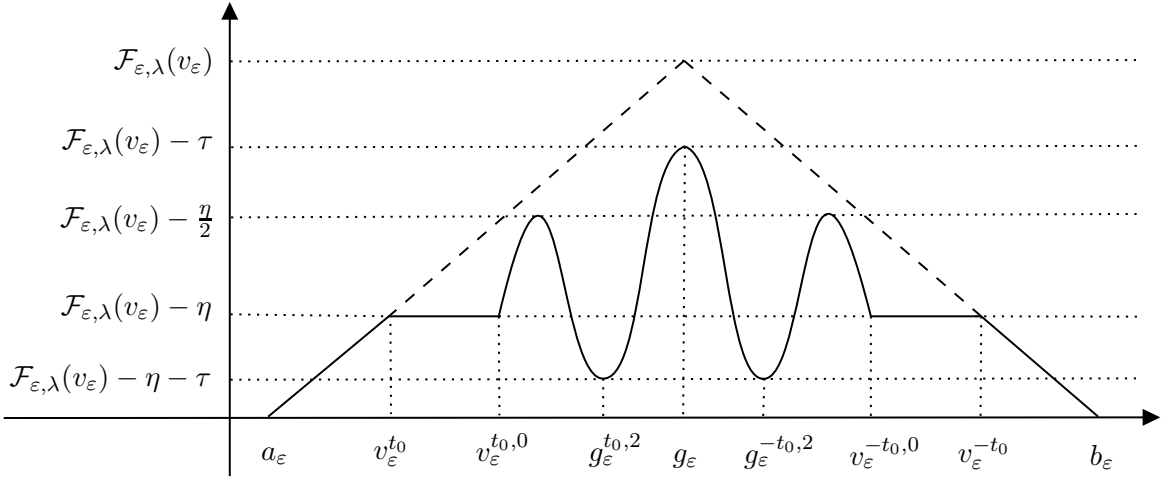

We now exhibit (for all $\varepsilon > 0$ sufficiently small) a continuous path in $W^{1,2}(N)$ between $a_\varepsilon$ and $b_\varepsilon$ which, under the assumption that $E$ is not $\mathcal{F}_\lambda$-minimising a small ball around an isolated singularity, contradicts the min-max characterisation of $E$ and proves Theorem \ref{thm: local minimisation in positive ricci}. We emphasise that this path is constructed with energy bounded above by a value strictly below $\mathcal{F}_{\varepsilon,\lambda}(v_\varepsilon)$ (independently of $\varepsilon$); approximate upper energy bounds are depicted by the solid curve in Figure \ref{fig: path} above. We thus conclude, by the arguments in Step 2, that the Allen--Cahn min-max procedure in positive Ricci curvature produces a hypersurface which is locally $\mathcal{F}_\lambda$-minimising.

\bigskip

The approximate upper energy bounds along the portions of the path we construct are computed explicitly in Section \ref{sec: paths}. However, in this section, we now sketch the various path constructions and motivate the upper energy bounds as a diffuse reflection of the underlying geometry.

\bigskip

First, for a given isolated singularity, $p \in \mathrm{Sing}(M)$, we are able to continuously deform the sets $E(t_0)$ (for a fixed $t_0 > 0$ sufficiently small) locally around $p$ so that the resulting deformation agrees with $E$ inside of a fixed ball $B_{r_0}(p)$ (for some $r_0 > 0$ determined only by the area of $M$). This is done by exploiting (\ref{eqn: geometry term}) in such a way that the $\mathcal{F}_\lambda$-energy of the deformations remain a fixed amount below $\mathcal{F}_\lambda(E)$. At the diffuse level this is replicated by placing $\overline{\mathbb{H}}^\varepsilon$ in the normal direction to the deformations as in the construction of $v_\varepsilon$, 
producing a $W^{1,2}(N)$ continuous path of functions with controlled energy; depicted in Figure \ref{fig: shifting functions I} below.
\bigskip
\begin{figure}[H]
\centering
\captionsetup{justification=justified,margin=1cm}
\begin{tikzpicture}[x=0.75pt,y=0.75pt,yscale=-1,xscale=1]

\draw  [dash pattern={on 0.84pt off 2.51pt}]  (210,205.64) -- (480.4,205.88) ;
\draw    (209.5,134.64) -- (240,135) ;
\draw    (449.6,134.6) -- (480.1,134.5) ;
\draw    (299.95,205.76) -- (390.45,205.76) ;
\draw    (299.95,182.76) -- (390.45,182.76) ;
\draw    (299.95,157.76) -- (390.45,157.76) ;
\draw  [dash pattern={on 4.5pt off 4.5pt}]  (240,135) -- (449.6,134.6) ;
\draw    (240,135) .. controls (283.75,135.25) and (259.25,205.75) .. (299.95,205.76) ;
\draw    (240,135) .. controls (283.75,135.25) and (259.25,182.75) .. (299.95,182.76) ;
\draw    (240,135) .. controls (283.75,135.25) and (259.25,157.75) .. (299.95,157.76) ;
\draw    (390.45,205.76) .. controls (430.75,205.75) and (403.75,135.25) .. (449.6,134.6) ;
\draw    (390.45,157.76) .. controls (430.75,157.75) and (403.75,135.25) .. (449.6,134.6) ;
\draw    (390.45,182.76) .. controls (430.75,182.75) and (403.75,135.25) .. (449.6,134.6) ;
\draw  [draw opacity=0][dash pattern={on 4.5pt off 4.5pt}] (316.16,205.45) .. controls (316.16,189.63) and (328.98,176.81) .. (344.8,176.81) .. controls (360.62,176.81) and (373.44,189.63) .. (373.44,205.45) .. controls (373.44,221.27) and (360.62,234.1) .. (344.8,234.1) .. controls (328.98,234.1) and (316.16,221.27) .. (316.16,205.45) -- cycle ;
\draw    (344.8,135) -- (345.18,202.76) ;
\draw [shift={(345.2,205.76)}, rotate = 269.67] [fill={rgb, 255:red, 0; green, 0; blue, 0 }  ][line width=0.08]  [draw opacity=0] (8.93,-4.29) -- (0,0) -- (8.93,4.29) -- cycle    ;
\end{tikzpicture}
\hspace{1mm}
\begin{tikzpicture}[x=0.75pt,y=0.75pt,yscale=-1,xscale=1]

\draw  [dash pattern={on 0.84pt off 2.51pt}]  (230,225.64) -- (319.95,225.76) ;
\draw    (229.5,154.64) -- (260,155) ;
\draw    (469.6,154.6) -- (500.1,154.5) ;
\draw    (319.55,225.45) -- (336.16,225.45) ;
\draw  [dash pattern={on 4.5pt off 4.5pt}]  (260,155) -- (469.6,154.6) ;
\draw    (260,155) .. controls (303.75,155.25) and (279.25,225.75) .. (319.95,225.76) ;
\draw    (410.45,225.76) .. controls (450.75,225.75) and (423.75,155.25) .. (469.6,154.6) ;
\draw  [dash pattern={on 4.5pt off 4.5pt}] (336.16,225.45) .. controls (336.16,209.63) and (348.98,196.81) .. (364.8,196.81) .. controls (380.62,196.81) and (393.44,209.63) .. (393.44,225.45) .. controls (393.44,241.27) and (380.62,254.1) .. (364.8,254.1) .. controls (348.98,254.1) and (336.16,241.27) .. (336.16,225.45) -- cycle ;
\draw  [dash pattern={on 0.84pt off 2.51pt}]  (410.45,225.76) -- (500.4,225.88) ;
\draw    (392.84,225.76) -- (410.45,225.76) ;

\draw (354.5,218.15) node [anchor=north west][inner sep=0.75pt]  [font=\huge]  {$v_{\varepsilon }$};
\end{tikzpicture}
\caption{In both graphics above the lower thin dashed horizontal line depicts $\overline{M}$, the zero level set of the function $v_\varepsilon$, and the upper thick dashed horizontal line depicts the zero level set, $\Gamma(t_0)$, of $v^{t_0}_\varepsilon = v^{t_0,1}_\varepsilon$ which is deformed in the construction of the shifted functions. In the left-hand graphic the solid lines depict various zero level sets of the $v^{t_0,s}_\varepsilon$ as we vary $s$ from $0$ to $1$. In the right-hand graphic the solid line depicts the zero level set of $v^{t_0,0}_\varepsilon$ and the thick dashed circle depicts the boundary of the ball, $B_{r_0}(p)$, in which $v^{t_0,0}_\varepsilon = v_\varepsilon$.}
\label{fig: shifting functions I}
\end{figure}
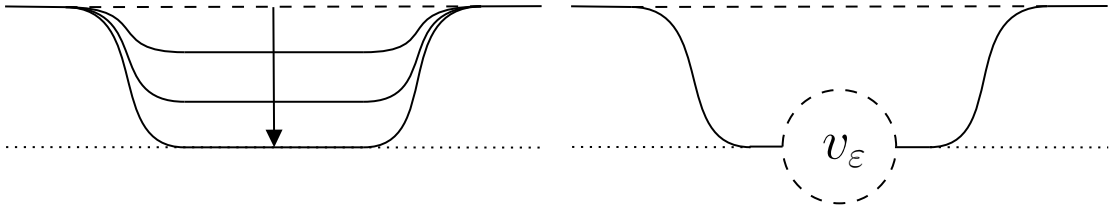
Specifically, we produce a continuous path of \textit{shifted functions}, $s \in [0,1] \rightarrow v^{t_0,s}_\varepsilon \in W^{1,2}(N)$, which satisfy $v^{t_0,1}_\varepsilon = v^{t_0}_\varepsilon$ on $N$ and $v^{t_0,0}_\varepsilon = v_\varepsilon$ inside $B_{r_0}(p)$. Furthermore, there exists an $\eta > 0$ such that for each $s \in [0,1]$ we have the following upper energy bound
    \begin{equation*}
        \mathcal{F}_{\varepsilon,\lambda}(v^{t_0,s}_\varepsilon) \leq \mathcal{F}_{\varepsilon,\lambda}(v_\varepsilon) - \eta,
    \end{equation*}
as depicted in Figure \ref{fig: path}. In this manner we have exhibited a continuous path in $W^{1,2}(N)$ from $v^{t_0}_\varepsilon$ to a function $v^{t_0,0}_\varepsilon$, equal to $v_\varepsilon$ in a fixed ball $B_{r_0}(p)$, with the energy along this path a fixed amount below the min-max value.

\bigskip

Next we construct a continuous path of functions, from $v^{t_0,0}_\varepsilon$ to the local energy minimiser $g_\varepsilon \in \mathcal{A}_{\varepsilon, \frac{R}{2}}(p)$ (recall Step 2) for a fixed $R \in (0, r_0)$. This is done in such a way that we only alter the functions inside $B_R(p)$; thus, in the following description we will only consider functions in the ball $B_R(p)$. The radius $R > 0$ here is chosen sufficiently small based on the energy drop, $\eta > 0$, above achieved outside of $B_{r_0}(p)$, ensuring that the energy along the constructed path will remain a fixed amount below the min-max value.

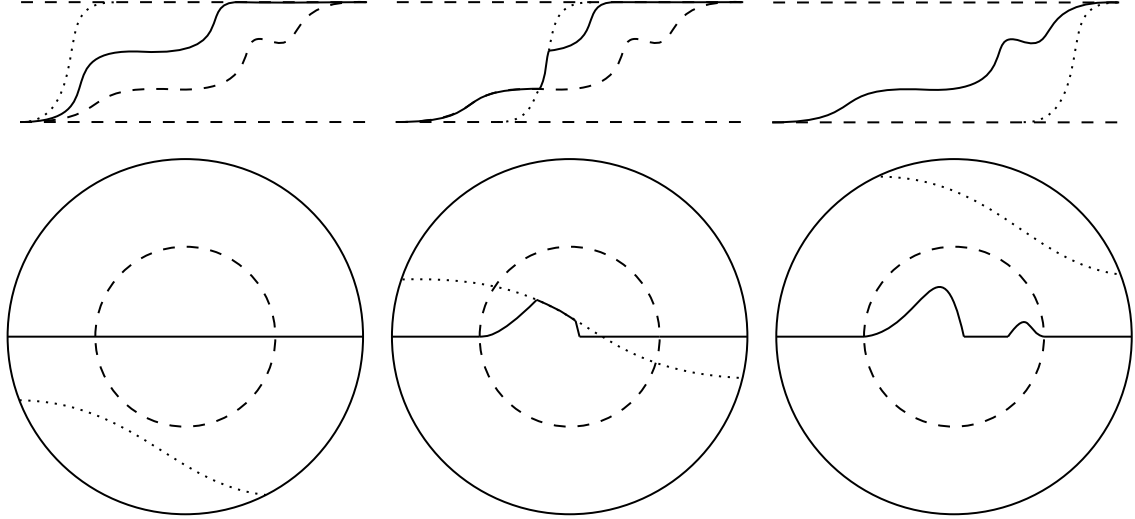
\begin{figure}[H]
    \centering
\captionsetup{justification=justified,margin=1cm}
\begin{tikzpicture}[x=0.75pt,y=0.75pt,yscale=-1,xscale=1]

\draw  [dash pattern={on 4.5pt off 4.5pt}]  (202,151) -- (377,151) ;
\draw  [dash pattern={on 4.5pt off 4.5pt}]  (202.5,90.5) -- (377,90.5) ;
\draw  [dash pattern={on 4.5pt off 4.5pt}]  (202,151) .. controls (252,151) and (228,132) .. (279,134.5) .. controls (330,137) and (307,103) .. (327.5,110) .. controls (348,117) and (332.5,88.5) .. (377,90.5) ;
\draw    (202,151) .. controls (252,151) and (212.5,113) .. (263.5,115.5) .. controls (314.5,118) and (290,90) .. (312.5,90.5) .. controls (335,91) and (359,90.5) .. (377,90.5) ;
\draw  [dash pattern={on 0.84pt off 2.51pt}]  (202,151) .. controls (220.2,151.4) and (224.68,133.82) .. (227.47,118.08) .. controls (230.27,102.35) and (230.6,90.2) .. (250.2,90.6) ;
\end{tikzpicture}
\hspace{1mm}
\begin{tikzpicture}[x=0.75pt,y=0.75pt,yscale=-1,xscale=1]

\draw  [dash pattern={on 4.5pt off 4.5pt}]  (202,151) -- (377,151) ;
\draw  [dash pattern={on 4.5pt off 4.5pt}]  (202.5,90.5) -- (377,90.5) ;
\draw  [dash pattern={on 4.5pt off 4.5pt}]  (202,151) .. controls (252,151) and (228,132) .. (279,134.5) .. controls (330,137) and (307,103) .. (327.5,110) .. controls (348,117) and (332.5,88.5) .. (377,90.5) ;
\draw    (202,151) .. controls (252,151) and (226.2,134.2) .. (274.6,134.2) ;
\draw    (279.2,115) .. controls (309.2,111.8) and (292.6,90.4) .. (312.7,90.5) .. controls (332.8,90.6) and (359.2,90.5) .. (377.2,90.5) ;
\draw    (274.6,134.2) .. controls (278.2,127.4) and (277.4,120.2) .. (279.2,115) ;
\draw  [dash pattern={on 0.84pt off 2.51pt}]  (253.2,151) .. controls (271.4,151.4) and (271.4,137) .. (274.6,134.2) ;
\draw  [dash pattern={on 0.84pt off 2.51pt}]  (279.2,115) .. controls (281.99,99.26) and (283.8,90.2) .. (301.4,90.6) ;
\end{tikzpicture}
\hspace{1mm}
\begin{tikzpicture}[x=0.75pt,y=0.75pt,yscale=-1,xscale=1]

\draw  [dash pattern={on 4.5pt off 4.5pt}]  (202,151) -- (377,151) ;
\draw  [dash pattern={on 4.5pt off 4.5pt}]  (202.5,90.5) -- (377,90.5) ;
\draw    (202,151) .. controls (252,151) and (228,132) .. (279,134.5) .. controls (330,137) and (307,103) .. (327.5,110) .. controls (348,117) and (332.5,88.5) .. (377,90.5) ;
\draw  [dash pattern={on 0.84pt off 2.51pt}]  (328.8,150.9) .. controls (347,151.3) and (351.48,133.72) .. (354.27,117.98) .. controls (357.07,102.25) and (357.4,90.1) .. (377,90.5) ;
\end{tikzpicture}
\vskip\floatsep 
\begin{tikzpicture}[scale=-0.0175, xscale=-1]
[x=0.75pt,y=0.75pt,yscale=-1,xscale=1]

\draw   (235,155) .. controls (235,80.44) and (295.44,20) .. (370,20) .. controls (444.56,20) and (505,80.44) .. (505,155) .. controls (505,229.56) and (444.56,290) .. (370,290) .. controls (295.44,290) and (235,229.56) .. (235,155) -- cycle ;
\draw    (235,155) -- (505,155) ;
\draw  [fill={rgb, 255:red, 0; green, 0; blue, 0 }  ,fill opacity=0 ][dash pattern={on 4.5pt off 4.5pt}] (301.63,155) .. controls (301.63,117.24) and (332.24,86.63) .. (370,86.63) .. controls (407.76,86.63) and (438.38,117.24) .. (438.38,155) .. controls (438.38,192.76) and (407.76,223.38) .. (370,223.38) .. controls (332.24,223.38) and (301.63,192.76) .. (301.63,155) -- cycle ;
\draw  [dash pattern={on 0.84pt off 2.51pt}]  (244,203.33) .. controls (332.33,204.67) and (362,268) .. (431.33,275.33) ;
\end{tikzpicture}
\hspace{1mm}
\begin{tikzpicture}[scale=-0.0175, xscale=-1][x=0.75pt,y=0.75pt,yscale=-1,xscale=1]

\draw   (235,155) .. controls (235,80.44) and (295.44,20) .. (370,20) .. controls (444.56,20) and (505,80.44) .. (505,155) .. controls (505,229.56) and (444.56,290) .. (370,290) .. controls (295.44,290) and (235,229.56) .. (235,155) -- cycle ;
\draw    (235,155) -- (301.63,155) ;
\draw  [fill={rgb, 255:red, 0; green, 0; blue, 0 }  ,fill opacity=0 ][dash pattern={on 4.5pt off 4.5pt}] (301.63,155) .. controls (301.63,117.24) and (332.24,86.63) .. (370,86.63) .. controls (407.76,86.63) and (438.38,117.24) .. (438.38,155) .. controls (438.38,192.76) and (407.76,223.38) .. (370,223.38) .. controls (332.24,223.38) and (301.63,192.76) .. (301.63,155) -- cycle ;
\draw    (438.38,155) -- (505,155) ;
\draw  [dash pattern={on 0.84pt off 2.51pt}]  (243,111.5) .. controls (392,107.5) and (370,180.5) .. (502.5,186.5) ;
\draw    (377.5,155) -- (438.38,155) ;
\draw    (301.63,155) .. controls (306,154.5) and (314.5,157.5) .. (345,127) ;
\draw    (374.5,143) -- (377.5,155) ;
\draw    (345,127) .. controls (367,136.5) and (369,139.5) .. (374.5,143) ;
\end{tikzpicture}
\hspace{1mm}
\begin{tikzpicture}[scale=-0.0175, xscale=-1][x=0.75pt,y=0.75pt,yscale=-1,xscale=1]

\draw   (235,155) .. controls (235,80.44) and (295.44,20) .. (370,20) .. controls (444.56,20) and (505,80.44) .. (505,155) .. controls (505,229.56) and (444.56,290) .. (370,290) .. controls (295.44,290) and (235,229.56) .. (235,155) -- cycle ;
\draw    (235,155) -- (301.63,155) ;
\draw  [fill={rgb, 255:red, 0; green, 0; blue, 0 }  ,fill opacity=0 ][dash pattern={on 4.5pt off 4.5pt}] (301.63,155) .. controls (301.63,117.24) and (332.24,86.63) .. (370,86.63) .. controls (407.76,86.63) and (438.38,117.24) .. (438.38,155) .. controls (438.38,192.76) and (407.76,223.38) .. (370,223.38) .. controls (332.24,223.38) and (301.63,192.76) .. (301.63,155) -- cycle ;
\draw    (438.38,155) -- (505,155) ;
\draw    (377.5,155) -- (411,155) ;
\draw    (411,155) .. controls (428,131) and (425.5,153) .. (438.38,155) ;
\draw    (301.63,155) .. controls (346,151) and (359,72) .. (377.5,155) ;
\draw  [dash pattern={on 0.84pt off 2.51pt}]  (314.33,33.33) .. controls (402.67,34.67) and (429,100.67) .. (498.33,108) ;
\end{tikzpicture}
    \caption{The graphics above depict various stages of the transition, in $W^{1,2}(N)$, between $v_\varepsilon$ and $\min\{v_\varepsilon,g_\varepsilon\}$ in $B_R(p)$. The first row depicts schematics of this transition and the second row depicts the geometry of the zero level sets of the local path functions; in this way, each column of the figure contains a schematic for the local path function with a corresponding depiction of the geometry of its zero level set below it. For the first row, in all three of the images the solid black line depicts the local path function in question (equal to $v_\varepsilon$ and $\min\{v_\varepsilon,g_\varepsilon\}$ in the left-hand and right-hand graphic respectively), the thick dashed curve depicts portions of $\min\{v_\varepsilon,g_\varepsilon\}$ that are not yet included in the path, the thin dashed curve depicts the diffuse sweep-out function in question, and the thick dashed upper and lower horizontal lines depict the functions $\pm 1$ respectively. Notice that the local path function depicted in the middle graphic includes portions of $v_\varepsilon$ that lie to the right of the diffuse sweep-out function, portions of $\min\{v_\varepsilon,g_\varepsilon\}$ that lie to the left of the diffuse sweep-out function and uses the diffuse sweep-out function itself to interpolate between $v_\varepsilon$ and $\min\{v_\varepsilon,g_\varepsilon\}$. The explicit construction of the local path functions replicating the behaviour of these schematics involves taking various maxima and minima of the functions in the diffuse sweep-out, $v_\varepsilon$ and $g_\varepsilon$; the diffuse sweep-out allows for this choice of maxima and minima (corresponding geometrically to a choice of zero level set) to be made continuously. For the second row, in all three graphics the outer solid circle depicts the boundary of $B_R(p)$ and the thick dashed inner circle depicts the boundary of $B_{\frac{R}{2}}(p)$. In the left-hand graphic the solid line depicts the zero level set of $v_\varepsilon$, which is $\overline{M}$, in the right-hand graphic the solid curve depicts the zero level set of $\min \{g_\varepsilon,v_\varepsilon\}$, and in the middle graphic the solid curve depicts the zero level set of a local path function in the transition between $v_\varepsilon$ and $\min\{v_\varepsilon,g_\varepsilon\}$. Each the thin dashed curves in the three graphics depict a given hypersurface in the “planar" sweep-out of $B_R(p)$, each separating $B_R(p)$ into two open sets (one above and one below it). The zero level set of the function in the local path associated to this hypersurface is chosen to be the portions of the zero level set of $\min\{v_\varepsilon,g_\varepsilon\}$ that are beneath the hypersurface, the portions of $\overline{M}$ that are above the hypersurface, and when the the hypersurface lies between the nodal sets of $v_\varepsilon$ and $\min\{v_\varepsilon,g_\varepsilon\}$, the hypersurface itself is chosen as the zero level set. The same ideas described above are used for the construction of the portion of the path from $\min\{v_\varepsilon,g_\varepsilon\}$ to $g_\varepsilon$ in $B_R(p)$.}
    \label{fig: local path}
\end{figure}
In order to construct this local path from $v_\varepsilon$ to $g_\varepsilon$ inside of $B_R(p)$ we utilise a sweep-out of the ball by images of Euclidean planes via a geodesic normal coordinate chart; depicted by the thin dashed curves in the second row of Figure \ref{fig: local path} above. The hypersurfaces in this sweep-out are used to continuously transition from our hypersurface and the local $\mathcal{F}_\lambda$-minimiser at the diffuse level; precisely, the planes facilitate the construction of a path between the diffuse representatives $v_\varepsilon$ and $g_\varepsilon$.

\begin{rem}
In the setting of the Almgren--Pitts min-max homotopy sweep-outs of $N$ by cycles are considered, as opposed to continuous paths in $W^{1,2}(N)$ in the Allen--Cahn min-max. However, there does not necessarily exist a homotopy of cycles between our hypersurface and any local $\mathcal{F}_\lambda$-minimiser. We overcome this for the Allen--Cahn min-max by directly exploiting the topology of $W^{1,2}(N)$, showing that the local path may be seen as a diffuse analogue of \cite[Lemma 1.12]{C-L-S}, and illustrating why we guarantee local $\mathcal{F}_\lambda$-minimisation (see Remark \ref{rem: CLS comparison}).
\end{rem}

In the same manner as in the construction of $v_\varepsilon$, by placing $\overline{\mathbb{H}}^\varepsilon$ in the normal direction to the hypersurfaces in this “planar" sweep-out we construct a sweep-out at the diffuse level that is continuous in $W^{1,2}(N)$; the diffuse sweep-out thus acts as an approximation for the underlying “planar" sweep-out. The diffuse sweep-out of $B_R(p)$ is utilised twice, first for the construction of a path from $v_\varepsilon$ to $\min\{g_\varepsilon,v_\varepsilon\}$, and second for the construction of a path from $\min\{g_\varepsilon,v_\varepsilon\}$ to $g_\varepsilon$. By taking a combination of maxima and minima of functions in the diffuse sweep-out, $v_\varepsilon$ and $g_\varepsilon$ (which ensure the resulting functions are in $W^{1,2}(N)$), we are able to produce a local path from $v_\varepsilon$ to $g_\varepsilon$; see Figure \ref{fig: local path} above for a description of this construction.

\bigskip

We note that $\mathcal{F}_{\varepsilon,\lambda}(g_\varepsilon) \leq \mathcal{F}_{\varepsilon,\lambda}(\min\{v_\varepsilon,g_\varepsilon\}) \leq \mathcal{F}_{\varepsilon,\lambda}(v_\varepsilon)$ (by local energy minimisation of $g_\varepsilon$) and that $R \in (0,r_0)$ is chosen based on $\eta > 0$ to ensure that the total energy contribution of the diffuse sweep-out functions in the ball is at most $\frac{\eta}{2}$. As a consequence of these two facts, the energy in $B_R(p)$ of any local path function can be estimated to be at most the energy of $v_\varepsilon$ in $B_R(p)$ plus $\frac{\eta}{2}$; from there one obtains the energy estimate in the whole of $N$. Specifically, we produce a continuous path of \textit{local functions}, $s \in [-2,2] \rightarrow g^{t_0,s}_\varepsilon \in W^{1,2}(N)$, which satisfy $g^{t_0,-2}_\varepsilon = v^{t_0,0}_\varepsilon$ on $N$, $g^{t_0,2} = g_\varepsilon$ in $B_{R}(p)$ and are all equal to $v^{t_0,0}_\varepsilon$ outside of $B_R(p)$. Furthermore, for each $s \in  [-2,2]$ we have the following upper energy bound
\begin{equation*}
    \mathcal{F}_{\varepsilon,\lambda}(g^{t_0,s}_\varepsilon) \leq \mathcal{F}_{\varepsilon,\lambda}(v_\varepsilon) - \frac{\eta}{2},
\end{equation*}
as depicted in Figure \ref{fig: path}. In this manner we have exhibited a continuous path in $W^{1,2}(N)$ from $v^{t_0,0}_\varepsilon$ to a function $g^{t_0,2}_\varepsilon$, changing $v^{t_0,0}_\varepsilon$ only inside of $B_R(p)$ (from $v_\varepsilon$ to the local energy minimiser $g_\varepsilon$).

\bigskip

In order to establish that $E$ is locally $\mathcal{F}_\lambda$-minimising, we now argue by contradiction and assume that there exists a $\tau > 0$ such that
\begin{equation}\label{eqn: contradiction assumption}
\mathcal{F}_{\varepsilon,\lambda}(v_\varepsilon) \geq \mathcal{F}_{\varepsilon,\lambda}(g_\varepsilon) + \tau \text{ for all } \varepsilon > 0 \text{ sufficiently small.}
\end{equation}
Note that by the results discussed in Step 2, contradicting (\ref{eqn: contradiction assumption}) will establish that $M$ is $\mathcal{F}_\lambda$-minimising in $B_{\frac{R}{4}}(p)$ for $R > 0$ as chosen above (as (\ref{eqn: function to geometry property}) must hold). Using the contradiction assumption, in addition to keeping the energy of the local path functions a fixed amount, $\frac{\eta}{2}$, below the min-max value, as $g^{t_0,2}_\varepsilon = g_\varepsilon$ in $B_R(p)$, we thus also have that
\begin{equation*}
    \mathcal{F}_{\varepsilon,\lambda}(g^{t_0,2}_\varepsilon) \leq \mathcal{F}_{\varepsilon,\lambda}(v_\varepsilon) - \eta - \tau,
\end{equation*}
as depicted in Figure \ref{fig: path}. We now directly exploit this extra energy drop afforded by the contradiction assumption to construct the next portion of the path.

\bigskip

To this end, we deform the rest of the set $E(t_0)$ entirely onto $E$ outside of $B_{r_0}(p)$. This is done in such a way that the deformations fix the inside of $B_{r_0}(p)$, thus preserving a drop in $\mathcal{F}_\lambda$-energy under the assumption that $E$ is not locally $\mathcal{F}_\lambda$-minimising. At the diffuse level this is replicated by placing $\overline{\mathbb{H}}^\varepsilon$ in the normal direction to these deformations, keeping the functions equal to the local energy minimiser $g_\varepsilon$ (which yields an energy drop by (\ref{eqn: contradiction assumption})) inside of $B_{r_0}(p)$ and producing a $W^{1,2}(N)$ continuous path of functions with controlled energy; this is depicted in Figure \ref{fig: shifting functions II}.

\bigskip
\begin{figure}[H]
\centering
\captionsetup{justification=justified,margin=1cm}
\begin{tikzpicture}[x=0.75pt,y=0.75pt,yscale=-1,xscale=1]

\draw  [dash pattern={on 0.84pt off 2.51pt}]  (230,225.64) -- (319.95,225.76) ;
\draw    (229.5,154.64) -- (260,155) ;
\draw    (469.6,154.6) -- (500.1,154.5) ;
\draw    (319.55,225.45) -- (336.16,225.45) ;
\draw  [dash pattern={on 4.5pt off 4.5pt}]  (260,155) -- (469.6,154.6) ;
\draw    (260,155) .. controls (303.75,155.25) and (279.25,225.75) .. (319.95,225.76) ;
\draw    (410.45,225.76) .. controls (450.75,225.75) and (423.75,155.25) .. (469.6,154.6) ;
\draw  [dash pattern={on 4.5pt off 4.5pt}] (336.16,225.45) .. controls (336.16,209.63) and (348.98,196.81) .. (364.8,196.81) .. controls (380.62,196.81) and (393.44,209.63) .. (393.44,225.45) .. controls (393.44,241.27) and (380.62,254.1) .. (364.8,254.1) .. controls (348.98,254.1) and (336.16,241.27) .. (336.16,225.45) -- cycle ;
\draw  [dash pattern={on 0.84pt off 2.51pt}]  (410.45,225.76) -- (500.4,225.88) ;
\draw    (392.84,225.76) -- (410.45,225.76) ;
\draw    (260.5,185) .. controls (304.25,185.25) and (278.85,225.45) .. (319.55,225.45) ;
\draw    (261.5,209) .. controls (305.25,209.25) and (278.85,225.45) .. (319.55,225.45) ;
\draw    (230,184.64) -- (260.5,185) ;
\draw    (231,208.64) -- (261.5,209) ;
\draw    (410.45,225.76) .. controls (450.75,225.75) and (423.5,184.5) .. (470,186) ;
\draw    (410.45,225.76) .. controls (450.75,225.75) and (423.5,209) .. (470,210.5) ;
\draw    (470,186) -- (500.5,186.36) ;
\draw    (470,210.5) -- (500.5,210.86) ;
\draw    (250.8,155) -- (251.18,222.76) ;
\draw [shift={(251.2,225.76)}, rotate = 269.67] [fill={rgb, 255:red, 0; green, 0; blue, 0 }  ][line width=0.08]  [draw opacity=0] (8.93,-4.29) -- (0,0) -- (8.93,4.29) -- cycle    ;
\draw    (479.3,155) -- (479.68,222.76) ;
\draw [shift={(479.7,225.76)}, rotate = 269.67] [fill={rgb, 255:red, 0; green, 0; blue, 0 }  ][line width=0.08]  [draw opacity=0] (8.93,-4.29) -- (0,0) -- (8.93,4.29) -- cycle    ;

\draw (354.5,218.15) node [anchor=north west][inner sep=0.75pt]  [font=\huge]  {$g_{\varepsilon }$};

\end{tikzpicture}
\hspace{1mm}
\begin{tikzpicture}[x=0.75pt,y=0.75pt,yscale=-1,xscale=1]

\draw    (229.5,225) -- (336.16,225.45) ;
\draw  [dash pattern={on 4.5pt off 4.5pt}]  (230.5,155) -- (501,155) ;
\draw  [dash pattern={on 4.5pt off 4.5pt}] (336.16,225.45) .. controls (336.16,209.63) and (348.98,196.81) .. (364.8,196.81) .. controls (380.62,196.81) and (393.44,209.63) .. (393.44,225.45) .. controls (393.44,241.27) and (380.62,254.1) .. (364.8,254.1) .. controls (348.98,254.1) and (336.16,241.27) .. (336.16,225.45) -- cycle ;
\draw    (392.84,225.76) -- (499.5,226.21) ;

\draw (354.5,218.15) node [anchor=north west][inner sep=0.75pt]  [font=\huge]  {$g_{\varepsilon }$};

\end{tikzpicture}
\caption{In both graphics above the lower horizontal lines depict $\overline{M}$, the upper horizontal line depicts the zero level set, $\Gamma(t_0)$, of $v^{t_0}_\varepsilon = v^{t_0,1}_\varepsilon$ which is deformed in the construction of the shifted functions and the thick dashed circle depicts the boundary of the ball, $B_{r_0}(p)$, in which the shifted functions $g^{t,2}_\varepsilon = g_\varepsilon$. In the left-hand graphic the solid lines depict various zero level sets of the $g^{t,2}_\varepsilon$ as we vary $t$ from $t_0$ to $0$. In the right-hand graphic the solid line depicts the zero level set of $g^{0,2}_\varepsilon = g_\varepsilon$.}
\label{fig: shifting functions II}
\end{figure}
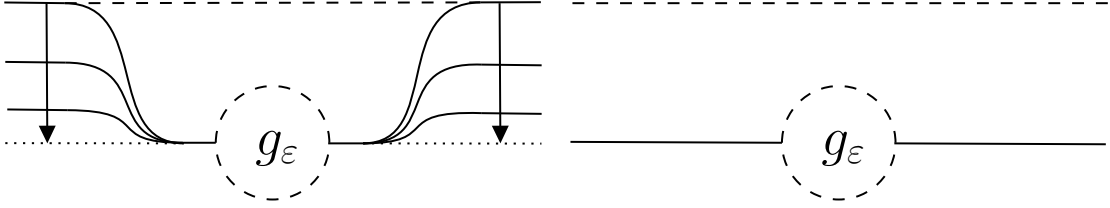
Specifically, we produce (under the assumption that (\ref{eqn: contradiction assumption}) holds) a continuous path of \textit{shifted functions}, $t \in [0,t_0] \rightarrow g^{t,2}_\varepsilon \in W^{1,2}(N)$, which are equal to the local function $g^{t_0,2}_\varepsilon$ when $t = t_0$ (justifying notation), equal to $g_\varepsilon$ when $t = 0$ and are all equal to $g_\varepsilon$ in $B_{r_0}(p)$. Furthermore, for each $t \in [0,t_0]$ have the following upper energy bound
\begin{equation*}
    \mathcal{F}_{\varepsilon,\lambda}(g^{t,2}_\varepsilon) \leq \mathcal{F}_{\varepsilon,\lambda}(v^{t,0}_\varepsilon) - \tau,
\end{equation*}
as depicted Figure \ref{fig: path}. In this manner we have exhibited a continuous path in $W^{1,2}(N)$ from $g^{t_0,2}_\varepsilon$ to the local energy minimiser $g_\varepsilon$, only changing $g^{t_0,2}_\varepsilon$ outside of $B_{r_0}(p)$.

\bigskip

To summarise all of the above, as the endpoints of each of the paths described agree with the start of the next, for all $\varepsilon > 0$ sufficiently small we have exhibited a continuous paths in $W^{1,2}(N)$ connecting $a_\varepsilon$ to the local energy minimiser $g_\varepsilon$. We also demonstrated that the energy along this path is bounded above by a value strictly below $\mathcal{F}_{\varepsilon,\lambda}(v_\varepsilon)$ (independently of $\varepsilon$), depicted in Figure \ref{fig: path}.

\bigskip

To complete the desired path from $a_\varepsilon$ to $b_\varepsilon$ it remains to construct the portion from $g_\varepsilon$ to $b_\varepsilon$. By considering $-t_0$ instead of $t_0$ in each of the paths sketched above, we ensure that through symmetric (with respect to the underlying hypersurfaces) deformations of the sets $E(-t_0)$, the relevant symmetric portions of the path may constructed with identical upper energy bounds; this portion of the path is depicted in Figure \ref{fig: path}, where the symmetry of the path with respect to $g_\varepsilon$ is made apparent.

\bigskip

To conclude the proof of Theorem \ref{thm: local minimisation in positive ricci} we concatenate the path from $a_\varepsilon$ to $g_\varepsilon$ and the path from $g_\varepsilon$ to $b_\varepsilon$, completing the desired continuous path in $W^{1,2}(N)$ between the two stable critical points $a_\varepsilon$ and $b_\varepsilon$. Under the assumption that (\ref{eqn: contradiction assumption}) holds, the upper energy bounds along this path (depicted by the solid curve in Figure \ref{fig: path}) and (\ref{eqn: 1d profile approx cmc energy}) ensure that for $\varepsilon > 0$ sufficiently small we have
\begin{equation*}
    \text{$\mathcal{F}_{\varepsilon,\lambda}$ energy along the path} \leq \mathcal{F}_\lambda(E) - \min\left\{\frac{\eta}{4},\frac{\tau}{2} \right\}.
\end{equation*}
As (\ref{eqn: min-max critical points converge to surface energy}) holds, as mentioned in Subsection \ref{subsec: Allen--Cahn minmax}, and as the above path is admissible in the Allen--Cahn min-max construction of $E$, we must contradict the assumption that (\ref{eqn: contradiction assumption}) holds. We thus conclude that any such $E$ as produced by the Allen--Cahn min-max procedure in Ricci positive curvature must be such that $E$ is locally $\mathcal{F}_\lambda$-minimising (as (\ref{eqn: function to geometry property}) holds), proving Theorem \ref{thm: local minimisation in positive ricci}. We then establish Theorems \ref{thm: generic regularity in dimension 8 in positive ricci} and \ref{thm: generic removability of isolated singularities with regular cones in positive ricci} by applying the results described in Step 1.

\subsection{Structure and remarks}

We now proceed as follows:

\begin{itemize}
    \item Section \ref{sec: cmc surgery} recalls the result of \cite{Kostas23} and uses it to perturb away isolated singularities of constant mean curvature hypersurfaces with regular tangent cones. 
    
    \item Section \ref{sec: signed distance and one-dimensional profile} analyses the signed distance function and introduces the one-dimensional profile, before showing that it approximates the underlying hypersurface in a suitable sense.

    \item Section \ref{sec: local energy properties} establishes a procedure to locally smooth Caccioppoli sets which are smooth in annular regions. This procedure is then used to relate local energy minimisation to the local geometric behaviour of the hypersurface.

    \item Section \ref{sec: paths} provides constructions of the various continuous paths in $W^{1,2}(N)$ along with calculations of upper bounds of the energy along these paths.

    \item Section \ref{sec: conclusion of the proof} ties together the results of the previous sections in order to prove the main results.
\end{itemize}

The following remarks should be kept in mind throughout:

\begin{rem}\label{rem: singular dimensions}
    As mentioned above, in \cite[Proposition B.1]{bellettini-wickramasekera-cmc-reg} it is shown that smooth constant mean curvature hypersurfaces are locally $\mathcal{F}_\lambda$-minimising. In proving Theorem \ref{thm: local minimisation in positive ricci} we aim to show that this also holds around isolated singularities for constant mean curvature hypersurfaces produced by the Allen--Cahn min-max in manifolds with positive Ricci curvature. We may thus restrict to the dimensions in which these objects may be singular and work under the assumptions that $N$ is of dimension $n + 1 \geq 8$. 
\end{rem}

\begin{rem}
    The assumption of positive Ricci curvature is used only to ensure that the upper energy bounds on the paths constructed in Section \ref{sec: paths} remain a fixed amount below the min-max value, $\mathcal{F}_\lambda(E)$. In particular, we note that the results of Sections \ref{sec: cmc surgery}, \ref{sec: signed distance and one-dimensional profile} and \ref{sec: local energy properties} as well as the paths constructed (but not their upper energy bounds) in Section \ref{sec: paths} make no use of the assumption of positive Ricci curvature.
\end{rem}

\begin{rem}
    We make the choice of a positive upper bound on $\varepsilon > 0$ finitely many times throughout the construction of the paths in the proof of Theorem \ref{thm: local minimisation in positive ricci}, ultimately constructing the paths for all $\varepsilon > 0$ smaller than a fixed positive constant. The specific choice of upper bound utilised in each instance may change, but we implicitly assume that a correct upper bound for which the desired property holds is used in each case. This remark will apply each time we choose $\varepsilon > 0$ sufficiently small.
\end{rem}

\subsection*{Acknowledgements}
The authors would like to thank the anonymous referee for constructive comments and suggestions.

\bigskip

CB was supported in part by the Engineering and Physical Sciences Research Council
[EP/S005641/1].

\bigskip

KMS was supported by the Engineering and Physical Sciences Research Council [EP/N509577/1], [EP/T517793/1]. He would like to thank Fritz Hiesmayr, Konstantinos Leskas and Myles Workman for many enlightening discussions. He is also grateful to Otis Chodosh for interesting email exchanges related to some aspects of this work, and to Marco Badran for comments on a preliminary version of the manuscript.

\section{Surgery procedure for isolated singularities with regular tangent cone}\label{sec: cmc surgery}

We show how the recent result of \cite{Kostas23} (see also \cite{BL24}) can be combined with a local perturbation of the metric to regularise a hypersurface of constant mean curvature around an isolated singularity with area-minimising regular tangent cone. We first collect some notation and definitions, phrased in notation in keeping with this paper, before stating the  main theorem from \cite{Kostas23} and using it to establish the surgery procedure.

\begin{rem}
    The initial version of this manuscript appeared prior to \cite{BL24}, and thus the notation that appears in this section aligns with that of \cite{Kostas23}. The results in \cite{BL24} are phrased solely in terms of Caccioppoli sets, as opposed to currents here and in \cite{Kostas23}. We therefore remark that one can very directly apply the main result of \cite{BL24} (rather than Theorem \ref{thm: Kostas foliation} below) in order to establish the surgery procedure, Proposition \ref{prop: CMC surgery} below, without making reference to the theory of currents (we also stress that the relevant functional in \cite{BL24} agrees with $\mathcal{F}_\lambda$ up to the addition of a constant).
\end{rem}

We reset notation for this section, letting $(N^{n+1},g)$ be a Riemannian manifold with no curvature assumption. Throughout this section $T = \partial [A]\llcorner B_1(p)$ will denote the multiplicity one integral current associated to a Caccioppoli set, $A \in \mathcal{C}(N)$, restricted to a ball $B_1(p)$ about a point $p \in N$, for which the following properties hold:
\begin{itemize}
    \item $\mathrm{Spt}(T)$ is connected.
    
    \item $\mathrm{Sing}(T) = \{p\}$ (so that $p$ is an isolated singularity of $T$).
    
    \item $T$ has a regular tangent cone at $p$.
\end{itemize}
For a given $\lambda \in \mathbb{R}$ we fix a choice of $0 < r_1 < r_2 < 1$ sufficiently small so that the following properties hold:
\begin{itemize}
    \item $\Gamma_0 = \partial(T \llcorner B_{r_1}(p))$ is a closed embedded, connected $(n-1)$-dimensional sub-manifold of $\partial B_{r_1}(p)$.

    \item $\mathrm{Spt}(T) \cap \partial B_{r_1}(p)$ is a transverse intersection.

    \item $\overline{A \cap B_{r_2}(p)}$ has connected complement in $B_{r_2}(p)$.
\end{itemize}
Let $\phi_j : \Gamma_0 \rightarrow \partial B_{r_1}(p)$ denote $C^2$ maps with
\begin{equation*}
    |\phi_j - i_{\Gamma_0}|_{C^2} \leq \frac{1}{j},
\end{equation*}
and for $\Gamma_j = (\phi_j)_* \Gamma_0$ we assume that $\Gamma_j \cap A^o \neq \emptyset$. Here we denote by $i_{\Gamma_0}$ the identity map on $\Gamma_0$, $(\phi_j)_* \Gamma_0$ the push-forward of $\Gamma_0$ by the map $\phi_j$ and $A^o$ the interior of $A$. We now provide a definition of $\mathcal{F}_\lambda$-minimisation in keeping with the notation used in \cite{Kostas23}:

\begin{definition}\label{def: one-sided minimiser}
For $\lambda \in \mathbb{R}$ we say that $T = (\partial [A])\llcorner B_1(p)$ is $\mathcal{F}_\lambda$-minimising in $B_1(p)$ if both of the following properties hold:
\begin{itemize}
    \item $A$ is a critical point of $\mathcal{F}_\lambda$ in $B_1(p)$.

    \item We have that 
    \begin{equation*}
        \mathcal{F}_\lambda(A) = \inf_{G \in \mathcal{C}(N)} \{ \mathcal{F}_\lambda(G) \: | \: G \setminus B_1(p) = A \setminus B_1(p) \};
    \end{equation*}
    i.e.~$A$ is $\mathcal{F}_\lambda$-minimising in $B_1(p)$.
\end{itemize}
\end{definition}

\begin{rem}
    Note that this definition of $\mathcal{F}_\lambda$-minimisation for currents directly implies the notion of minimisation introduced in \cite{Kostas23}.
\end{rem}

Using the above notation and definition, \cite{Kostas23} then proves the following foliation result, generalising the results of \cite{hardt-simon} to the case of constant mean curvature hypersurfaces.

\begin{thm}\label{thm: Kostas foliation}\cite{Kostas23}
Let $\lambda \in \mathbb{R}$ and $T$ as defined above be $\mathcal{F}_\lambda$-minimising in $B_1(p)$. Then, for every $j \geq 1$, there exist integral $n$-currents, $S_j$, that satisfy the following properties:
\begin{itemize}
    \item $\mathrm{Spt}(S_j) \subset \overline{B_{r_1}(p)}$.

    \item There exist sets of finite perimeter $B_j$, with $\overline{B_j} \subset \overline{A \cap B_{r_1}(p)}$, such that $S_j = \partial [B_j] \llcorner B_{r_1}(p)$.

    \item $\Gamma_j = \mathrm{Spt}(S_j) \cap \partial B_{r_1}(p)$.

    \item Each $S_j$ is a critical point of $\mathcal{F}_\lambda$ in $B_{r_1}(p)$.

    \item For the measures associated to the supports of the $S_j$ and $T$ we have $\mu_{S_j} \rightarrow \mu_{T \llcorner B_{r_1}(p)}$. Thus, in particular on compact sets we have that $\mathrm{Spt}(S_j) \rightarrow \mathrm{Spt}(T)$ in the Hausdorff distance.

    \item The $\mathrm{Spt}(S_j)$ are smooth hypersurfaces. \qed
\end{itemize}
\end{thm}

Using Theorem \ref{thm: Kostas foliation} we now establish the desired surgery procedure, the proof of which is similar to \cite[Proposition 4.1]{C-L-S}.

\begin{prop}\label{prop: CMC surgery}
Let $(N^{n+1},g)$ be a Riemannian manifold and $T = (\partial [A]) \llcorner B_1(p)$ be a current associated to $A \in \mathcal{C}(N)$ with the properties as stated above and satisfying the hypotheses of Theorem \ref{thm: Kostas foliation}. Given $r \in (0,r_1)$ and any $\varepsilon > 0$ there exists a current $\tilde{T}$ and metric $\tilde{g}$ with the following properties:
\begin{itemize}
    \item $\mathrm{Sing}(\tilde{T}) = \emptyset$.

    \item $\tilde{T}$ is a critical point of $\mathcal{F}_\lambda$ in $B_1(p)$ with respect to the metric $\tilde{g}$.

    \item $\mathrm{Spt}(\tilde{T}) \setminus B_r(p) = \mathrm{Spt}(T) \setminus B_r(p)$ and $\partial( \tilde{T} \llcorner B_r(p)) = \partial (T \llcorner B_r(p))$.

    \item $d_{\mathcal{H}}(\mathrm{Spt}(T),\mathrm{Spt}(\tilde{T})) < \varepsilon$, where here $d_{\mathcal{H}}$ denotes the Hausdorff distance.
    
    \item $||\tilde{g} - g||_{C^{k,\alpha}} < \varepsilon$ for any $k \geq 1$ and $\alpha \in (0,1)$, with $g = \tilde{g}$ on $N \setminus B_r(p)$.
\end{itemize}
\end{prop}
\begin{proof}
The case $\lambda = 0$ is precisely the content of \cite[Proposition 4.1]{C-L-S}. We thus consider the case $\lambda \in \mathbb{R} \setminus \{0\}$.

\bigskip

As $\mathrm{Sing}(T) = \{p\}$, for each $r \in (0,r_1)$ we have that $(\mathrm{Spt}(T) \cap B_r (p) ) \setminus \{ p \}$ is smooth. We apply Theorem \ref{thm: Kostas foliation} to see that there exists some sequence, $S_j$, of smooth constant mean curvature hypersurfaces such that the $S_j$ converge as currents to $T \llcorner B_r (p)$. Allard's Theorem, (see \cite[Chapter 5]{simongmt} and \cite[Lemma 1.14]{hardt-simon}) implies that we may write (for $j$ sufficiently large) the intersection of $\mathrm{Spt}(S_j)$ with the annulus $A(p, \frac{r}{4}, \frac{3r}{4}) = B_{\frac{3r}{4}}(p) \setminus \overline{B_{\frac{r}{4}}(p)}$ as a smooth graph over $\mathrm{Spt}(T) \cap A(p, \frac{r}{4}, \frac{3r}{4})$. 

\bigskip

Explicitly, let $u_j \in C^2(\mathrm{Spt}(T) \cap A(p, \frac{r}{4}, \frac{3r}{4}))$ denote the graphing function of $\mathrm{Spt}(S_j)$ over $\mathrm{Spt}(T) \cap A(p, \frac{r}{4}, \frac{3r}{4})$ and let $\varphi$ be a smooth cutoff function taking values in $[0,1]$ such that $\varphi = 1$ on $B_\frac{3r}{8}(p)$ and $\varphi = 0$ outside $B_\frac{5r}{8}(p)$. We then denote by $T + \varphi u_j$ the image of the normal graph of the function $\varphi u_j$ over $\mathrm{Spt}(T) \cap A(p, \frac{r}{4}, \frac{3r}{4})$.

\bigskip

We now define $\widetilde{T} = (T \setminus B_r(p)) \cup (S_j \cap B_\frac{r}{4}) \cup ((T + \varphi u_j) \cap A(p, \frac{r}{4}, \frac{3r}{4}))$ for $j$ large enough to ensure that $\mathrm{Spt}(S_j)$ is smooth and graphical as above. Note that the $\mathrm{Spt}(\widetilde{T})$ is smooth as $\mathrm{Spt}(S_j)$, $\mathrm{Spt}(T) \setminus B_r(p)$ and $(\mathrm{Spt}(T) + \varphi u_j) \cap A(p,\frac{r}{4}, \frac{3r}{4})$ are smooth; hence $\mathrm{Sing}(\widetilde{T}) = \emptyset$. By construction we ensure that $\mathrm{Spt}(\tilde{T}) \setminus B_r(p) = \mathrm{Spt}(T) \setminus B_r(p)$ and $\partial( \tilde{T} \llcorner B_r(p)) = \partial (T \llcorner B_r(p))$. Note that for $j$ sufficiently large we have that $d_{\mathcal{H}}(\mathrm{Spt}(T),\mathrm{Spt}(\tilde{T})) < \varepsilon$ by the properties in the conclusion of Theorem \ref{thm: Kostas foliation}.

\bigskip
Let $H_g(x)$ denote the mean curvature of the hypersurface $\widetilde{T}$ with respect to the metric $g$ at a point $x \in \widetilde{T}$. By construction, $H_g$ may not be equal to $\lambda$ only on $A(p, \frac{r}{4}, \frac{3r}{4})$ (as both $T$ and $S_j$ are critical points of $\mathcal{F}_\lambda$). It remains to construct a new metric $\tilde{g}$, close to $g$, and show that $\widetilde{T}$ has constant mean curvature $\lambda$ with respect to this new metric. In the construction of $\widetilde{T}$ above, as $\lambda \in \mathbb{R} \setminus \{0\}$ and the graphing functions, $u_j$, converge to $0$, we may choose $j$ large enough to ensure that $H_g$ and $\lambda$ have the same sign, so that $\frac{H_g(x)}{\lambda} > 0$ for each $x \in \mathrm{Spt}(\widetilde{T})$.

\bigskip

For some smooth function $f$ on $N$ to be determined we set $\tilde{g} = e^{2f}g$; by standard results for conformal change of metric (e.g.~see \cite[Chapter II, Proposition 3.9]{sakai}) we then have that the mean curvature, $H_{\tilde{g}}$, of $\widetilde{T}$ with respect to the metric $\tilde{g}$ satisfies $H_{\tilde{g}}(x) = e^{-f}(H_g(x) + \frac{\partial f}{\partial \nu})$ at each point $x \in \widetilde{T}$, where $\nu$ is the unit normal on $\widetilde{T}$ (agreeing with $T$ outside of $A(p,\frac{r}{4},\frac{3r}{4})$).

\bigskip

We define a smooth cutoff function, $z$, such that $z(x) \equiv 1$ if $\mathrm{dist}_N(y,\mathrm{Spt}(\widetilde{T}) \cap A(p,\frac{r}{4},\frac{3r}{4})) < \frac{r}{20}$ and $z \equiv 0$ whenever $\mathrm{dist}_N(y,\widetilde{T} \cap A(p,\frac{r}{4},\frac{3r}{4})) > \frac{r}{10}$. We denote  by $\Pi(y)$ the closest point projection to $\widetilde{T}$ so that $H_g(\Pi(y))$ is a well defined function, supported in a tubular neighbourhood of $\mathrm{Spt}(\widetilde{T}) \cap A(p,\frac{r}{4},\frac{3r}{4})$. We now solve for $H_{\tilde{g}} = \lambda$ for each $x \in \mathrm{Spt}(\tilde{T})$ by setting $f(y) = \log\left(\frac{H_g(\Pi(y))}{\lambda}\right)z(y)$; note then  by construction $\mathrm{Spt}(f) \subset A(p,\frac{r}{4},\frac{3r}{4})$ and $\frac{\partial f}{\partial \nu} = 0$ on $\mathrm{Spt}(\widetilde{T}) \cap A(p,\frac{r}{4},\frac{3r}{4})$ (as $f$ is constant along normal geodesics to this region).

\bigskip

Thus we have that $H_{\tilde{g}}(\widetilde{T}) = \lambda$ for the choice of $f$ as above, so in particular $\widetilde{T}$ is a critical point of $\mathcal{F}_\lambda$, with the change in metric occurring only in $A(p, \frac{r}{4}, \frac{3r}{4})$. Furthermore, by the smoothness of $z$  and the fact that the graphing functions, $u_j$, are converging to $0$ ensures $||H_g(\Pi(y)) - \lambda||_{C^{2,\alpha}} \rightarrow 0$ as $j \rightarrow \infty$, we see that for each $k \geq 1$ and $\alpha \in (0,1)$ we have
\begin{equation*}
   ||e^{2f} - 1 ||_{C^{k,\alpha}} = \bigg|\bigg|\left(\frac{H_g(\Pi(y))}{\lambda}\right)^{2z(y)} - 1\bigg|\bigg|_{C^{k,\alpha}} \rightarrow 0 \text{ as } j \rightarrow \infty,
\end{equation*}
and so
\begin{equation*}
|| \tilde{g} - g ||_{C^{k,\alpha}} \leq ||e^{2f} - 1 ||_{C^{k,\alpha}} || g ||_{C^{k,\alpha}} \rightarrow 0 \text{ as } j \rightarrow \infty.
\end{equation*}
Hence by choosing $j$ sufficiently large we ensure the metric change is arbitrarily small.
\end{proof}

\begin{rem}\label{rem: surgery works on min-max cmc}
In the proof of Theorem \ref{thm: local minimisation in positive ricci} we will establish that the one-parameter Allen--Cahn min-max procedure of \cite{bellettini-wickramasekera} with constant prescribing function $\lambda$, in a compact Riemannian manifold of dimension $3$ or higher with Ricci positive curvature, produces a closed embedded hypersurface, $M$, of constant mean curvature $\lambda$ which is locally $\mathcal{F}_\lambda$-minimising around isolated singular points. For $M$ as above, around isolated singularities with regular tangent cones the results of \cite{simon} imply that $M$ is locally a graph over its unique tangent cone. Thus, $M$ will satisfy all the required properties of $T$ specified throughout this section in a sufficiently small ball around an isolated singular point with regular tangent cone; with this in hand we may thus apply the results of this section to $M$.
\end{rem}

\section{Signed distance and the one-dimensional profile}\label{sec: signed distance and one-dimensional profile}

\subsection{Singular behaviour of the distance function}
Recall that $M = \partial^* E \subset N$ is assumed to be a closed embedded hypersurface of constant mean curvature $\lambda$, smooth away from a closed singular set, denoted $\mathrm{Sing}(M) = \overline{M} \setminus M$, of Hausdorff dimension at most $n - 7$. We now adapt some of the analysis in \cite[Section 3]{bellettini} to the setting of hypersurfaces of constant mean curvature.

\bigskip

Denoting by $S_{d_{\overline{M}}}$ the set of points in $N \setminus \overline{M}$ where the distance function $d_{\overline{M}}$ fails to be differentiable (precisely the set of points $x \in N \setminus \overline{M}$ such that there exist two or more geodesics realising $d_{\overline{M}}(x)$), we then have that $d_{\overline{M}}$ is $C^1$ on $N \setminus (\overline{M} \cup \overline{S_{d_{\overline{M}}}})$ and that $S_{d_{\overline{M}}}$ is countably $n$-rectifiable (by \cite{Alberti1994}). We now show that $\overline{S_{d_{\overline{M}}}}$ is a countably $n$-rectifiable set; this fact will allow us to work with the smooth portions of the signed level sets of $d_{\overline{M}}$. To establish this we will need the following lemma; the proof of which is identical to that of \cite[Lemma 3.1]{bellettini} replacing the use of the sheeting theorem in \cite{wickramasekera} by the one in \cite{bellettini-wickramasekera-pmc-reg}.

\begin{lem}\label{lem: geodesic touching} (Geodesic Touching)
Let $x \in N \setminus \overline{M}$, then any minimising geodesic connecting $x$ to $\overline{M}$ (i.e.~a geodesic whose length realises $d_{\overline{M}}(x)$) with endpoint $y \in \overline{M}$ is such that $y \in M$ (i.e.~is a regular point of $M$). \qed
\end{lem}

Let $F(y, t) = \exp_y(t \nu(y))$ where $y \in M$, $t \in \mathbb{R}$ and $\nu$ the choice of unit normal to $M$ pointing into $E$. We then have that $F(y,t)$ is a geodesic emanating from $M$ orthogonally; here we interpret $F(y,t)$ for $t < 0$ by $\exp_y(t(\nu^-(y)))$ where $\nu^-$ is the unit normal to $M$ pointing into $N \setminus E$. We define $\sigma^+(y), \sigma^-(y) \in \mathbb{R}$ with $\sigma^+(y) > 0$ and $\sigma^-(y) < 0$ chosen so that $F(y,t)$ is the minimising geodesic between its endpoint $y$, on $M$, and $F(y,t)$ for all $\sigma^-(y) \leq t \leq \sigma^+(y)$ but fails to be the minimising geodesic, between $y$ and $F(y,t)$, for $t > \sigma^+(y)$ or $t < \sigma^-(y)$. With this definition we define the $\textit{cut locus}$ of $M$ to be 
\begin{equation}\label{eqn: cut locus of M}
    \mathrm{Cut}(M) = \{ F(y,\sigma^\pm(y)) \: : \: y \in M, \sigma^\pm(y) < \infty \}.
\end{equation}
Standard theory (e.g. \cite{sakai}) for geodesics characterises the cut locus in the following manner: if $x = F(y,\sigma^\pm(y)) \in \mathrm{Cut}(M)$ then either their exist (at least) two distinct geodesics realising $d_{\overline{M}}(x)$ or $F : M \times (0,\infty) \rightarrow N$ is such that $dF(y,\sigma^\pm (y))$ is not invertible. We then see that $\overline{S_{d_{\overline{M}}}} \cap (N \setminus \overline{M}) = \mathrm{Cut}(M)$ (c.f. proposition 4.6. in \cite{mantegazza-mennucci}) and so in order to establish the countable $n$-rectifiability of $\overline{S_{d_{\overline{M}}}}$ it is sufficient to show that $\mathrm{Cut}(M) \setminus S_{d_{\overline{M}}}$ is countably $n$-rectifiable in $N \setminus \overline{M}$; this fact is the analogue of proposition 4.9. in \cite{mantegazza-mennucci} in our setting (where the surface $M$ may be singular). Observe that $\overline{S_{d_{\overline{M}}}} \cap \overline{M} \subset \mathrm{Sing}(M)$ and as noted above $\mathrm{Sing}(M)$ has zero $\mathcal{H}^n$ measure, so this set does not affect the rectifiability. 

\bigskip

The proof of \cite[Proposition 4.9.]{mantegazza-mennucci} may be adapted (exactly as in \cite[Section 3]{bellettini}) to our setting by virtue of the fact that their arguments are local to points away from $\overline{M}$, and hence we may apply the arguments used in their proof in our situation without change. Thus we conclude that $\overline{S_{d_{\overline{M}}}}$ is countably $n$-rectifiable, consequently its $\mathcal{H}^{n+1}$ measure is zero. 

\bigskip
The level sets of $d_{\overline{M}}$ are smooth on $N \setminus (\overline{M} \cup \overline{S_{d_{\overline{M}}}})$ by virtue of the Implicit Function Theorem, invertibility of $F$ and differentiability of $d_{\overline{M}}$ on this set. The arguments in the proof of \cite[Proposition 4.6.]{mantegazza-mennucci} show further that $F$ restricts to a diffeomorphism,
\begin{equation*}
F : \{ (y, t) \: : \: y \in M,  t \in (\sigma^-(y),\sigma^+(y)) \} \rightarrow N \setminus (\overline{M} \cup \overline{S_{d_{\overline{M}}}}).
\end{equation*}
We then extend $F$ to $M \times \{0\}$ by setting $F(y,0) = y$ so that the image of the extension of $F$ is now $N \setminus (\mathrm{Sing}(M) \cup \overline{S_{d_{\overline{M}}}})$. Finally, by defining
\begin{equation}\label{eqn: coords on M}
    V_M = \{ (y,t) \: | \: y \in M, t \in \left(\sigma^-(y), \sigma^+(y)\right)\}. 
\end{equation}
we have that $V_M$ is diffeomorphic to $N \setminus (\mathrm{Sing}(M) \cup \overline{S_{d_{\overline{M}}}})$.

\begin{rem}\label{rem: two-sided neighbourhoods of compact subsets of M}
For each compact $K \subset M$ the continuity of the functions $\sigma^\pm (x)$ (which follows from \cite[Chapter III, Lemma 4.2]{sakai}) on $M$ implies that there exists some $c_K > 0$ such that $0 < c_K < \min_{x \in K}\{\sigma^+ (x), |\sigma^-(x)|\}$. For such a set $K$, as $M$ itself is a two-sided hypersurface, there is a two-sided tubular neighbourhood of $K$ when viewed as a subset of $V_M$ (the coordinates defined in (\ref{eqn: coords on M})), given by $K \times (-c_K,c_K)$ with its closure a subset of $V_M$. Furthermore, the image under the map $F$ of this two-sided tubular neighbourhood, $F(K \times (-c_K,c_K)) \subset N$, is such that $F(K \times (-c_K,c_K)) \cap (\mathrm{Sing}(M) \cup \overline{S_{d_{\overline{M}}}}) = \emptyset$, by the definition of $c_K$.
\end{rem}

\begin{rem}\label{rem: projection definition}
We define a projection to $M$ on $N \setminus (\mathrm{Sing}(M) \cup \overline{S_{d_{\overline{M}}}})$. For each $y \in N \setminus (\mathrm{Sing}(M) \cup \overline{S_{d_{\overline{M}}}})$ there exists a unique geodesic in $N$ with endpoint $x \in M$ realising $d_{\overline{M}}(y)$ (i.e.~such that $d_{\overline{M}}(y) = d_N(x,y)$). We denote by $\Pi$ the smooth projection from a point in $N \setminus (\mathrm{Sing}(M) \cup \overline{S_{d_{\overline{M}}}} ))$ to its unique endpoint in $M$. Note that we may express this as $\Pi(x) = F \circ \Pi_{V_M} \circ F^{-1}(x)$, where $\Pi_{V_M}$ is the smooth projection map on $V_M$ sending a point $(x,s) \in V_M$ to $(x,0) \in V_M$.
\end{rem}

\subsection{Level sets of the signed distance function}\label{subsec: level sets of distance function}

We now define the \textit{signed distance}, corresponding to our choice of unit normal, $\nu$, to the hypersurface $M$ pointing into $E$ as
\begin{equation*}
    d^\pm_{\overline{M}} = 
    \begin{cases}
    + d_{\overline{M}} (x), \text{ if } x \in E\\
    0, \text{ if } x \in \overline{M}\\
    - d_{\overline{M}} (x), \text{ if } x \in N \setminus E
    \end{cases},
\end{equation*}
so that the positive sign corresponds to our point lying in $E$. Denoting by $S_{d^\pm_{\overline{M}}}$ the set of points in $N \setminus \overline{M}$ where $d^\pm_{\overline{M}}$ fails to be differentiable, we then have that $\overline{S_{d^\pm_{\overline{M}}}} = \overline{S_{d_{\overline{M}}}}$ and so $\overline{S_{d^\pm_{\overline{M}}}}$ is countably $n$-rectifiable with zero $\mathcal{H}^{n+1}$ measure.

\bigskip

We now consider the level sets

\begin{equation*}
    \Gamma(s) = \{ x \in N \: | \: d^\pm_{\overline{M}} = s\}.
\end{equation*}

These level sets, $\Gamma(s)$, of $d^\pm_{\overline{M}}$ are smooth in the open set $N \setminus (\overline{S_{d^\pm_{\overline{M}}}} \cup \overline{M})$ by the Implicit Function Theorem, as $d^\pm_{\overline{M}}$ is smooth and the exponential map is invertible on this open set. Note that we have $\Gamma(s) = \emptyset$ for $|s| > d(N)$, where $d(N)$ is the diameter of $N$.

\bigskip

We use the following notation to refer to the smooth parts of the level sets of $d^\pm_{\overline{M}}$, setting $\widetilde{\Gamma}(0) = M$, and for $s \neq 0$,
\begin{equation*}
    \widetilde{\Gamma}(s) = \Gamma(s) \setminus \overline{S_{d^\pm_{\overline{M}}}}.
\end{equation*}
As $\overline{S_{d^\pm_{\overline{M}}}}$ is countably $n$-rectifiable, and hence has vanishing $\mathcal{H}^{n+1}$ measure, we may apply the co-area formula (slicing with $d^\pm_{\overline{M}}$) to conclude that 
\begin{equation}\label{eqn: smooth parts of level sets a.e.}
 \text{for almost every } s \in \mathbb{R} \text{ we have } \mathcal{H}^n\left(\Gamma(s) \cap \overline{S_{d^\pm_{\overline{M}}}}\right) = 0.
\end{equation}
Recall the diffeomorphism $F : V_{M} \rightarrow N \setminus (\mathrm{Sing}(M) \cup Cut(M))$ defined by $F(y,t) = \exp_y(t\nu(y))$ and the coordinates $V_{M}$ on $N \setminus (\mathrm{Sing}(M) \cup Cut(M))$ as defined in (\ref{eqn: coords on M}). We equip $V_{M}$ with the pull-back metric via the map $F$, giving the usual induced metric for $M$ on $M \times \{0\} \subset V_{M}$. Note that we have $F^{-1}(\widetilde{\Gamma}(s)) = M \times \{s\} \subset V_{M}$. We now work on $V_{M}$ to establish facts about the level sets of $d^\pm_{\overline{M}}$.

\bigskip

We choose local coordinates, $(x_1, \dots, x_n,s)$, on $V_{M}$ such that $\frac{\partial}{\partial x_1}, \dots, \frac{\partial}{\partial x_n}$ are a local frame around a point $x_0 \in M$ such that $\frac{\partial}{\partial s}$ is the unit speed of geodesics with constant base-point in $M$. The pullback metric (via $F$) then induces a volume form on $V_{M}$, and hence, at a point $(x_0,s_0) \in V_{M}$, an area element, $\theta(x,s)$, on the set $M \times \{s_0\} \subset V_{M}$. We set $\theta(x,s) = 0$ for $(x,s) \in (M \times \mathbb{R}) \setminus V_M$. We then have that by the structure of $V_M$ and $F$ that
\begin{equation}\label{eqn: integral of area element is good part of level set}
    \int_{M} \theta(x,s) dx_1\cdots dx_n = \mathcal{H}^n(\widetilde{\Gamma}(s)).
\end{equation}
Note that as the pull-back metric gives the induced metric for $M$ on $M \times \{0\}$ we have that
\begin{equation*}
    \int_{M} \theta(x,0) dx_1\cdots dx_n = \mathcal{H}^n(M).
\end{equation*}
As the volume form is smooth on $V_{M}$ we have that the induced area element, $\theta(x,s)$, is continuous in both variables, in particular we have for any $x \in M$ that
\begin{equation}\label{eqn: aream elements converge back to 0}
    \theta(x,s) \rightarrow \theta(x,0) \text{ as } s \rightarrow 0.
\end{equation}
We have from \cite[Theorem 3.11.]{tubes} that
\begin{equation}\label{eqn: differential equation for area elements log form}
    \frac{\partial}{\partial s} \log( \theta(x,s)) = -H(x,s),
\end{equation}
so that
\begin{equation}\label{eqn: differential equation for area elements}
    \partial_s \theta(x,s) = -H(x,s)\theta(x,s).
\end{equation}
Here $H(x,s)$ denotes the mean curvature of the the pullback of $\Gamma(s)$ to $V_{M}$ at the point $(x,s) \in V_{M}$. Note that $H(x,0) = \lambda$ as $M$ is a hypersurface of constant mean curvature $\lambda$.

\bigskip

We also recall the Ricatti equation, \cite[Corollary 3.6.]{tubes}, in the following form
\begin{equation}\label{eqn: Ricatti equation}
    \frac{\partial}{\partial s}H(x,s) \geq \min_N \mathrm{Ric}_g.
\end{equation}
Let us hereafter denote $m = \min_N \mathrm{Ric}_g> 0$. We then have that
\begin{equation}\label{eqn: geometry along level sets}
    \begin{cases}
    H(x,s) \geq \lambda + ms \text{ for } s > 0\\
    H(x,0) = \lambda\\
    H(x,s) \leq \lambda + ms \text{ for } s < 0
    \end{cases}.
\end{equation}
Combining (\ref{eqn: differential equation for area elements log form}), (\ref{eqn: geometry along level sets}) and applying the Fundamental Theorem of Calculus we see that for each $t \in \mathbb{R}$ we have 
\begin{equation*}
    \log (\theta(x,t)) \leq  - \int_0^t (ms + \lambda) \: ds,
\end{equation*}
from which we conclude that
\begin{equation}\label{eqn: first bound on area elements}
    \theta(x,t) \leq e^{-t\left(\frac{mt}{2} + \lambda \right)}.
\end{equation}
 Noting that the quadratic $-t\left(\frac{mt}{2} + \lambda \right)$ is maximised for $t = -\frac{\lambda}{m}$, from (\ref{eqn: first bound on area elements}) we see that
\begin{equation}\label{eqn: second bound on area elements}
    \theta(x,s) \leq e^{\frac{\lambda^2}{2m}}.
\end{equation}
We then apply the Dominated Convergence Theorem to see that by (\ref{eqn: aream elements converge back to 0}) and (\ref{eqn: second bound on area elements}) we have
\begin{equation*}
    \int_{M} \theta(x,s) dx_1 \cdots dx_n \rightarrow \int_{M} \theta(x,0) dx_1 \cdots dx_n,
\end{equation*}
so in particular by (\ref{eqn: integral of area element is good part of level set})
\begin{equation}\label{eqn: convergence of measure of level sets to M}
\mathcal{H}^n(\widetilde{\Gamma}(s)) \rightarrow \mathcal{H}^n(M) \text{ as } s \rightarrow 0.
\end{equation}
Let us record for later use that (\ref{eqn: convergence of measure of level sets to M}) above implies that
\begin{equation}\label{eqn: essential infimum of measure of level sets converges to M}
   {\essinf}_{s \in [-2\varepsilon \Lambda_\varepsilon, 2\varepsilon \Lambda_\varepsilon]}\mathcal{H}^n(\widetilde{\Gamma}(s)) \rightarrow \mathcal{H}^n(M) \text{ as } \varepsilon \rightarrow 0,
\end{equation}
and
\begin{equation}\label{eqn: essential supremum of measure of level sets converges to M}
   {\esssup}_{s \in [-2\varepsilon \Lambda_\varepsilon, 2\varepsilon \Lambda_\varepsilon]}\mathcal{H}^n(\widetilde{\Gamma}(s)) \rightarrow \mathcal{H}^n(M) \text{ as } \varepsilon \rightarrow 0.
\end{equation}

\subsection{Approximation with the one-dimensional profile}\label{subsec: one dimensional profile}

Denote by $\mathbb{H}$ the monotonically increasing solution to the following ODE 
$$u'' - W'(u) = 0,$$
namely the one-dimensional Allen--Cahn equation, subject to the conditions that 
$$\begin{cases}
    \mathbb{H}(0) = 0\\
    \lim_{s \rightarrow \pm \infty} \mathbb{H}(s) = \pm 1
\end{cases}.$$ In particular, for the standard choice of potential, $W(s) = \frac{(1-s^2)^2}{4}$, we have explicitly that $\mathbb{H}(s) = \mathrm{tanh}(\frac{s}{\sqrt{2}})$. We note also that the re-scaled function $\mathbb{H}^\varepsilon(s) = \mathbb{H}(\frac{s}{\varepsilon})$ solves the ODE
$$\varepsilon u'' - \frac{W'(u)}{\varepsilon} = 0.$$
We now recall the construction of a smooth, increasing truncation of the one-dimensional heteroclinic solution to the rescaled Allen--Cahn equation, $\overline{\mathbb{H}}^\varepsilon$, as used in \cite{bellettini}. For the construction of Allen--Cahn approximations, this truncation will be placed in the normal direction for the various underlying hypersurfaces sketched in Step 2 of the proof strategy outlined in Subsection \ref{subsec: strategy}. 

\bigskip

The truncation we construct has the advantage that it will be constant, identically equal to either $\pm 1$, outside of an interval (depending on $\varepsilon > 0$) of the form $[-6\varepsilon|\log (\varepsilon)|,6\varepsilon|\log (\varepsilon)|]$; thus outside of a tubular neighbourhood (of radius $6\varepsilon|\log (\varepsilon)|$) our approximating functions will contribute no energy (as $e_\varepsilon(\overline{\mathbb{H}}^\varepsilon(s))= 0$ outside of the interval $[-6\varepsilon|\log (\varepsilon)|,6\varepsilon\log (\varepsilon)]$).

\bigskip

We first perform the truncation for $\mathbb{H}$ and then re-scale to the $\varepsilon$ level. We let $\chi \in C^\infty_c(\mathbb{R})$ be a smooth cutoff function such that the following properties hold  
$$
\begin{cases}
    \chi \equiv 1 \text{ on } (-1,1)\\
    \chi \equiv 0 \text{ on } \mathbb{R} \setminus (-2,2)\\
    \chi(-s) = \chi(s) \text{ for all } s \in \mathbb{R}\\
    \chi'(s) \leq 0 \text{ for } s \geq 0
\end{cases}.$$
Denoting $\Lambda_\varepsilon = 3|\log (\varepsilon)|$ we then define the trunctation
$$\overline{\mathbb{H}}(s) = 
\begin{cases}
    \chi\left(\frac{s}{\Lambda_\varepsilon}\right)\mathbb{H}(s) + \left(1 - \chi\left(\frac{\Lambda_\varepsilon}{|s|}\right)\right) \text{ for } s > 0\\
     \chi\left(\frac{s}{\Lambda_\varepsilon}\right)\mathbb{H}(s) - \left(1 - \chi\left(\frac{\Lambda_\varepsilon}{|s|}\right)\right) \text{ for } s < 0
\end{cases}.$$
With this choice of truncation we then ensure that the following properties hold for the re-scaled truncation, $\overline{\mathbb{H}}^\varepsilon(s) = \overline{\mathbb{H}}(\frac{s}{\varepsilon})$,
$$
\begin{cases}
    \overline{\mathbb{H}}^\varepsilon \equiv \mathbb{H}^\varepsilon \text{ on } (-\varepsilon\Lambda_\varepsilon,\varepsilon\Lambda_\varepsilon)\\
    \overline{\mathbb{H}}^\varepsilon \equiv 1 \text{ on } (-\infty,-2\varepsilon\Lambda_\varepsilon]\\
    \overline{\mathbb{H}}^\varepsilon \equiv 1 \text{ on } [2\varepsilon\Lambda_\varepsilon, \infty)\\
\end{cases};$$
thus we have that $e_\varepsilon(\overline{\mathbb{H}}^\varepsilon(s)) = 0$ for $|s| \geq 2\varepsilon \Lambda_\varepsilon$, as desired. 

\bigskip

Summarising the computations carried out in \cite[Section 2.2]{bellettini}, we then deduce that
\begin{equation}\label{eqn: 1d energy approx 1}
    \mathcal{E}_\varepsilon (\overline{\mathbb{H}}^\varepsilon) \rightarrow 1 \text{ as } \varepsilon \rightarrow 0,
\end{equation}
where the specific convergence is such that for fixed $\beta > 0$ we have
\begin{equation}\label{eqn: 1d solution errors in eps^2}
   1 - \beta \varepsilon^2 \leq \mathcal{E}_\varepsilon(\overline{\mathbb{H}}^\varepsilon) \leq 1 + \beta \varepsilon^2.
\end{equation}
We define the one-dimensional profile to be
\begin{equation*}
v_\varepsilon (x) =
\overline{\mathbb{H}}^\varepsilon ( d^\pm_{\overline{M}}(x)).
\end{equation*}
As $\overline{\mathbb{H}}^\varepsilon$ is smooth and the distance function is Lipschitz, it follows that $v_\varepsilon \in W^{1,2}(N)$. We now prove that (\ref{eqn: 1d profile approx cmc energy}) holds,
showing that $v_\varepsilon$ acts as an Allen--Cahn "approximation" of $M$ in the sense that it recovers the $\varepsilon \rightarrow 0$ limit of the energies of the critical points obtained by the min-max in \cite{bellettini-wickramasekera}. We will exploit this fact directly in the construction of the paths in Section \ref{sec: paths}, eventually allowing us to establish Theorem \ref{thm: local minimisation in positive ricci}.

\bigskip

By the co-area formula to $d^\pm_{\overline{M}}$, noting that $|\nabla d^\pm_{\overline{M}}| = 1$, we compute, similarly to \cite[Section 3.6]{Bellettini-Workman}, that
\begin{align*}
    \mathcal{F}_{\varepsilon,\lambda}(v_\varepsilon) &= \mathcal{E}_\varepsilon(v_\varepsilon) - \frac{\lambda}{2}\int_N v_\varepsilon\\
    &= \frac{1}{2\sigma}\int_N \frac{\varepsilon}{2}|\nabla v_\varepsilon|^2 + \frac{W(v_\varepsilon)}{\varepsilon} - \frac{\lambda}{2}\int_N v_\varepsilon\\
    &= \frac{1}{2\sigma} \int_{\mathbb{R}}\int_{\Gamma(t)} \frac{\varepsilon}{2}\left(\left(\overline{\mathbb{H}}^\varepsilon\right)'(t)\right)^2 + \frac{W(\overline{\mathbb{H}}^\varepsilon(t))}{\varepsilon} d\mathcal{H}^n dt - \frac{\lambda}{2}\int_{\mathbb{R}}\int_{\Gamma(t)} \overline{\mathbb{H}}^\varepsilon(t)d\mathcal{H}^ndt.    
\end{align*}
Using (\ref{eqn: smooth parts of level sets a.e.}) and the definition of $e_\varepsilon$ as in Subsection \ref{subsec: notation}, we then have that
\begin{equation*}
    \mathcal{F}_{\varepsilon,\lambda}(v_\varepsilon) = \frac{1}{2\sigma}\int_{\mathbb{R}} e_\varepsilon(\overline{\mathbb{H}}^\varepsilon(t))\mathcal{H}^n(\widetilde{\Gamma}(t)) dt - \frac{\lambda}{2}\int_{\mathbb{R}} \overline{\mathbb{H}}^\varepsilon(t) \mathcal{H}^n(\widetilde{\Gamma}(t)) dt. 
\end{equation*}
From the properties of $\overline{\mathbb{H}}^\varepsilon$ as stated above we may obtain the following bounds
\begin{equation*}
    \mathcal{F}_{\varepsilon,\lambda}(v_\varepsilon) \geq {\essinf}_{s \in [-2\varepsilon \Lambda_\varepsilon, 2\varepsilon \Lambda_\varepsilon]}\mathcal{H}^n(\widetilde{\Gamma}(s)) \mathcal{E}_\varepsilon(\overline{\mathbb{H}}^\varepsilon) - \frac{\lambda}{2}\int_{-2\varepsilon\Lambda_\varepsilon}^\infty \mathcal{H}^n(\widetilde{\Gamma}(t))dt + \frac{\lambda}{2}\int_{-\infty}^{-2\varepsilon\Lambda_\varepsilon} \mathcal{H}^n(\widetilde{\Gamma}(t))dt,
\end{equation*}
and also 
\begin{equation*}
    \mathcal{F}_{\varepsilon,\lambda}(v_\varepsilon) \leq 
    {\esssup}_{s \in [-2\varepsilon \Lambda_\varepsilon, 2\varepsilon \Lambda_\varepsilon]}\mathcal{H}^n(\widetilde{\Gamma}(s)) \mathcal{E}_\varepsilon(\overline{\mathbb{H}}^\varepsilon) - \frac{\lambda}{2}\int_{2\varepsilon\Lambda_\varepsilon}^\infty \mathcal{H}^n(\widetilde{\Gamma}(t))dt + \frac{\lambda}{2}\int_{-\infty}^{2\varepsilon\Lambda_\varepsilon} \mathcal{H}^n(\widetilde{\Gamma}(t))dt.
\end{equation*}
Observe that
\begin{equation*}
    \mathrm{Vol}_g(E) = \mathcal{H}^{n+1}(\{x \in N \: | \: d^\pm_{\overline{M}} > 0\}) = \int_0^\infty \mathcal{H}^n(\widetilde{\Gamma}(t))dt = \lim_{\varepsilon \rightarrow 0} \int_{\pm 2\varepsilon\Lambda_\varepsilon}^\infty \mathcal{H}^n(\widetilde{\Gamma}(t))dt,
\end{equation*}
and
\begin{equation*}
    \mathrm{Vol}_g(N \setminus E) = \mathcal{H}^{n+1}(\{x \in N \: | \: d^\pm_{\overline{M}} < 0\}) = \int_{-\infty}^0 \mathcal{H}^n(\widetilde{\Gamma}(t))dt = \lim_{\varepsilon \rightarrow 0} \int^{\pm 2\varepsilon\Lambda_\varepsilon}_{-\infty} \mathcal{H}^n(\widetilde{\Gamma}(t))dt.
\end{equation*}
Combining these two identities with (\ref{eqn: essential infimum of measure of level sets converges to M}), (\ref{eqn: essential supremum of measure of level sets converges to M}) and (\ref{eqn: 1d energy approx 1}) in the above two bounds on $\mathcal{F}_{\varepsilon,\lambda}(v_\varepsilon)$ we conclude that
\begin{equation*}
\mathcal{F}_{\varepsilon,\lambda}(v_\varepsilon) \rightarrow \mathcal{H}^n(M) - \lambda \mathrm{Vol}_g(E) + \frac{\lambda}{2}\mathrm{Vol}_g(N) = \mathcal{F}_\lambda(E) \text{ as } \varepsilon \rightarrow 0,
\end{equation*}
as desired.

\section{Relating local properties of the energy to the geometry}\label{sec: local energy properties}

This section relates local behaviour of the energy of the one-dimensional profile, $v_\varepsilon = \overline{\mathbb{H}}^\varepsilon \circ d^\pm_{\overline{M}}$, to the local geometric properties of $E$. Recall that $M = \partial^* E \subset N$ is assumed to be a closed embedded hypersurface of constant mean curvature $\lambda$, smooth away from a closed singular set, denoted $\mathrm{Sing}(M) = \overline{M} \setminus M$, of Hausdorff dimension at most $n - 7$. Furthermore, $M$ separates $N \setminus \overline{M}$ into open sets, $E$ and $N \setminus E$, with common boundary $\overline{M}$ and $E \in \mathcal{C}(N)$ with $M = \partial^* E$. For an isolated singularity $p \in \overline{M}$ we fix throughout this section some $0 < r_1 < r_2 < \min\{R_p,R_l\}$ such that $M \cap \overline{B_{r_2}(p)} \setminus B_{r_1}(p)$ is smooth.

\subsection{Local smoothing of Caccioppoli sets}\label{subsec: local smoothing of caccioppoli sets}

We first establish a procedure to locally perturb Caccioppoli sets that are assumed to be smooth in an annular region to ensure that they are smooth in the entire ball.

\begin{prop}\label{prop: local smoothing}
Suppose that $F \in \mathcal{C}(N)$ is such that $F \setminus B_{r_1}(p) = E \setminus B_{r_1}(p)$. Fix some $\tilde{r}_2 \in (r_1,r_2)$, then, for each $\delta > 0$, there exists $F_\delta \in \mathcal{C}(N)$ with the following properties:
\begin{itemize}
    \item $\partial F_\delta$ is smooth in $B_{r_2}(p)$.
    
    \item $F_\delta \setminus B_{\tilde{r}_2}(p) = E \setminus B_{\tilde{r}_2}(p)$.
    
    \item $|\mathrm{Per}_g(F_\delta) - \mathrm{Per}_g(F)| \leq \delta$.
    
    \item $Vol_g(F_\delta \Delta F) \leq \delta$.
\end{itemize}   
Thus, in particular we have that $F_\delta$ agrees with $F$ outside of $B_{\tilde{r}_2}(p)$ and is such that 
\begin{equation*}
    |\mathcal{F}_\lambda(F_\delta) - \mathcal{F}_\lambda(F)| 
\leq (1 + \lambda)\delta.
\end{equation*}
\end{prop}

\begin{rem}
Though Proposition \ref{prop: local smoothing} is phrased in our setting for a constant mean curvature hypersurface, the proof makes no use of the variational assumption on $M$. Thus, the same result holds for any Caccioppoli set which satisfies the same properties as $M$, without any condition on the mean curvature.
\end{rem}

\begin{proof}
We divide the proof up into three steps, which we now briefly describe. First, by mollification we produce smooth functions, $s_\theta$, which converge to $\chi_F$ in the BV norm as $\theta \rightarrow 0$. Second, we show that almost every level set of the approximating functions is graphical in an annulus over $\partial F$ with $C^1$ norm of this graph converging to zero as $\theta \rightarrow 0$. Third, we produce our local smoothing, $F_\delta$, by interpolating in an annulus between a smooth level set of some $s_\theta$ (for $\theta$ chosen sufficiently small based on the given $\delta > 0$) and $\partial F$ via a cutoff function; concluding by showing that this local smoothing satisfies the desired properties.

\bigskip

\textit{Step 1:}~Recall that, as $r_2 < R_l$, $B_{r_2}(p)$ is $2$-bi-Lipschitz to the Euclidean ball $B_{r_2}^{\mathbb{R}^{n+1}}(0)$ via some a geodesic normal coordinate chart, $\phi$, with $\phi(p) = 0$ and so that $\frac{1}{2} \leq \sqrt{|g|} \leq 2$ on $B_{r_2}(p)$. We consider a radially symmetric mollifier $\rho \in C^\infty_c(B^{\mathbb{R}^{n+1}}_1(0))$ such that $\int_{\mathbb{R}^{n+1}} \rho = 1$ and for each $\theta > 0$ define $\rho_\theta (x) = \frac{1}{\theta^{n+1}}\rho(\frac{x}{\theta})$. We fix $\tilde{r}_2 \in (r_1, r_2)$ and consider, for $\theta < r_2 - \tilde{r}_2$, the function $s_\theta = (\chi_{\phi(F)} * \rho_\theta) \circ \phi \in C^\infty(B_{\tilde{r}_2}(p))$. We now show that $s_\theta$ approximates $\chi_F$ in the BV norm. First, by standard properties of mollifiers we have
\begin{align*}
    ||\chi_F - s_\theta||_{L^1(B_{\tilde{r}_2}(p)))} &= \int_{B_{\tilde{r}_2}(p)} |\chi_F - (\chi_{\phi(F)} * \rho_\theta) \circ \phi| d\mathcal{H}^n_g\\
    &= \int_{\phi(B_{\tilde{r}_2}(p))} |\chi_{\phi(F)} - (\chi_{\phi(F)} * \rho_\theta)|\sqrt{|g|}d\mathcal{L}^n\\
    &\leq 2 || \chi_{\phi(F)} - (\chi_{\phi(F)} * \rho_\theta)||_{L^1(\phi(B_{\tilde{r}_2}(p)))} \rightarrow 0 \text{ as } \theta \rightarrow 0. 
\end{align*}
Hence by the lower semi-continuity of the perimeter we have that 
\begin{equation*}
    \mathrm{Per}_g(F) = |D_g \chi_F|(B_{\tilde{r}_2}(p)) \leq \liminf_{\theta \rightarrow 0} |D_g s_\theta|(B_{\tilde{r}_2}(p)).
\end{equation*}
We let $X \in \Gamma^1_c(TB_{\tilde{r}_2}(p))$ with $|X|_g \leq 1$ and compute
\begin{align*}
    \int_{B_{\tilde{r}_2}(p)} s_\theta div_g X d\mathcal{H}^n_g &= \int_{\phi(B_{\tilde{r}_2}(p))} (\chi_{\phi(F)}*\rho_\theta) \partial_i(\sqrt{|g|}\hat{X}^i)d\mathcal{L}^n\\
    &= \int_{\phi(B_{\tilde{r}_2}(p))} \chi_{\phi(F)} \partial_i(\rho_\theta *\sqrt{|g|}\hat{X}^i)d\mathcal{L}^n\\
    &= \int_{B_{\tilde{r}_2}(p)} \chi_F div_g Y \leq |D_g \chi_F|(B_{\tilde{r}_2}(p)),
\end{align*}
where here $\hat{X} = \hat{X}^i\partial_i^\phi$ with $\hat{X}^i = X^i \circ \phi^{-1}$, and $Y = Y^i\partial_i^\phi$ with $Y^i = \frac{1}{\sqrt{|g|}} (\rho_\theta * \sqrt{|g|} \hat{X}^i) \circ \phi$ (so $|Y|_g \leq 1$). As the choice of $X$ above was arbitrary here we thus conclude that 
\begin{equation*}
    \lim_{\theta \rightarrow 0} |D_g s_\theta|(B_{\tilde{r}_2}(p)) = |D_g \chi_F|(B_{\tilde{r}_2}(p)). 
\end{equation*}
Arguing identically to the proof of \cite[Theorem 13.8]{maggi_2012} we conclude that the level sets of $s_\theta$ (which have smooth boundary for almost every $t \in (0,1)$ by Sard's Theorem), $L^t_\theta = \{s_\theta > t\}$ for a.e. $t \in (0,1)$ provide a sequence of open sets with smooth boundary such that $Vol_g(L^t_\theta \Delta F) \rightarrow 0$ and $Per(L^t_\theta ; A) \rightarrow Per(F ; A)$ for each open set $A \subset B_{\tilde{r}_2}(p)$, whenever $Per(F ; \partial A) = 0,$ as $\theta \rightarrow 0$.

\bigskip

\textit{Step 2:}~As the exponential map is a radial isometry we have that the geodeisic normal coordinate chart $\phi$ is such that $B_{r_2}^{\mathbb{R}^{n+1}}(0) = \phi(B_{r_2}(p))$ and $S = \phi\left(M \cap (\overline{B}_{r_2}(p) \setminus B_{r_1}(p))\right)$ is smooth in $B_{r_2}^{\mathbb{R}^{n+1}}(0) \setminus B_{r_1}^{\mathbb{R}^{n+1}}(0)$. In particular we may choose $r_1 < r_a < r_b < \tilde{r}_2$ so that
\begin{equation*}
    T = \{x + b\nu_S(x) \: | \: x \in S \cap (\overline{B_{r_b}^{\mathbb{R}^{n+1}}(0)} \setminus B_{r_a}^{\mathbb{R}^{n+1}}(0)), b \in [-\delta,\delta]\} \subset \subset B_{\tilde{r}_2}^{\mathbb{R}^{n+1}}(0) \setminus B_{r_1}^{\mathbb{R}^{n+1}}(0),
\end{equation*} 
for some $\delta > 0$. Here above we are denoting the unit normal induced on $S$ from $M$ by $\nu_S$; note then that 
\begin{equation}\label{eqn: E in $T$}
    F \cap \phi^{-1}(T) = \phi^{-1}(\{x + b\nu_S(x) \in T \: | \: b > 0\}).
\end{equation}
We now consider normal derivatives of the function $\chi_{\widehat{F}} * \rho_\theta = s_\theta \circ \phi^{-1}$ for $\theta < \delta$ (where we denote $\widehat{F} = \phi(F)$) in the tubular neighbourhood $T$ defined above. For $x + b\nu_S(x) \in T$ and $\theta > 0$ small enough so that $\nu_S(y) \cdot \nu_S(x) \geq \frac{1}{2}$ for any $y \in B_{2\theta}(x)$ ($S$ is smooth so $\nu_S$ varies continuously) we compute that when $b \in (-\theta,\theta)$ we have
\begin{align*}
    \nabla (\chi_{\widehat{F}} * \rho_\theta)(x + b\nu_S(x)) \cdot \nu_S(x) &= (\nabla \chi_{\widehat{F}} * \rho_\theta) (x + b\nu_s(x)) \cdot \nu_S(x) \\
    &= -((\mathcal{H}^{n}\llcorner S) \nu_S * \rho_\theta)(x + b\nu_S(x)) \cdot \nu_S(x)\\
    &= -\left(\int_{S \cap B_\theta(x + b\nu_S(x))} \rho_\theta (x + b\nu_S(x) - y)\nu_S(y)\cdot \nu_S(x)  d\mathcal{H}^n(y)\right) \\
    & \leq -\frac{1}{2}\left( \int_{S \cap B_\theta(x + b\nu_S(x))} \rho_\theta (x + b\nu_S(x) - y) d\mathcal{H}^n(y) \right) < 0.
\end{align*}
Here we've used that $D_g \chi_{\widehat{F}} = -(\mathcal{H}^n \llcorner S) \nu_S$ in $T$ as $S$ is smooth and by the triangle inequality we have both $ B_{\theta}(x + b\nu_S(x)) \subset B_{2\theta}(x)$ for $b < \theta$ and that the strict inequality holds as $S \cap B_{\theta - b}(x) \subset S \cap B_\theta (x + b\nu_S(x))$.

\bigskip

Thus for $b \in (-\theta,\theta)$ we have that $\nabla (\chi_{\widehat{F}} * \rho_\theta)(x + b\nu_s(x)) \cdot \nu_S(x) < 0$ and $\nabla (\chi_{\widehat{F}} * \rho_\theta)(x + b\nu_S(x)) = 0$ for $b \notin (-\theta,\theta)$ (as then $\chi_{\widehat{F}} * \rho_\theta \in \{0,1\}$). We conclude that for each $t \in (0,1)$ and $x \in S \cap (\overline{B_{r_b}^{\mathbb{R}^{n+1}}(0)} \setminus B_{r_a}^{\mathbb{R}^{n+1}}(0))$ there exists a unique $b_\theta^t(x) \in (-\theta, \theta)$ such that $(\chi_{\widehat{F}} * \rho_\theta )(x + b_\theta^t(x)\nu_S (x)) = t$; with $b^t_\theta \rightarrow 0$ point-wise on $S \cap (\overline{B_{r_b}^{\mathbb{R}^{n+1}}(0)} \setminus B_{r_a}^{\mathbb{R}^{n+1}}(0))$ as $\theta \rightarrow 0$. Note also that by the structure of $T$ we have
\begin{equation}\label{eqn: L^t_theta in T}
    L^t_\theta \cap \phi^{-1}(T) =  \phi^{-1}(\{x + b\nu_S(x) \in T \: | \: b < b^t_\theta(x)\}).
\end{equation}
Consider the function $h_\theta : S \cap (\overline{B_{r_b}^{\mathbb{R}^{n+1}}(0)} \setminus B_{r_a}^{\mathbb{R}^{n+1}}(0)) \times (-\theta,\theta) \rightarrow T$ defined by \begin{equation*}
    h^t_\theta(x,b) = (\chi_{\widehat{F}}*\rho_\theta)(x + b\nu_S(x)) - t.
\end{equation*} 
By the above calculation we have that $\frac{\partial h_\theta}{\partial b}(x + b\nu_S(x)) = \nabla (\chi_{\widehat{F}}*\rho_\theta)(x + b\nu_S(x)) \cdot \nu_S(x) < 0$ and so by the Implicit Function Theorem (working in charts in which $S \cap (\overline{B_{r_b}^{\mathbb{R}^{n+1}}(0)} \setminus B_{r_a}^{\mathbb{R}^{n+1}}(0))$ is locally a graph) we have that the function $b^t_\theta$ is smooth on $S \cap (\overline{B_{r_b}^{\mathbb{R}^{n+1}}(0)} \setminus B_{r_a}^{\mathbb{R}^{n+1}}(0))$ (noting that $h^t_\theta(x,b^t_\theta(x)) = 0$ for $x$ in this set). 

\bigskip

We now show that the directional derivatives of $b^t_\theta$ converge to zero as $\theta \rightarrow 0$. At a point $x \in S \cap (\overline{B_{r_b}^{\mathbb{R}^{n+1}}(0)} \setminus B_{r_a}^{\mathbb{R}^{n+1}}(0))$, denoting directional derivatives by $\frac{\partial}{\partial x_i}$, we compute
\begin{align*}
    \left| 
    \frac{\Big(\frac{\partial h^t_\theta}{\partial x_i}\Big)}{\Big(\frac{\partial h^t_\theta}{\partial b}\Big)} \right| &= \left| \frac{\nabla(\chi_{\widehat{F}} * \rho_\theta)(x+b\nu_S(x)) \cdot \frac{\partial}{\partial x_i}(x + b\nu_S(x))}{\nabla(\chi_{\widehat{F}} * \rho_\theta)(x+b\nu_S(x)) \cdot \nu_S(x))} \right|\\
    &= \left| \frac{\int_{S \cap B_\theta(x + b\nu_S(x))} \rho_\theta (x + b\nu_S(x) - y)\nu_S(y) \cdot \frac{\partial}{\partial x_i}(x + b\nu_S(x)) d\mathcal{H}^n(y)}{\int_{S \cap B_\theta(x + b\nu_S(x))} \rho_\theta (x + b\nu_S(x) - y)\nu_S(y) \cdot \nu_S(x) d\mathcal{H}^n(y)} \right|\\
    &\leq 2 \sup_{y \in S \cap B_\theta(x + b\nu_S(x))} \left|\nu_S(y) \cdot \frac{\partial}{\partial x_i}(x + b\nu_S(x))\right| \rightarrow 0 \text{ as } \theta \rightarrow 0.
\end{align*}
Here we see that the last term converges to zero as $\theta \rightarrow 0$ by using the triangle inequality, noting that $\nu_S(y) \cdot \frac{\partial}{\partial x_i} x \rightarrow 0$ by the continuity of $\nu_S$ and $\nu_S(y) \cdot b
\frac{\partial \nu_S(x)}{\partial x_i} \leq C \theta$ for a constant $C$ depending only on the normal of $S$. We then have from the Implicit Function Theorem that
\begin{equation*}
    \left| \frac{\partial b^t_\theta}{\partial x_i}  \right|(x) =  \left| 
    \frac{\Big(\frac{\partial h^t_\theta}{\partial x_i}\Big)}{\Big(\frac{\partial h^t_\theta}{\partial b}\Big)} \right|(x,b^t_\theta(x)) \rightarrow 0 \text{ as } \theta \rightarrow 0. 
\end{equation*}
We conclude that for $t \in (0,1)$ the function $b_\theta^t$ defined on $S \cap (\overline{B_{r_b}^{\mathbb{R}^{n+1}}(0)} \setminus B_{r_a}^{\mathbb{R}^{n+1}}(0))$ as above is smooth and in fact converges to zero uniformly in the $C^1$ norm as $\theta \rightarrow 0$.

\bigskip

\textit{Step 3:}~We fix $\tilde{r}_a \in (r_a, r_b)$ so that for $\theta < \tilde{r}_a - r_a$ we ensure (by the triangle inequality) that $x + b^t_\theta(x) \nu_S(x) \in B_{\tilde{r}_a}(p)$ whenever $x \in \partial B_{r_a}(p)$ (i.e.~when $|x| = r_a$). Fix a further $\tilde{r}_b \in (\tilde{r}_a,r_b)$ so that for $\theta < r_b - \tilde{r}_b$ we ensure (by the triangle inequality) that $x + b^t_\theta(x) \nu_S(x) \in N \setminus B_{\tilde{r}_b}(p)$ whenever $x \in \partial B_{r_b}(p)$ (i.e.~when $|x| = r_b$).

\bigskip

We then define $\eta \in C^\infty_c(\mathbb{R}^{n+1})$ to be a radially symmetric cut-off function with the following properties
\begin{equation*}
\begin{cases}
0 \leq \eta \leq 1 \text{ on } \mathbb{R}^{n+1}\\|\nabla \eta | \leq C_\eta \text{ on } \mathbb{R}^{n+1}\\
\eta(x) = 1 \text{ if } |x| \leq \tilde{r}_a \\
\eta(x) = 0 \text{ if } |x| \geq \tilde{r}_b,
\end{cases}
\end{equation*}
for some $C_\eta = C_\eta(\tilde{r}_a,\tilde{r}_b) > 0$.
\bigskip

Before proceeding to define the local smoothing we let
\begin{equation*}
    H^t_\theta = \phi^{-1}\left(\{x + b\nu_S(x) \in T \: | \: b < \eta(x)b^t_\theta(x) \}\right),
\end{equation*}
and, slightly abusing notation, we denote
\begin{equation*}
    \partial H^t_\theta = \phi^{-1}\left(\left\{ x + \eta(x)b^t_\theta(x) \nu_S(x) \: | \: x \in S \cap (\overline{B_{r_b}^{\mathbb{R}^{n+1}}(0)} \setminus B_{r_a}^{\mathbb{R}^{n+1}}(0))\right\}\right),
\end{equation*}
where the notation for boundary is justified by virtue of the fact that, as $M$ separates $T$, the "graph" given by $\phi^{-1}(x + \eta(x)b^t_\theta(x)\nu_S(x))$ separates $\phi^{-1}(T)$. 

\bigskip

We then consider the following hypersurface
\begin{equation*}
    \partial G^t_\theta = (M \setminus B_{\tilde{r}_b}(p)) \cup (\partial L^t_\theta \cap B_{\tilde{r}_a}(p)) \cup \left(\partial H^t_\theta \cap (B_{\tilde{r}_b}(p) \setminus B_{\tilde{r}_a}(p))\right)
\end{equation*}
which is smooth inside $B_{r_b}(p)$ as $H^t_\theta \cap \partial B_{\tilde{r}_a}(p) \in  \partial L^t_\theta$
and $H^t_\theta \cap \partial B_{\tilde{r}_b}(p) \in M$. We note also that as $\overline{M}$ satisfies the geodesic touching property from Lemma \ref{lem: geodesic touching},  $\partial G^t_\theta$ also satisfies the conclusions of Lemma \ref{lem: geodesic touching}.

\bigskip

Furthermore, by (\ref{eqn: E in $T$}) and (\ref{eqn: L^t_theta in T}), $\partial G^t_\theta$ arises as the boundary of the open set 
\begin{equation*}
    G^t_\theta = (E \setminus ((B_{\tilde{r}_b}(p) \cup \phi^{-1}(T)))  \cup (L^t_\theta \cap B_{\tilde{r}_b}(p) \setminus \phi^{-1}(T)) \cup (H^t_\theta \cap \phi^{-1}(T)),
\end{equation*}
and hence $G^t_\theta \setminus B_{\tilde{r}_2}(p) = E \setminus B_{\tilde{r}_2}(p)$ as $\phi^{-1}(T) \subset \subset B_{\tilde{r}_2}(p) \setminus B_{r_1}(p)$.

\bigskip

We may write
\begin{equation*}
    (L^t_\theta \Delta E) \cap \phi^{-1}(T) = \phi^{-1}\left(\left\{x + b\nu_S(x) \in T \: \bigg| \: 
    \begin{cases}
    0 < b < b^t_\theta(x), \text{ if } b^t_\theta(x) > 0\\
    b^t_\theta < b < 0, \text{ if } b^t_\theta(x) < 0
    \end{cases}\right\}\right).
\end{equation*}
As $0 \leq \eta \leq 1$, $\eta(x)b^t_\theta(x)$ has the same sign as $b^t_\theta$ and $|\eta(x)b^t_\theta(x)| \leq |b^t_\theta(x)|$. Therefore, by construction of $H^t_\theta$, we have that $(G^t_\theta \Delta E) \cap \phi^{-1}(T) \subset (L^t_\theta \Delta E) \cap \phi^{-1}(T)$. As $G^t_\theta$ agrees with either $E$ or $L^t_\theta$ outside of $\phi^{-1}(T)$ we conclude that $G^t_\theta \Delta F \subset L^t_\theta \Delta F$ and hence $Vol_g(G^t_\theta \Delta F) \leq Vol_g(L^t_\theta \Delta F) \rightarrow 0$ as $\theta \rightarrow 0$. 

\bigskip

Thus we may choose $\theta > 0$ to ensure that 
\begin{equation*}
    Vol_g(G^t_\theta \Delta F) \leq \delta.
\end{equation*}
We note that we may always ensure that the various radii $r_a,\tilde{r}_a,\tilde{r}_b$ and $r_b$ as chosen above are done so to ensure that
\begin{equation*}
    \mathrm{Per}_g(F;\partial (B_{\tilde{r}_b}(p) \setminus \phi^{-1}(T))) = 0,
\end{equation*} 
so that $\mathrm{Per}_g(L^t_\theta;B_{\tilde{r}_b}(p) \setminus \phi^{-1}(T)) \rightarrow \mathrm{Per}_g(F;B_{\tilde{r}_b}(p) \setminus \phi^{-1}(T))$ as $\theta \rightarrow 0$; in particular for $\theta > 0$ sufficiently small we ensure that 
\begin{equation*}
    |\mathrm{Per}_g(L^t_\theta;B_{\tilde{r}_b}(p) \setminus \phi^{-1}(T)) - \mathrm{Per}_g(F;B_{\tilde{r}_b}(p) \setminus \phi^{-1}(T))| \leq \frac{\delta}{2}.
\end{equation*}
As we have $G^t_\theta = E$ outside of $B_{\tilde{r}_b} \cup \phi^{-1}(T)$ we have that 
\begin{equation*}
    \mathrm{Per}_g(G^t_\theta; N \setminus (B_{\tilde{r}_b} \cup \phi^{-1}(T))) = \mathrm{Per}_g(E; N \setminus (B_{\tilde{r}_b} \cup \phi^{-1}(T))).
\end{equation*} 
Note that as $|\nabla \eta| < C_\eta$ and $b^t_\theta \rightarrow 0$ uniformly in the $C^1$ norm we may choose $\theta > 0$ potentially smaller to ensure that
\begin{equation*}
    \sup_{y \in M \cap (B_{r_b}(p) \setminus B_{r_a}(p))} |J_{\phi^{-1}\circ (Id + (\eta b^t_\theta) )\circ \phi}(y) - 1| \leq \frac{\delta}{2 \mathcal{H}^n(M \cap (B_{r_b}(p) \setminus B_{r_a}(p)))},
\end{equation*}
where here $J_f$ denotes the Jacobian of a function $f$ (which depends only on the $C^1$ norm of the function $f$), and so by the area formula we have
\begin{align*}
    |\mathcal{H}^n(\partial H^t_\theta \cap \phi^{-1}(T)) - \mathcal{H}^n(M \cap \phi^{-1}(T))| &= \left|\int_{M \cap (B_{r_b}(p) \setminus B_{r_a}(p))} (J_{\phi^{-1}\circ (Id + (\eta b^t_\theta) )\circ \phi}(y) - 1) d\mathcal{H}^n(y) \right| \leq \frac{\delta}{2}.
\end{align*}

\bigskip

We combine the above facts to compute that for $\theta > 0$ sufficiently small we have
\begin{align*}
    |\mathrm{Per}_g(G^t_\theta) - \mathrm{Per}_g(F)| &= |\mathrm{Per}_g(G^t_\theta;\phi^{-1}(T)) - \mathrm{Per}_g(F;\phi^{-1}(T)) \\
    &\quad\quad+ \mathrm{Per}_g(G^t_\theta;N \setminus \phi^{-1}(T)) - \mathrm{Per}_g(F;N \setminus \phi^{-1}(T))|\\
    &\leq |\mathcal{H}^n(\partial H^t_\theta \cap \phi^{-1}(T)) - \mathcal{H}^n(M \cap \phi^{-1}(T))|\\
    &\quad\quad+ |\mathrm{Per}_g(L^t_\theta;B_{\tilde{r}_b}(p) \setminus \phi^{-1}(T)) - \mathrm{Per}_g(F;B_{\tilde{r}_b}(p) \setminus \phi^{-1}(T))|\\
    &\leq \delta
\end{align*}
and so
\begin{equation*}
    |\mathcal{F}_\lambda(G_\theta^t) - \mathcal{F}_\lambda(F)| 
    \leq (1 + \lambda)\delta.
\end{equation*}
Setting $F_\delta = G^t_\theta$ for one choice of almost any $t \in (0,1)$ and a choice of $\theta > 0$ sufficiently small then provides a Caccioppoli satisfying the desired conclusions.
\end{proof}

\begin{rem}\label{rem: local smoothings satisfy geodesic touching lemma}
As $\partial F_\delta$ is smooth in $B_{r_2}(p)$ and agrees with $M$ outside of $B_{\tilde{r}_2}(p)$, we have that $\partial F_\delta$ satisfies the conclusions of Lemma \ref{lem: geodesic touching}. More specifically, for any $x \in N \setminus \partial F_\delta$, any minimising geodesic connecting $x$ to $\partial F_\delta$ has endpoint in a regular point of $\partial F_\delta$.
\end{rem}

\subsection{Local energy minimisation}\label{subsec: function minimisation}

We now prove the existence of local energy minimisers that agree with our one-dimensional solution, $v_\varepsilon= d^\pm_{\overline{M}} \circ \overline{\mathbb{H}}^\varepsilon$, outside of a fixed ball in $N$. Minimisers of such problems will be used in the next subsection to conclude local statements about constant mean curvature hypersurfaces under appropriate assumptions on the behaviour of the energy of $v_\varepsilon$ as $\varepsilon \rightarrow 0$.

\begin{lem}\label{lem: function minimisation}
Let $B_{\rho}(q) \subset N$ be a ball, of radius $\rho > 0$, centred at a point $q \in \overline{M}$ and define for $\varepsilon \in (0,1)$
\begin{equation*}
    \mathcal{A}_{\varepsilon,\rho}(q) = \{ u \in W^{1,2}(N) \: | \: |u| \leq 1, u = v_\varepsilon \text{ on } N \setminus B_{\rho}(q) \},
\end{equation*}
then there exists $g_\varepsilon \in \mathcal{A}_{\varepsilon,\rho}(q)$ such that \begin{equation*}
    \mathcal{F}_{\varepsilon,\lambda}(g_\varepsilon) = \inf_{u \in \mathcal{A}_{\varepsilon,\rho}(q)} \mathcal{F}_{\varepsilon,\lambda}(u).
\end{equation*}
\end{lem}
\begin{proof}
We use the direct method of the calculus of variations. Let $I_\varepsilon = \inf_{u \in \mathcal{A}_{\varepsilon,\rho}(q)} \mathcal{F}_{\varepsilon,\lambda} (u)$ and consider $\{ u_k \}_{k = 1}^\infty \in \mathcal{A}_{\varepsilon,\rho}(q)$ such that $\mathcal{F}_{\varepsilon,\lambda}(u_k) \rightarrow I_\varepsilon$. Note that for any $u \in \mathcal{A}_{\varepsilon,\rho}(q)$ we have 
\begin{equation*}
\mathcal{F}_{\varepsilon,\lambda}(u) = \frac{1}{2\sigma} \int_N \varepsilon\frac{|\nabla u|^2}{2} + \frac{W(u)}{\varepsilon} - \frac{\lambda}{2}\int_N u \geq \frac{1}{2\sigma} \int_N \varepsilon\frac{|\nabla u|^2}{2} - \frac{\lambda}{2} \mathrm{Vol}(N),
\end{equation*}
using the fact that $|u| \leq 1$ $\mathcal{H}^{n+1}$-a.e. and $W \geq 0$; hence for $k$ sufficiently large we have that
\begin{equation*}
\int_N |\nabla u_k|^2 \leq \frac{4\sigma}{\varepsilon}
\left(I_\varepsilon + 1 + \frac{\lambda}{2}\mathrm{Vol}(N)\right),
\end{equation*}
so $\sup_k ||\nabla u_k||_{L^2(N)} < \infty$. Using the triangle and Poincar\'e inequality (as $u_k - v_\varepsilon \in W^{1,2}_0(B_{\rho}(q))$ for each $k \in \mathbb{N}$) we have
\begin{align*}
||u_k||_{L^2(N)} &= || u_k - v_\varepsilon + v_\varepsilon ||_{L^2(N)} \leq || u_k - v_\varepsilon ||_{L^2(N)} + ||v_\varepsilon||_{L^2(N)} \leq C_P|| \nabla u_k - \nabla v_\varepsilon ||_{L^2(N)} + ||v_\varepsilon||_{L^2(N)}\\
&\leq C_P (\sup_k ||\nabla u_k||_{L^2(N)} + ||\nabla v_\varepsilon ||_{L^2(N)}) + ||v_\varepsilon||_{L^2(N)},
\end{align*}
(where $C_P$ is the Poincar\'e constant for $N$) so $\sup_k ||u_k||_{L^2(N)} < \infty$. 
Thus $\{ u_k \}_{k = 1}^\infty$ is a bounded sequence in $W^{1,2}(N)$ and hence, by Rellich-Kondrachov compactness, there exists $u \in W^{1,2}(N)$ and a sub-sequence of the $\{u_k\}_{k = 1}^\infty$ (not relabelled) such that $u_k$ converges to $u$, weakly in $W^{1,2}(N)$, strongly in $L^2(N)$ (hence strongly in $L^1(N)$) and $\mathcal{H}^{n+1}\text{-}a.e. \text{ in } N$. We then have that $|u| \leq 1$ and $u = v_\varepsilon$ on $N \setminus B_{\rho}(q)$ at $\mathcal{H}^{n+1}$-$a.e.$ point, thus $u \in \mathcal{A}_{\varepsilon,\rho}(q)$. Combining the weak convergence of the $u_k$ in $W^{1,2}(N)$, almost everywhere convergence of the $u_k$ with Fatou's lemma (using the continuity of $W$ and that $W \geq 0$) and the strong $L^1(N)$ convergence of the $u_k$, we have $\mathcal{F}_{\varepsilon,\lambda}(u) = I_\varepsilon$. 

\bigskip

The above arguments show that, for sufficiently small $\varepsilon > 0$, the energy minimisation problem in the class $\mathcal{A}_{\varepsilon,\rho}(q)$ is well posed. Hence, for each  $\varepsilon \in (0,1)$ we may produce a function $g_\varepsilon \in W^{1,2}(N)$ such that $|g_\varepsilon| \leq 1$ on $N$, $g_\varepsilon = v_\varepsilon$ on
$N \setminus B_{\rho}(q)$ and $\mathcal{F}_{\varepsilon,\lambda}(g_\varepsilon) = \inf_{u \in \mathcal{A_\varepsilon}}\mathcal{F}_{\varepsilon,\lambda} (u)$. \end{proof}

\subsection{Recovery functions for local geometric properties}\label{subsec: recovery functions}

Using Proposition \ref{prop: local smoothing} and our analysis from Subsection \ref{subsec: one dimensional profile} we now prove the following recovery type Lemma relating the energy of $v_\varepsilon$ to local geometric properties of constant mean curvature hypersurfaces. Recall that for an isolated singularity $p \in \overline{M}$ we fixed $0 < r_1 < r_2 < \min\{R_p,R_l\}$ such that $M \cap \overline{B_{r_2}(p)} \setminus B_{r_1}(p)$ is smooth.

\begin{lem}\label{lem: recovery functions}
Let $F \in \mathcal{C}(N)$ be such that $F \setminus B_{r_1}(p) = E \setminus B_{r_1}(p)$ and
\begin{equation*}
    \mathcal{F}_\lambda(F) = \inf_{G \in \mathcal{C}(N)} \{ \mathcal{F}_\lambda(G) \: | \: G \setminus B_{r_1}(p) = E \setminus B_{r_1}(p) \}.
\end{equation*}
Then, for $\varepsilon > 0$ sufficiently small, there exist functions $f_\varepsilon \in \mathcal{A}_{\varepsilon,r_2}(p)$, so that $|f_\varepsilon| \leq 1$ and $f_\varepsilon = v_\varepsilon$ on $N \setminus B_{r_2}(p)$, such that 
\begin{equation*}
    \mathcal{F}_{\varepsilon,\lambda}(f_\varepsilon) \rightarrow \mathcal{F}_\lambda(F).
\end{equation*}
Furthermore, if the functions $v_\varepsilon$ are such that $\mathcal{F}_{\varepsilon,\lambda}(v_\varepsilon) \leq \mathcal{F}_{\varepsilon,\lambda}(g_\varepsilon) + \tau_\varepsilon$, for some sequence $\tau_\varepsilon \rightarrow 0$ as $\varepsilon \rightarrow 0$, where the $g_\varepsilon$ are defined as in Lemma \ref{lem: function minimisation} for $B_{r_2}(p)$, then $E$ solves the above minimisation problem.
\end{lem}

\begin{rem}
There exist minimisers to the variational problem in Lemma \ref{lem: recovery functions} by \cite[Section 12.5]{maggi_2012}.
\end{rem}

\begin{proof}
Let $F$ solve the minimisation problem and consider, for each $\varepsilon, \delta > 0$, the functions $f_{\varepsilon,\delta} = \overline{\mathbb{H}}^\varepsilon \circ d^\pm_{\partial F_\delta}$ where $F_\delta$ is a local smoothing of $F$ as in Proposition \ref{prop: local smoothing} and we define the signed distance function to $\partial F_\delta$ by \begin{equation*}
    d^\pm_{\partial F_\delta} = 
    \begin{cases}
    + d_{\partial F_\delta} (x), \text{ if } x \in F_\delta\\
    0, \text{ if } x \in \partial F_\delta\\
    - d_{\partial F_\delta} (x), \text{ if } x \in N \setminus F_\delta
    \end{cases},
\end{equation*}
where here $d_{\partial F_\delta}$ is the usual distance function to the set $\partial F_\delta$. Note that the functions $ f_{\varepsilon,\delta} \in \mathcal{A}_{\varepsilon,r_2}(p)$ for $\varepsilon > 0$ small enough so that $2\varepsilon \Lambda_\varepsilon < r_2 - \tilde{r}_2$, where $\tilde{r}_2$ is chosen as in Proposition \ref{prop: local smoothing}. To see this we note that as $F_\delta= E$ outside of $B_{\tilde{r}_2}(p)$ we have that $d_{\partial F_\delta} = d_{\overline{M}}$ on the set $(N \setminus B_{r_2}(p)) \cap \{d_{\partial F_\delta} \leq 2\varepsilon\Lambda_\varepsilon\}$ and hence $f_{\varepsilon,\delta} = v_\varepsilon$ on  $N \setminus B_{r_2}(p)$.

\bigskip

By Remark \ref{rem: local smoothings satisfy geodesic touching lemma} the hypersurfaces $\partial F_\delta$ satisfy the conclusions of Lemma \ref{lem: geodesic touching}. Hence we may repeat the analysis for $\partial F_\delta$ (of the cut locus, level sets etc.) as we did for $M$ in Subsection \ref{subsec: one dimensional profile} in order to conclude that $\mathcal{F}_{\varepsilon,\lambda}(f_{\varepsilon,\delta}) \rightarrow \mathcal{F}_\lambda(F_\delta)$ as $\varepsilon \rightarrow 0$. By Proposition \ref{prop: local smoothing} we have that $|\mathcal{F}_\lambda(F_\delta) - \mathcal{F}_\lambda(F)| \leq (1+ \lambda)\delta$ and thus by setting $f_\varepsilon = f_{\varepsilon,\frac{\varepsilon}{1 + \lambda}}$ we ensure that this sub-sequence is such that $\mathcal{F}_{\varepsilon,\lambda}(f_\varepsilon) \rightarrow \mathcal{F}_\lambda(F)$ as desired.

\bigskip

Under the assumption that the functions $v_\varepsilon$ are such that $\mathcal{F}_{\varepsilon,\lambda}(v_\varepsilon) \leq \mathcal{F}_{\varepsilon,\lambda}(g_\varepsilon) + \tau_\varepsilon$, for some sequence $\tau_\varepsilon \rightarrow 0$ as $\varepsilon \rightarrow 0$, and the fact that $f_\varepsilon \in \mathcal{A}_{\varepsilon,r_2}(p)$ we see that
\begin{equation*}
    \mathcal{F}_{\varepsilon,\lambda}(v_\varepsilon) \leq \mathcal{F}_{\varepsilon,\lambda}(g_\varepsilon) + \tau_\varepsilon \leq \mathcal{F}_{\varepsilon,\lambda}(f_\varepsilon) + \tau_\varepsilon,
\end{equation*}
where here we are using that $\mathcal{F}_{\varepsilon,\lambda}(g_\varepsilon) = \inf_{u \in \mathcal{A}_{\varepsilon,r_2}(p)} \mathcal{F}_{\varepsilon,\lambda}(u) \leq  \mathcal{F}_{\varepsilon,\lambda}(f_\varepsilon)$ by construction. Hence, by (\ref{eqn: 1d profile approx cmc energy}) and the fact that $\mathcal{F}_{\varepsilon,\lambda}(f_\varepsilon) \rightarrow \mathcal{F}_\lambda(F)$ from the above, letting $\varepsilon \rightarrow 0$ we conclude that
\begin{equation*}
    \mathcal{F}_\lambda(E) \leq \mathcal{F}_\lambda(F).
\end{equation*}
In particular we have that
\begin{equation*}
    \mathcal{F}_\lambda(E) = \inf_{G \in \mathcal{C}(N)} \{ \mathcal{F}_\lambda(F) \: | \: G \setminus B_{r_1}(p) = E \setminus B_{r_1}(p) \},
\end{equation*}
as desired.
\end{proof}

\begin{rem}\label{rem: almost min area min cones}
Lemma \ref{lem: recovery functions} combined with the results of \cite[Section 28.2]{maggi_2012} imply that if the functions $v_\varepsilon$ are such that $\mathcal{F}_{\varepsilon,\lambda}(v_\varepsilon) \leq \mathcal{F}_{\varepsilon,\lambda}(g_\varepsilon) + \tau_\varepsilon$, for some sequence $\tau_\varepsilon \rightarrow 0$ as $\varepsilon \rightarrow 0$, then any tangent cone to the hypersurface $M$ at a point of $B_{r_1}(p) \cap \partial E$ is area-minimising (in the sense that the cone is a perimeter minimiser in $\mathbb{R}^{n+1}$).
\end{rem}

\section{Construction of paths}\label{sec: paths}

In order to prove Theorem \ref{thm: local minimisation in positive ricci} we will work under the assumption that the functions $v_\varepsilon = \overline{\mathbb{H}}^\varepsilon \circ d^\pm_{\overline{M}}$, defined in Subsection \ref{subsec: one dimensional profile}, are such that $\mathcal{F}_{\varepsilon,\lambda}(v_\varepsilon) \geq \mathcal{F}_{\varepsilon,\lambda}(g_\varepsilon) + \tau$ for some $\tau > 0$, where the $g_\varepsilon$ are defined as in Lemma \ref{lem: function minimisation} in a ball centred on an isolated singularity of $M$. We then exhibit a continuous path in $W^{1,2}(N)$, from $a_\varepsilon$ to $b_\varepsilon$, with energy along the path bounded a fixed amount below $\mathcal{F}_\lambda(E)$, independent of $\varepsilon$. In this manner we will have exhibited, see Section \ref{sec: conclusion of the proof}, an admissible path in the min-max procedure of \cite{bellettini-wickramasekera} which contradicts the assumption that $M$ arose from this construction and allows us to establish Theorem \ref{thm: local minimisation in positive ricci}. In the following Subsections we construct the separate pieces of our desired path, with maximum energy along these pieces bounded a fixed amount below $\mathcal{F}_\lambda(E)$, independently of $\varepsilon$.

\subsection{Paths with energy drop fixing a ball}
\label{subsec: shifted paths}

We first construct the shifted functions described in Step 2 of Subsection \ref{subsec: strategy} and compute upper energy bounds.

\subsubsection{Choosing radii for local properties}\label{subsec: choosing radii}

We fix various radii, dependent only on the geometry of $M$ around an isolated singularity, in order to later define functions for our path.

\bigskip

Let $p \in \mathrm{Sing}(M)$ be an isolated singularity of $M$, there then exists an $R_p > 0$ such that $\overline{B_{R_p}(p)} \cap \mathrm{Sing}(M) = \{p\}$, i.e.~such that $M \cap \overline{B_{R_p}(p)} \setminus \{p\}$ is smooth. 

\bigskip

The bounded mean curvature of $M$ (specifically as the monotonicity formula holds at all points of $\overline{M}$) provides Euclidean volume growth, namely that there exists two constants $C_p > 0$ and $r_p \in (0,R_p)$ both depending on the point $p$ such that for all $r < r_p$ we have
\begin{equation}\label{eqn: monotonicity gives euclidean volume growth for M around p}
   \mathcal{H}^n(M \cap B_{r_p}(p)) \leq C_pr^n.
\end{equation}
We also fix $r_0 \in (0, \frac{r_p}{4})$ such that 
\begin{equation}\label{eqn: fixed amount of M outside a ball}
    \mathcal{H}^n(M \setminus B_{r_0}(p)) > \frac{3\mathcal{H}^n(M)}{4}.
\end{equation}
Set $\widehat{K} = M \cap (\overline{B_{r_p}(p)} \setminus B_{r_0}(p))$, which is compact in $M$, so by Remark \ref{rem: two-sided neighbourhoods of compact subsets of M} we ensure that there exists a $c_{\widehat{K}} > 0$ such that $F(\widehat{K} \times (-c_{\widehat{K}}, c_{\widehat{K}})) \cap (\mathrm{Sing}(M) \cup \overline{S_{d_{\overline{M}}}}) = \emptyset$. Defining on $K \times (-c_{\widehat{K}},c_{\widehat{K}}) \subset V_M$ the function
\begin{equation}\label{eqn: definition of h}
    h(x,a) = \mathrm{dist}_N(F(\Pi_{V_M}(x,a)),p),
\end{equation}
so that $h$ is the distance to $p$, we ensure that for $(x,a) \in K \times (-c_{\widehat{K}},c_{\widehat{K}}) \subset V_M$ we have, by Remark \ref{rem: projection definition}, that
\begin{equation}\label{eqn: gradient bound for h}
    |\nabla h(x,a)|_{(x,a)} \leq C_h,
\end{equation}
where here we are computing the gradient, $\nabla$, on $V_M$ with respect to the pullback metric $F^*g$, for some constant $C_h > 0$ dependent only on $\widehat{K}$, $N$ and $g$.

\bigskip

We may now fix $0 < r < \frac{3}{4}r_0$ sufficiently small to ensure that on the annulus $\mathcal{A} = M \cap (\overline{B}_{r_0 + 3r}(p) \setminus B_{r_0 +2r})$ we have
\begin{equation}\label{eqn: annulus ratio for energy drop}
    \frac{\mathcal{H}^n(\mathcal{A})}{r^2} \leq \left(\frac{me^{-\frac{\lambda^2}{2m}}\mathcal{H}^n(M)}{8C^2_h} \right);
\end{equation}
where here we are using the Euclidean volume growth of the mass of $M$ (which follows from the monotonicity formula) to ensure that $\mathcal{H}^n(\mathcal{A})$ is of order $r^n$, combined with the assumption that, from Remark \ref{rem: singular dimensions}, $n \geq 7$. The above choice of $r$ will become clear when calculating the energy of the functions $v^{t,s}_\varepsilon$ in Subsection \ref{subsec: energy calculations for shifted functions}.

\bigskip

We now define the balls $B_i = B_{r_0 + ir}(p)$ for $i \in \{ 1,2,3,4\}$ which are such that $B_1 \subset \subset B_2 \subset \subset B_3 \subset \subset B_4$; with this notation we have that $\mathcal{A} = M \cap (\overline{B_3} \setminus B_2)$.

\subsubsection{Defining the shifted functions}\label{subsec: defining shifted functions}

We define, for $t \in [-t_1,t_1]$ where $t_1 > 0$ is to be chosen and all $\varepsilon > 0$ sufficiently small, functions $v^{t,s}_\varepsilon \in W^{1,2}(N)$ for $s \in [0,1]$ with the paths $t \in [-t_1,t_1] \rightarrow v^{t,s}_\varepsilon$ and $s \in [0,1] \rightarrow v^{t,s}_\varepsilon$ continuous in $W^{1,2}(N)$. The functions $v^{t,s}_\varepsilon$ will be defined so that the following properties hold for all $t \in [-t_1,t_1]$ and $s \in [0,1]$
\begin{equation}\label{eqn: properties of shifting functions}
    \begin{cases}
     v^{t,1}_\varepsilon = \overline{\mathbb{H}}^\varepsilon \circ ( d^\pm_{\overline{M}} - t) \text{ on } N\\
     v^{0,s}_\varepsilon = v_\varepsilon \text{ on } N\\
     v^{t,0}_\varepsilon = v_\varepsilon \text{ in } B_{r_0}(p) \subset B_1
    \end{cases},
\end{equation}
and in such a way that, for some $E(\varepsilon) \rightarrow 0$ as $\varepsilon \rightarrow 0$, we have
\begin{equation*}
    \mathcal{F}_{\varepsilon,\lambda}(v^{t,s}_\varepsilon) \leq \mathcal{F}_{\varepsilon,\lambda}(v_\varepsilon) + E(\varepsilon).
\end{equation*}
Furthermore, we will show in Lemma \ref{lem: shifting path upper energy bounds} that there exists some $0 < t_0 < t_1$ and $\eta > 0$ such that for all $s \in [0,1]$ we have
\begin{equation*}
\mathcal{F}_{\varepsilon,\lambda}(v^{\pm t_0,s}_\varepsilon) \leq \mathcal{F}_{\varepsilon,\lambda}(v_\varepsilon) - \eta + E(\varepsilon).    
\end{equation*}
Finally, by Lemma \ref{lem: sliding path upper energy bounds}, for $t > t_0$ we will have
\begin{equation*}
   \mathcal{F}_{\varepsilon,\lambda}(v^{\pm t}_\varepsilon) \leq \mathcal{F}_{\varepsilon,\lambda}(v^{\pm t_0,1}_\varepsilon) \leq \mathcal{F}_{\varepsilon,\lambda}(v_\varepsilon) - \eta + E(\varepsilon).
\end{equation*}
These upper bounds on the energy along the paths provided by the functions $v^{t,s}_\varepsilon$ will be calculated explicitly in Subsection \ref{subsec: energy calculations for shifted functions}.

\bigskip

Thus, assuming the above, and combined with the path provided in Subsection \ref{subsec: local path}, which changes the functions $v^{t,0}_\varepsilon$ only in $B_1$ to decrease their energy, we construct a path from $a_\varepsilon$ to $b_\varepsilon$ in $W^{1,2}(N)$ whose energy along the path remains bounded strictly below $\mathcal{F}_\lambda(E)$. We now proceed to define the $v^{t,s}_\varepsilon$ explicitly.

\bigskip

Consider $B_1 \subset \subset B_2 \subset \subset B_3 \subset \subset B_4$ centred at $p \in M$ as specified in Subsection \ref{subsec: choosing radii}. Let $K = M \cap \overline{B_4} \setminus B_1$, which is compact in $M$, and fix $c_K = c_{M \cap \overline{B_4} \setminus B_1} > 0$ as in Remark \ref{rem: two-sided neighbourhoods of compact subsets of M}; we then have that $K \times (-c_K, c_K) \subset V_M$ is a two-sided tubular neighbourhood of $M \cap \overline{B_4} \setminus B_1$ in $N$ under the map $F$. Note that we have $F\left(K \times (-c_K,c_K)\right) \cap (\mathrm{Sing}(M) \cup\overline{S_{d_{\overline{M}}}}) = \emptyset$ by the choice of $c_K$ and as $K \subset \widehat{K}$ we ensure that $c_K \leq c_{\widehat{K}}$ as in Subsection \ref{subsec: choosing radii}.

\bigskip

For $r > 0$ as chosen we define the following function on $M$,
\begin{equation}\label{eqn: definition of f}
    f(y) = \begin{cases}
    1 \text{ if } y \in M \setminus \overline{B_3}\\
    0 \text{ if } y \in M \cap \overline{B_2}\\
    \frac{1}{r}(\mathrm{dist}_N(y,p) - 2r -r_0) \text{ if } y \in \mathcal{A}.
    \end{cases}
\end{equation}
Then we have $0 \leq f \leq 1$ and set
\begin{equation*}
    f_s(x) = s + (1-s)f(x).
\end{equation*}
Recall, from Remark \ref{rem: projection definition}, the definition of the smooth projection, $\Pi$ to $M$ defined on $N \setminus (\mathrm{Sing}(M) \cup \overline{S_{d_{\overline{M}}}})$. Fix a choice of $t_1 > 0$ and $\tilde{\varepsilon} > 0$ sufficiently small to ensure that $2\varepsilon\Lambda_{\varepsilon} + t_1 < \min\{c_K, \frac{r}{2}\}$, for all $0 < \varepsilon < \tilde{\varepsilon}$. We now define, for $t \in [-t_1,t_1]$, $s \in [0,1]$ and $0 < \varepsilon < \tilde{\varepsilon}$, the functions
\begin{equation*}
    v^{t,s}_\varepsilon(x) = \begin{cases}
    \overline{\mathbb{H}}^\varepsilon(d^{\pm}_{\overline{M}}(x) - t f_s(\Pi(x)))) \text{ if } x \in F(K \times [-c_K,c_K])\\
    \overline{\mathbb{H}}^\varepsilon(d^{\pm}_{\overline{M}}(x) - ts) \text{ if } x \in B_{r_0 + \frac{3r}{2}}(p)\\
    \overline{\mathbb{H}}^\varepsilon(d^{\pm}_{\overline{M}}(x) - t)\text{ if } x \in N \setminus B_{r_0 + \frac{7r}{2}}(p)\\
    1 \text{ if } x \in E \cap \{d^\pm_{\overline{M}} > 2\varepsilon\Lambda_\varepsilon + t_1\}\\
    -1 \text{ if } x \in (N \setminus E) \cap \{d^\pm_{\overline{M}} <  - 2\varepsilon\Lambda_\varepsilon - t_1\}\\
    \end{cases}.
\end{equation*}
We now show that the $v^{t,s}_\varepsilon$ are well defined Lipschitz functions on $N$ so that in particular, as $F(\mathcal{A} \times [-2\tilde{\varepsilon}\Lambda_{\tilde{\varepsilon}} - t_1,2\tilde{\varepsilon}\Lambda_{\tilde{\varepsilon}} + t_1]) \subset \subset \overline{B_{r_0 + \frac{7r}{2}}(p)} \setminus B_{r_0 + \frac{3r}{2}}(p)$ for $2\tilde{\varepsilon}
\Lambda_{\tilde{\varepsilon}} + t_1 < \frac{r}{2}$, we have
\begin{equation*}
    v^{t,s}_\varepsilon(x) = \overline{\mathbb{H}}^\varepsilon(d^\pm_{\overline{M}}(x) - tf_s(\Pi(x))) \text{ on } N \setminus (\mathrm{Sing}(M) \cup \overline{S_{d_{\overline{M}}}}).
\end{equation*}
Note that if $x \in \overline{B_{r_0 + \frac{7r}{2}}(p)} \setminus B_{r_0 + \frac{3r}{2}}(p)$ and $|d^\pm_{\overline{M}}| \leq 2\varepsilon\Lambda_\varepsilon + t_1$ then by the triangle inequality we see that $x \in F(K \times [-c_K,c_K])$. In $F(K \times [2\tilde{\varepsilon}\Lambda_{\tilde{\varepsilon}} + t_1, c_K])$ we have that $v^{t,s}_\varepsilon \equiv 1$ and in $F(K \times [-c_K, -(2\tilde{\varepsilon}\Lambda_{\tilde{\varepsilon}} + t_1)])$ we have that $v^{t,s}_\varepsilon \equiv -1$. Combining these two facts we see that the $v^{t,s}_\varepsilon$ are in fact Lipschitz functions in $\overline{B_{r_0 + \frac{7r}{2}}(p)} \setminus B_{r_0 + \frac{3r}{2}}(p)$.

\bigskip

Noting that $F(\mathcal{A} \times [-2\tilde{\varepsilon}\Lambda_{\tilde{\varepsilon}} - t_1,2\tilde{\varepsilon}\Lambda_{\tilde{\varepsilon}} + t_1]) \subset \subset \overline{B_{r_0 + \frac{7r}{2}}(p)} \setminus B_{r_0 + \frac{3r}{2}}(p)$ as $2\tilde{\varepsilon}
\Lambda_{\tilde{\varepsilon}} + t_1 < \frac{r}{2}$, we have by definition of $f$ that if $x \in B_{r_0 + \frac{3r}{2}}(p) \cap F(K \times [-c_K,c_K])$ then $v^{t,s}_\varepsilon = \overline{\mathbb{H}}^\varepsilon(d^\pm_{\overline{M}}(x) - ts)$, and if $x \in (N \setminus B_{r_0 + \frac{7r}{2}}(p)) \cap F(K \times [-c_K,c_K])$ then $v^{t,s}_\varepsilon = \overline{\mathbb{H}}^\varepsilon(d^\pm_{\overline{M}}(x) - t)$. 

\bigskip

As noted in the above paragraphs, the functions  $v^{t,s}_\varepsilon$ are Lipschitz in $\overline{B_{r_0 + \frac{7r}{2}}(p)} \setminus B_{r_0 + \frac{3r}{2}}(p)$. Furthermore, in the sets $N \setminus B_{r_0 + \frac{7r}{2}}(p)$ and $B_{r_0 + \frac{3r}{2}}(p)$ the $v^{t,s}_\varepsilon$ are Lipschitz functions by definition (as on these sets $v^{t,s}_\varepsilon = \overline{\mathbb{H}}^\varepsilon(d^\pm_{\overline{M}}(x) - t)$ and $v^{t,s}_\varepsilon = \overline{\mathbb{H}}^\varepsilon(d^\pm_{\overline{M}}(x) - ts)$ respectively). In conclusion we have that the $v^{t,s}_\varepsilon$ are well defined continuous functions on $N$ that are Lipschitz on each of the three sets $\overline{B_{r_0 + \frac{7r}{2}}(p)} \setminus B_{r_0 + \frac{3r}{2}}(p)$, $N \setminus B_{r_0 + \frac{7r}{2}}(p)$, and $B_{r_0 + \frac{3r}{2}}(p)$; so that the $v_\varepsilon^{t,s}$ are indeed well defined Lipschitz functions on $N$. In particular we have that $v^{t,s}_\varepsilon \in W^{1,2}(N)$ for $t \in [-t_1,t_1]$, $s \in [0,1]$ and $0 < \varepsilon < \tilde{\varepsilon}$.

\subsubsection{Energy calculations and continuity of paths}\label{subsec: energy calculations for shifted functions}

We calculate an upper bound on the energy of the functions $v^{t,s}_\varepsilon$ for $t \in [t_0,t_0]$, for some $0 < t_0 \leq t_1$, all $s \in [0,1]$ and $\varepsilon > 0$ sufficiently small. The calculation method here is similar to those in \cite[Sections 4, 6.1 and 7.1]{Bellettini-Workman}.

\begin{lem}\label{lem: shifting path upper energy bounds}
There exists $t_0 \in (0, t_1)$ such that for all $s \in [0,1]$ and $t \in [-t_0,t_0]$ we have that
\begin{equation*}
    \mathcal{F}_{\varepsilon,\lambda}(v^{t,s}_\varepsilon) \leq \mathcal{F}_{\varepsilon,\lambda}(v_\varepsilon) + E(\varepsilon)
\end{equation*}
where $E(\varepsilon) \rightarrow 0$ as $\varepsilon \rightarrow 0$. Furthermore, we have that, for all $\varepsilon > 0$ sufficiently small there exists
\begin{equation*}
    \eta = \frac{m}{8}\mathcal{H}^n(M)t_0^2> 0
\end{equation*} such that for all $s \in [0,1]$ we have
\begin{equation*}
    \mathcal{F}_{\varepsilon,\lambda}(v^{\pm t_0,s}_\varepsilon) \leq \mathcal{F}_{\varepsilon,\lambda}(v_\varepsilon) - \eta.
\end{equation*}
\end{lem}
\begin{proof}
We define, for $(x,a) \in V_M$, $\hat{v}^{t,s}_\varepsilon ((x,a)) = v^{t,s}_\varepsilon (F(x,a))$ so that, as $\Pi \circ F = F \circ \Pi_{V_M}$, we have that
\begin{equation}\label{eqn: shifting functions on cylinder}
    \hat{v}^{t,s}_\varepsilon(x,a) = \overline{\mathbb{H}}^\varepsilon(a - tf_s(F(\Pi_{V_M}(x,a)))) \text{ on } V_M.
\end{equation}
Note that by (\ref{eqn: definition of h}) and (\ref{eqn: definition of f}) we have for points $(x,a) \in V_M$ with $x \in \mathcal{A}$ that 
\begin{equation*}
   \nabla \big(f_s(F(\Pi_{V_M}(x,a)))\big) = \frac{1-s}{r}\nabla h(x,a),
\end{equation*} 
and if $x \in M \setminus \mathcal{A}$ then $f$ is constant and so 
\begin{equation*}
    \nabla \big(f_s(F(\Pi_{V_M}(x,a)))\big) = 0.
\end{equation*}
We then compute by the generalised Gauss' Lemma, \cite[Chapter 2.4]{tubes}, that we have
\begin{equation}\label{eqn: Gauss lemma on shifting functions}
    |\nabla \hat{v}^{t,s}_\varepsilon(x,a)|_{(x,a)}^2 = \left((\overline{\mathbb{H}}^\varepsilon)'(a - tf_s(x)\right)^2 \left(1 + \frac{(1-s)^2t^2}{r^2}|\nabla h(x,a)|_{(x,a)}^2 \chi_{\mathcal{A} \times \mathbb{R}}(x,a)\right),
\end{equation}
where here $\chi_{\mathcal{A} \times
\mathbb{R}}$ is the indicator function for the set $\mathcal{A} \times
\mathbb{R}$ on $V_M$.

\bigskip

Using the co-area formula (slicing with $a$) and Fubini's Theorem we have
\begin{align*}
    2\sigma\mathcal{F}_{\varepsilon,\lambda}(v^{t,s}_\varepsilon) &= \int_{V_M} e_\varepsilon(\hat{v}^{t,s}_\varepsilon) - \sigma\lambda\hat{v}^{t,s}_\varepsilon d\mathcal{H}^{n+1}_{F^*g}\\
    &= \int_{V_M} \frac{\varepsilon}{2}|\nabla \hat{v}^{t,s}_\varepsilon|^2 + \frac{W(\hat{v}^{t,s}_\varepsilon)}{\varepsilon} - \sigma\lambda\hat{v}^{t,s}_\varepsilon d\mathcal{H}^{n+1}_{F^*g}(x,a)\\
    &= \int_{V_M} \frac{\varepsilon}{2}\left((\overline{\mathbb{H}}^\varepsilon)'(a - tf_s(x))\right)^2 + \frac{W(\overline{\mathbb{H}}^\varepsilon(a - tf_s(x))}{\varepsilon} - \sigma\lambda\overline{\mathbb{H}}^\varepsilon(a - tf_s(x)) d\mathcal{H}^{n+1}_{F^*g}(x,a)\\
    &\quad\quad + \int_{V_M} \frac{\varepsilon}{2}\left((\overline{\mathbb{H}}^\varepsilon)'(a - tf_s(x))\right)^2 \frac{(1-s)^2t^2}{r^2}|\nabla h(x,a)|_{(x,a)}^2  \chi_{\mathcal{A} \times \mathbb{R}}(x,a) d\mathcal{H}^{n+1}_{F^*g}(x,a)\\
    &= \int_M \int_{\sigma^-(x)}^{\sigma^+(x)} \left[e_\varepsilon\left(\overline{\mathbb{H}}^\varepsilon(a - tf_s(x))\right) - \sigma\lambda\overline{\mathbb{H}}^\varepsilon(a - tf_s(x)) \right]\theta(x,a) da\, d\mathcal{H}^n(x)\\
    &\quad\quad + \int_{\mathcal{A}} \int_{\sigma^-(x)}^{\sigma^+(x)} \frac{\varepsilon}{2}\left((\overline{\mathbb{H}}^\varepsilon)'(a - tf_s(x))\right)^2 \frac{(1-s)^2t^2}{r^2}|\nabla h(x,a)|_{(x,a)}^2 \theta(x,a) da\, d\mathcal{H}^n(x).
\end{align*}
Then, as $v_\varepsilon = v^{0,s}_\varepsilon$ for any $s \in [0,1]$, we have
\begin{align}
    \label{eqn: energy term}
    2\sigma(\mathcal{F}_{\varepsilon,\lambda}(v^{t,s}_\varepsilon) &- \mathcal{F}_{\varepsilon,\lambda}(v_\varepsilon)) =  \int_M \int_{\sigma^-(x)}^{\sigma^+(x)} \left[e_\varepsilon\left(\overline{\mathbb{H}}^\varepsilon(a - tf_s(x))\right) - e_\varepsilon\left(\overline{\mathbb{H}}^\varepsilon(a)\right) \right]\theta(x,a) da\, d\mathcal{H}^n(x)\\
    \label{eqn: 1d term}
    &\quad\quad\quad -\int_M \int_{\sigma^-(x)}^{\sigma^+(x)} \sigma\lambda\left[\overline{\mathbb{H}}^\varepsilon(a - tf_s(x)) - \overline{\mathbb{H}}^\varepsilon(a) \right]\theta(x,a) da\, d\mathcal{H}^n(x)\\
    \label{eqn: gradient term}
    &\quad + \int_{\mathcal{A}} \int_{\sigma^-(x)}^{\sigma^+(x)} \frac{\varepsilon}{2}\left((\overline{\mathbb{H}}^\varepsilon)'(a - tf_s(x))\right)^2 \frac{(1-s)^2t^2}{r^2}|\nabla h(x,a)|_{(x,a)}^2 \theta(x,a) da\, d\mathcal{H}^n(x)
\end{align}
We now analyse the three terms above separately. 

\bigskip

First, consider (\ref{eqn: 1d term}), by the Fundamental Theorem of Calculus and Fubini's Theorem we see that
\begin{align*}
    (\ref{eqn: 1d term}) &= -\int_M \int_{\sigma^-(x)}^{\sigma^+(x)} \sigma\lambda\left[\overline{\mathbb{H}}^\varepsilon(a - tf_s(x)) - \overline{\mathbb{H}}^\varepsilon(a) \right]\theta(x,a) da\, d\mathcal{H}^n(x)\\
    &= \int_0^t \int_M f_s(x) \int_{\sigma^-(x)}^{\sigma^+(x)} \sigma\lambda  (\overline{\mathbb{H}}^\varepsilon)'(a - rf_s(x))\theta(x,a) da\,d\mathcal{H}^n(x)\,dr
\end{align*}
Second, consider (\ref{eqn: energy term}), by the Fundamental Theorem of Calculus, Fubini's Theorem and integrating by parts then, by setting $\theta(x,\sigma^\pm(x)) = \lim_{a \rightarrow \sigma^\pm(x)} \theta(x,a)$ (which exists by continuity of the volume element) and recalling (\ref{eqn: differential equation for area elements}), we obtain
\begin{align*}
    (\ref{eqn: energy term}) &= \int_M \int_{\sigma^-(x)}^{\sigma^+(x)} \left[e_\varepsilon\left(\overline{\mathbb{H}}^\varepsilon(a - tf_s(x))\right) - e_\varepsilon\left(\overline{\mathbb{H}}^\varepsilon(a)\right) \right]\theta(x,a) da\, d\mathcal{H}^n(x)\\
    &= -\int_0^t\int_M f_s(x) \int_{\sigma^-(x)}^{\sigma^+(x)} e'_\varepsilon\left(\overline{\mathbb{H}}^\varepsilon(a - rf_s(x))\right)\theta(x,a) da\, d\mathcal{H}^n(x)\,dr\\
    &= \int_0^t \int_M f_s(x) \int_{\sigma^-(x)}^{\sigma^+(x)} e_\varepsilon(\overline{\mathbb{H}}^\varepsilon(a - rf_s(x)))\partial_a \theta(x,a) da \, d\mathcal{H}^n(x) \, dr\\
    &\quad\quad -\int_0^t\int_M f_s(x) e_\varepsilon\left(\overline{\mathbb{H}}^\varepsilon(\sigma^+(x) - rf_s(x))\right)\theta(x,\sigma^+(x)) da\, d\mathcal{H}^n(x)\,dr\\
    &\quad\quad\quad +\int_0^t\int_M f_s(x) e_\varepsilon\left(\overline{\mathbb{H}}^\varepsilon(\sigma^-(x) - rf_s(x))\right)\theta(x,\sigma^-(x)) da\, d\mathcal{H}^n(x)\,dr\\
    &= \int_0^t \int_M f_s(x) \int_{\sigma^-(x)}^{\sigma^+(x)} e_\varepsilon(\overline{\mathbb{H}}^\varepsilon(a - rf_s(x)))(\lambda - H(x,a))\theta(x,a) da \, d\mathcal{H}^n(x) \, dr\\
    &\quad - \int_0^t \int_M f_s(x) \int_{\sigma^-(x)}^{\sigma^+(x)} \lambda e_\varepsilon(\overline{\mathbb{H}}^\varepsilon(a - rf_s(x))) \theta(x,a) da \, d\mathcal{H}^n(x) \, dr\\
    &\quad\quad -\int_0^t\int_M f_s(x) e_\varepsilon\left(\overline{\mathbb{H}}^\varepsilon(\sigma^+(x) - rf_s(x))\right)\theta(x,\sigma^+(x)) \, d\mathcal{H}^n(x)\,dr\\
    &\quad\quad\quad +\int_0^t\int_M f_s(x) e_\varepsilon\left(\overline{\mathbb{H}}^\varepsilon(\sigma^-(x) - rf_s(x))\right)\theta(x,\sigma^-(x)) \, d\mathcal{H}^n(x)\,dr.
\end{align*}
Note that in the last equality above we added and subtracted the term 
\begin{equation*}
    \int_0^t \int_M f_s(x) \int_{\sigma^-(x)}^{\sigma^+(x)} \lambda e_\varepsilon(\overline{\mathbb{H}}^\varepsilon(a - rf_s(x)))\theta(x,a) da \, d\mathcal{H}^n(x) \, dr
\end{equation*}
in order to introduce the quantity $\lambda - H(x,a)$ to the calculation; thus we consider the term
\begin{equation}\label{eqn: underlying geometry term}
    \int_0^t \int_M f_s(x) \int_{\sigma^-(x)}^{\sigma^+(x)} e_\varepsilon(\overline{\mathbb{H}}^\varepsilon(a - rf_s(x)))(\lambda - H(x,a))\theta(x,a) da \, d\mathcal{H}^n(x) \, dr , 
\end{equation}
which we now control by the assumption of positive Ricci curvature.

\bigskip

By the continuity of the functions $\sigma^\pm$ and $\theta(x,a)$, along with (\ref{eqn: fixed amount of M outside a ball}), we may fix a compact $L \subset (M \setminus B_4)$ and some $l > 0$ sufficiently small so that $|\sigma^\pm(x)| > l$ for $x \in L$,
\begin{equation}\label{eqn: fixed area for energy drop}
    \mathcal{H}^n(\{ (x,l) \in V_M \: | \: x \in L\}) > \frac{\mathcal{H}^n(M)}{2},
\end{equation} 
and so that for some $t_0 \leq \min \{t_1,2d(N)\}$ we ensure $\theta(x,a) \geq \theta(x,l)$ for all $a \in [-t_0,t_0]$ (this can be done since $\theta(x,a)$ is decreasing in $a$ at $a = 0$ for each $x \in M$ by (\ref{eqn: differential equation for area elements}) and (\ref{eqn: geometry along level sets})). Note that as $L \cap B_4 = \emptyset$ we have that $f(x) = 1$ for all $x \in L$.

\bigskip

By (\ref{eqn: geometry along level sets}) (relying on the positive Ricci curvature assumption) and the above paragraph we thus have, for $a \in [-t_0,t_0]$ and $x \in L$ that
\begin{equation*}
\begin{cases}
    (\lambda - H(x,a))\theta(x,a) \leq -m\,a\,\theta(x,l) \text{ if } a \in [0,t_0]\\
    (\lambda - H(x,a))\theta(x,a) \geq -m\,a\,\theta(x,l) \text{ if } a \in [-t_0,0].
\end{cases}
\end{equation*}
Using these inequalities, (\ref{eqn: fixed area for energy drop}), and the fact that $e_\varepsilon(\overline{\mathbb{H}}^\varepsilon(a - r)) \neq 0$ for $a \in (r - 2\varepsilon\Lambda_\varepsilon, r + 2\varepsilon\Lambda_\varepsilon)$, for $t \in [-t_0,t_0]$ we have 
\begin{align*}
    (\ref{eqn: underlying geometry term}) &\leq -\mathcal{E}_\varepsilon(\overline{\mathbb{H}}^\varepsilon) \left[\frac{ m\sigma}{2}\mathcal{H}^n(M)t^2 - 2\varepsilon\Lambda_\varepsilon m\sigma\mathcal{H}^n(M)t\right]\\
    &\leq -(1 - \beta \varepsilon^2)\left[\frac{ m\sigma}{2}\mathcal{H}^n(M)t^2 - 2\varepsilon\Lambda_\varepsilon m\sigma\mathcal{H}^n(M)t\right];
\end{align*}
where for the first inequality above we work on $L$ (where $f_s(x) = 1$), noting that $a > r - 2\varepsilon \Lambda_\varepsilon$ in the interval where $e_\varepsilon(\overline{\mathbb{H}}^\varepsilon(a - r)) \neq 0$, and for the second inequality we use (\ref{eqn: 1d solution errors in eps^2}). Third, we focus on (\ref{eqn: gradient term}). Using (\ref{eqn: second bound on area elements}), (\ref{eqn: 1d solution errors in eps^2}), (\ref{eqn: gradient bound for h}) (\ref{eqn: annulus ratio for energy drop}) and noting that $s \in [0,1]$, we see that
\begin{equation*}
    (\ref{eqn: gradient term}) \leq 2\sigma t^2 \left(\frac{\mathcal{H}^n(\mathcal{A})C^2_h e^{\frac{\lambda^2}{2m}}}{r^2}\right)\mathcal{E}_\varepsilon(\overline{\mathbb{H}}^\varepsilon) \leq \frac{m\sigma}{4}\mathcal{H}^n(M)t^2 (1 + \beta \varepsilon^2).
\end{equation*}
We now simplify the terms as computed above and see that for $t \in [-t_0,t_0]$ we have
\begin{align}
    \label{eqn: diffused energy drop term} 2\sigma(\mathcal{F}_{\varepsilon,\lambda}(v^{t,s}_\varepsilon) - \mathcal{F}_{\varepsilon,\lambda}(v_\varepsilon)) &\leq  -(1 - \beta \varepsilon^2)\left[\frac{ m\sigma}{4}\mathcal{H}^n(M)t^2 - 2\varepsilon\Lambda_\varepsilon m\sigma\mathcal{H}^n(M)t\right]\\
    \label{eqn: theta error term 1}
    &\quad+ \int_0^t \int_M f_s(x) \int_{\sigma^-(x)}^{\sigma^+(x)} \sigma\lambda  (\overline{\mathbb{H}}^\varepsilon)'(a - rf_s(x))\theta(x,a) da\,d\mathcal{H}^n(x)\,dr\\
    \label{eqn: theta error term 2}
    &\quad - \int_0^t \int_M f_s(x) \int_{\sigma^-(x)}^{\sigma^+(x)} \lambda e_\varepsilon(\overline{\mathbb{H}}^\varepsilon(a - rf_s(x))) \theta(x,a) da \, d\mathcal{H}^n(x) \, dr\\
    \label{eqn: sigma + term}
    &\quad\quad -\int_0^t\int_M f_s(x) e_\varepsilon\left(\overline{\mathbb{H}}^\varepsilon(\sigma^+(x) - rf_s(x))\right)\theta(x,\sigma^+(x)) \, d\mathcal{H}^n(x)\,dr\\
    \label{eqn: sigma - term}
    &\quad\quad +\int_0^t\int_M f_s(x) e_\varepsilon\left(\overline{\mathbb{H}}^\varepsilon(\sigma^-(x) - rf_s(x))\right)\theta(x,\sigma^-(x)) \, d\mathcal{H}^n(x)\,dr.
\end{align}
The objective now is to show that the sum of the terms (\ref{eqn: theta error term 1}) to (\ref{eqn: sigma - term}) are errors uniformly small in $\varepsilon$.

\bigskip

First we treat (\ref{eqn: theta error term 1}) and (\ref{eqn: theta error term 2}) together. We define
\begin{equation*}
    m_\varepsilon(rf_s(x),x) = \max_{a \in [rf_s(x) - 2\varepsilon\Lambda_\varepsilon,rf_s(x) + 2\varepsilon\Lambda_\varepsilon]} \theta(x,a) - \min_{a \in [rf_s(x) - 2\varepsilon\Lambda_\varepsilon,rf_s(x) + 2\varepsilon\Lambda_\varepsilon]} \theta(x,a)
\end{equation*}
so that by (\ref{eqn: second bound on area elements}), $0 \leq m_\varepsilon(rf_s(x),x) \leq \max_{a \in\mathbb{R}}\theta(x,a) \leq e^\frac{\lambda^2}{2m}$ and $m_\varepsilon(rf_s(x),x) \rightarrow 0$ as $\varepsilon \rightarrow 0$. Hence, by the Dominated Convergence Theorem we see that
\begin{equation*}
   M_\varepsilon(t) = 2\sigma\lambda \int_M \int_0^t m_\varepsilon(rf_s(x),x)d\mathcal{H}^n(x) \, dr \rightarrow 0 \text{ as } \varepsilon \rightarrow 0.
\end{equation*}
Noting that for $0 < \tilde{\varepsilon} < \varepsilon$ we have that $0 \leq m_{\tilde{\varepsilon}}(rf_s(x),x) \leq m_\varepsilon(rf_s(x),x)$ we may apply Dini's Theorem (as the $M_\varepsilon(t)$ are increasing/decreasing in $\varepsilon$ for positive/negative values of $t \in \mathbb{R}$ respectively) to the functions $M_\varepsilon(t)$ to conclude that 
\begin{equation}\label{eqn: M term uniform null on compact sets}
    M_\varepsilon \rightarrow 0 \text{ uniformly on compact subsets of } \mathbb{R}
\end{equation}
Using the above we compute that for (\ref{eqn: theta error term 1}) and (\ref{eqn: theta error term 2}), by $|f_s|,|\overline{\mathbb{H}}^\varepsilon| \leq 1$, (\ref{eqn: 1d solution errors in eps^2}) and  Fubini's Theorem, we have, for $t \in [-t_0,t_0]$, that 
\begin{align*}
    &\left|\int_0^t \int_M f_s(x) \int_{\sigma^-(x)}^{\sigma^+(x)} \sigma\lambda  (\overline{\mathbb{H}}^\varepsilon)'(a - rf_s(x))\theta(x,a)
   - \lambda e_\varepsilon(\overline{\mathbb{H}}^\varepsilon(a - rf_s(x))) \theta(x,a) da \, d\mathcal{H}^n(x) \, dr \right|\\
   &\leq \int_0^t \int_M \left|\int_{\sigma^-(x)}^{\sigma^+(x)}  \sigma\lambda  (\overline{\mathbb{H}}^\varepsilon)'(a - rf_s(x))\theta(x,a)
   - \lambda e_\varepsilon(\overline{\mathbb{H}}^\varepsilon(a - rf_s(x))) \theta(x,a) da \right| d\mathcal{H}^n(x) \, dr\\
   &\leq 2\sigma\lambda \int_0^t \int_M  m_\varepsilon(rf_s(x),x) d\mathcal{H}^n(x) \, dr +  2\sigma\lambda\mathcal{H}^n(M)e^\frac{\lambda^2}{2m}t_0\beta \varepsilon^2\\
   &\leq \max_{t \in [-2d(N),2d(N)]} M_\varepsilon(t) +  2\sigma\lambda\mathcal{H}^n(M)e^\frac{\lambda^2}{2m}t_0\beta \varepsilon^2.
\end{align*}
So 
\begin{equation}\label{eqn: theta error bound}
    |(\ref{eqn: theta error term 1}) + (\ref{eqn: theta error term 2})| \leq \max_{t \in [-2d(N),2d(N)]} M_\varepsilon(t) +  2\sigma\lambda\mathcal{H}^n(M)e^\frac{\lambda^2}{2m}t_0\beta \varepsilon^2.
\end{equation}
Finally we work on (\ref{eqn: sigma + term}) and (\ref{eqn: sigma - term}). If $t \geq 0$ then as $f_s, \theta$ and $e_\varepsilon$ are non-negative we have $(\ref{eqn: sigma + term}) \leq 0$; similarly, if $t \leq 0$ then we have $(\ref{eqn: sigma - term}) \leq 0$. As $|f_s| \leq 1$ by definition, $\theta \leq e^\frac{\lambda^2}{2m}$ by (\ref{eqn: second bound on area elements}) and $e_\varepsilon(\overline{\mathbb{H}}^\varepsilon(\sigma^\pm(x) -rf_s(x)) \neq 0$ only when $\sigma^\pm(x) \in (rf_s(x) - 2\varepsilon\Lambda_\varepsilon, rf_s(x) + 2\varepsilon\Lambda_\varepsilon)$ by definition of $\overline{\mathbb{H}}^\varepsilon$, we conclude that
\begin{equation}\label{eqn: sigma + error}
    (\ref{eqn: sigma + term}) \leq (2\sigma + \beta \varepsilon^2) t_0 e^\frac{\lambda^2}{2m}\mathcal{H}^n(\{ x \in M \: | \: \sigma^+(x) \leq 2\varepsilon\Lambda_\varepsilon \}
\end{equation}
and
\begin{equation}\label{eqn: sigma - error}
    (\ref{eqn: sigma - term}) \leq (2\sigma + \beta \varepsilon^2) t_0 e^\frac{\lambda^2}{2m}\mathcal{H}^n(\{ x \in M \: | \: \sigma^-(x) \geq -2\varepsilon\Lambda_\varepsilon \}.
\end{equation}
As $\mathcal{H}^n(\{x \in M \: | \: \sigma^\pm(x) = 0 \}) = 0$ by the Dominated Convergence Theorem we conclude that 
\begin{equation}\label{eqn: sigma + error null}
    \mathcal{H}^n(\{ x \in M \: | \: \sigma^+(x) \leq 2\varepsilon\Lambda_\varepsilon \} \rightarrow 0 \text{ as } \varepsilon \rightarrow 0
\end{equation}
and
\begin{equation}\label{eqn: sigma - error null}
    \mathcal{H}^n(\{ x \in M \: | \: \sigma^-(x) \geq -2\varepsilon\Lambda_\varepsilon \} \text{ as } \varepsilon \rightarrow 0.
\end{equation}

\bigskip

Hence we conclude that for $t \in [-t_0,t_0]$
\begin{equation*}
    \mathcal{F}_{\varepsilon,\lambda}(v^{t,s}_\varepsilon) \leq \mathcal{F}_{\varepsilon,\lambda}(v_\varepsilon) - \frac{m}{8}\mathcal{H}^n(M)t^2 + E(\varepsilon)
\end{equation*}
where by (\ref{eqn: diffused energy drop term}), (\ref{eqn: theta error bound}), (\ref{eqn: sigma + error}) and (\ref{eqn: sigma - error}) we have noting that as $t_0 \leq 2d(N)$ we have
\begin{align*}
    2\sigma E(\varepsilon) &=  \frac{3m\sigma}{4}\mathcal{H}^n(M)(2d(N))^2\beta\varepsilon^2 + (1 - \beta \varepsilon^2)\left[ 2\varepsilon\Lambda_\varepsilon m\sigma\mathcal{H}^n(M)(2d(N))\right]\\
    &\quad\quad+ \max_{t \in [-2d(N),2d(N)]} M_\varepsilon(t) +  2\sigma\lambda\mathcal{H}^n(M)e^\frac{\lambda^2}{2m}(2d(N)) \beta \varepsilon^2\\
    &\quad\quad\quad+ (2\sigma + \beta \varepsilon^2) (2d(N)) e^\frac{\lambda^2}{2m}\mathcal{H}^n(\{ x \in M \: | \: \sigma^+(x) \leq 2\varepsilon\Lambda_\varepsilon \}\\
    &\quad\quad\quad\quad+ (2\sigma + \beta \varepsilon^2) (2d(N)) e^\frac{\lambda^2}{2m}\mathcal{H}^n(\{ x \in M \: | \: \sigma^-(x) \geq -2\varepsilon\Lambda_\varepsilon \}
\end{align*}
and so $E(\varepsilon) \rightarrow 0$ as $\varepsilon \rightarrow 0$ by virtue of (\ref{eqn: M term uniform null on compact sets}), (\ref{eqn: sigma + error null}) and (\ref{eqn: sigma - error null}). Setting once and for all 
\begin{equation*}
    \eta = \frac{m}{8}\mathcal{H}^n(M)t_0^2
\end{equation*}
we conclude, taking $\varepsilon > 0$ sufficiently small so that $E(\varepsilon) < \eta$, the energy bounds as desired.
\end{proof}

We now conclude this subsection by showing that the paths $t \in [-t_1,t_1] \rightarrow v^{t,s}_\varepsilon$ and $s \in [-1,1] \rightarrow v^{t,s}_\varepsilon$ are continuous, for fixed $s \in [0,1]$ and $t \in [-t_1,t_1]$ respectively, in $W^{1,2}(N)$ (note that we make no claim of, and indeed do not require, the joint continuity in both variables). Recall that the functions $F$, $\Pi_{V_M}$ and $f_s$ are continuous. For each $t \in [-t_1,t_1]$ consider that for $s,\tilde{s} \in [0,1]$ we have, by the Dominated Convergence Theorem (using $|\overline{\mathbb{H}}^\varepsilon| \leq 1$) and (\ref{eqn: shifting functions on cylinder}), that
\begin{align*}
    ||v^{t,s}_\varepsilon - v^{t,\tilde{s}}_\varepsilon ||^2_{L^2(N)} &= \int_{V_M} |\hat{v}^{t,s}_\varepsilon - \hat{v}^{t,\tilde{s}}_\varepsilon|^2 d\mathcal{H}^{n+1}_{F^*g}\\
    &= \int_{V_M} |\overline{\mathbb{H}}^\varepsilon(a - tf_s(F(\Pi_{V_M}(x,a))))- \overline{\mathbb{H}}^\varepsilon(a - tf_{\tilde{s}}(F(\Pi_{V_M}(x,a))))|^2 d\mathcal{H}^{n+1}_{F^*g}(x,a)\\
    &\rightarrow 0 \text{ as } \tilde{s} \rightarrow s.
\end{align*}
Also, we see that by the Dominated Convergence Theorem, (\ref{eqn: shifting functions on cylinder}) and (\ref{eqn: Gauss lemma on shifting functions}) we have
\begin{align*}
    ||\nabla v^{t,s}_\varepsilon - \nabla v^{t,\tilde{s}}_\varepsilon||^2_{L^2(N)} 
    &=  \int_{V_M} |\nabla \hat{v}^{t,s}_\varepsilon - \nabla \hat{v}^{t,\tilde{s}}_\varepsilon|^2 d\mathcal{H}^{n+1}_{F^*g}\\
    &\leq \int_{V_M} \Big((\overline{\mathbb{H}}^{{\varepsilon}})' \big(a - tf_s(F(\Pi_{V_M}(x,a)))\big) \\
    &\quad\quad\quad\quad\quad\quad- (\overline{\mathbb{H}}^{{\varepsilon}})'\big(a - tf_{\tilde{s}}(F(\Pi_{V_M}(x,a)))\big)\Big)^2d\mathcal{H}^{n+1}_{F^*g}(x,a)\\
    &\quad + \int_{V_M} t^2 \big| \nabla\big(f_{\tilde{s}}(F(\Pi_{V_M}(x,a)))\big)(\overline{\mathbb{H}}^{{\varepsilon}})'\big(a - tf_{\tilde{s}}(F(\Pi_{V_M}(x,a)))\big)\\
    &\quad \quad \quad \quad \quad  -\nabla \big(f_{s}(F(\Pi_{V_M}(x,a)))\big)(\overline{\mathbb{H}}^{{\varepsilon}})'\big(a - tf_{s}(F(\Pi_{V_M}(x,a)))\big) \big|^2d\mathcal{H}^{n+1}_{F^*g}(x,a)\\
    &\rightarrow 0 \text{ as } \tilde{s} \rightarrow s,
\end{align*}
(where we also use (\ref{eqn: gradient bound for h}) to ensure we can apply the Dominated Convergence Theorem). Hence, for fixed $t \in [-t_1,t_1]$, the path
\begin{equation}\label{eqn: sliding functions cts path in s}
  s \in [0,1] \rightarrow v^{t,s}_\varepsilon \text{ is continuous in } W^{1,2}(N).
\end{equation}
Analogous arguments to those above show that, for fixed $s \in [0,1]$, the path
\begin{equation}\label{eqn: sliding functions cts path in t}
   t \in [-t_1,t_1] \rightarrow v^{t,s}_\varepsilon \text{ is continuous in } W^{1,2}(N).
\end{equation}

\subsection{Sliding the one-dimensional profile}\label{subsec: sliding the one-dimensional profile on M}

We now define, for $t \in \mathbb{R}$, the functions $v^t_\varepsilon \in W^{1,2}(N)$ by setting
\begin{equation*}
    v^t_\varepsilon(x) = 
\overline{\mathbb{H}}^\varepsilon ( d^\pm_{\overline{M}}(x) - t),
\end{equation*}
and note that, by (\ref{eqn: properties of shifting functions}), for $t \in [-t_1,t_1]$ we have $v^t_\varepsilon = v^{t,1}_\varepsilon$. By the continuity of translations on $L^p$ for $1 \leq p < \infty$ we have that $v^t_\varepsilon \in W^{1,2}(N)$ for all $t \in \mathbb{R}$, the path $t \rightarrow v^t_\varepsilon$ is continuous in $W^{1,2}(N)$ and that $v^0_\varepsilon = v_\varepsilon$. By choosing $\varepsilon > 0$ sufficiently small so that $2 \varepsilon \Lambda_\varepsilon < d(N)$ we ensure that both $v^t_\varepsilon = -1$ everywhere on $N$ for all $t \geq 2d(N)$ and $v^t_\varepsilon = +1$ everywhere on $N$ for all $t \leq -2d(N)$.

\bigskip

We now show that, using the assumption of positive Ricci curvature, the path provided by the functions $v^t_\varepsilon$, along with the energy reducing paths (to be constructed later) from $-1$ and $+1$ provided by negative gradient flow of the energy to $a_\varepsilon$ and $b_\varepsilon$ respectively, provides a “recovery path" for the value $\mathcal{F}_\lambda(E)$; this path connects $a_\varepsilon$ to $b_\varepsilon$, passing through $v_\varepsilon$, with the maximum value of the energy along this path approximately $\mathcal{F}_\lambda(E)$.

\begin{lem}\label{lem: sliding path upper energy bounds}
For all $t \in [-2d(N),2d(N)]$ we have that
\begin{equation*}
    \mathcal{F}_{\varepsilon,\lambda}(v^t_\varepsilon) \leq \mathcal{F}_{\varepsilon,\lambda}(v_\varepsilon) + E(\varepsilon)
\end{equation*}
where $E(\varepsilon) \rightarrow 0$ as $\varepsilon \rightarrow 0$ is as in Lemma \ref{lem: shifting path upper energy bounds}. Furthermore, we have that
\begin{equation*}
    \mathcal{F}_{\varepsilon,\lambda}(v^{ \tilde{t}}_\varepsilon) \leq \mathcal{F}_{\varepsilon,\lambda}(v^{ t}_\varepsilon)+ E(\varepsilon),
\end{equation*}
whenever $\tilde{t} \geq t \geq 0$ or $\tilde{t} \leq t \leq 0$; thus, in particular for $t > t_0$ or $t < -t_0$ we have
\begin{equation*}
   \mathcal{F}_{\varepsilon,\lambda}(v^{ t}_\varepsilon)  \leq \mathcal{F}_{\varepsilon,\lambda}(v_\varepsilon) - \eta + E(\varepsilon).
\end{equation*}
\end{lem}
\begin{proof}
We compute in an identical manner to the proof of Lemma \ref{lem: shifting path upper energy bounds}, writing
\begin{align*}
    2\sigma \mathcal{F}_{\varepsilon,\lambda}(v^t_\varepsilon) = \int_M \int_{\sigma^-(x)}^{\sigma^+(x)} \left[e_\varepsilon\left(\overline{\mathbb{H}}^\varepsilon(a - t)\right) - \sigma\lambda\overline{\mathbb{H}}^\varepsilon(a - t) \right]\theta(x,a) da\, d\mathcal{H}^n(x).
\end{align*}
Assuming that either $\tilde{t} \geq t > 0$ or $\tilde{t} \leq t < 0$, as before we compute
\begin{align*}
    2\sigma(\mathcal{F}_{\varepsilon,\lambda}(v^{\tilde{t}}_\varepsilon) - \mathcal{F}_{\varepsilon,\lambda}(v^t_\varepsilon)) &= \int_M \int_{\sigma^-(x)}^{\sigma^+(x)} \left[e_\varepsilon\left(\overline{\mathbb{H}}^\varepsilon(a - \tilde{t})\right) - e_\varepsilon\left(\overline{\mathbb{H}}^\varepsilon(a - t)\right) \right]\theta(x,a) da\, d\mathcal{H}^n(x)\\
    \label{eqn: 1d term}
    &\quad -\int_M \int_{\sigma^-(x)}^{\sigma^+(x)} \sigma\lambda\left[\overline{\mathbb{H}}^\varepsilon(a - \tilde{t}) - \overline{\mathbb{H}}^\varepsilon(a - t) \right]\theta(x,a) da\, d\mathcal{H}^n(x)
\end{align*}
which yields, noting that $\lambda - H(x,a) \leq 0$ for all $a \in \mathbb{R}$ by the assumption of positive Ricci curvature, the expression
\begin{align*}
     2\sigma(\mathcal{F}_{\varepsilon,\lambda}(v^{\tilde{t}}_\varepsilon) - \mathcal{F}_{\varepsilon,\lambda}(v^t_\varepsilon)) &\leq  \int_t^{\tilde{t}} \int_M \int_{\sigma^-(x)}^{\sigma^+(x)} \sigma\lambda  (\overline{\mathbb{H}}^\varepsilon)'(a - r)\theta(x,a) da\,d\mathcal{H}^n(x)\,dr\\
     &\quad - \int_t^{\tilde{t}} \int_M \int_{\sigma^-(x)}^{\sigma^+(x)} \lambda e_\varepsilon(\overline{\mathbb{H}}^\varepsilon(a - r)) \theta(x,a) da \, d\mathcal{H}^n(x) \, dr\\
     &\quad\quad -\int_t^{\tilde{t}} \int_M e_\varepsilon\left(\overline{\mathbb{H}}^\varepsilon(\sigma^+(x) - r)\right)\theta(x,\sigma^+(x)) \, d\mathcal{H}^n(x)\,dr\\
     &\quad\quad\quad +\int_t^{\tilde{t}} \int_M e_\varepsilon\left(\overline{\mathbb{H}}^\varepsilon(\sigma^-(x) - r)\right)\theta(x,\sigma^-(x)) \, d\mathcal{H}^n(x)\,dr.
\end{align*}
Near identical computations to those in the proof of Lemma \ref{lem: shifting path upper energy bounds} (for the error terms (\ref{eqn: theta error term 1}), (\ref{eqn: theta error term 2}) giving the bound (\ref{eqn: theta error bound}), as well the bounds on (\ref{eqn: sigma + error}), (\ref{eqn: sigma - error}) giving (\ref{eqn: sigma + error null}), (\ref{eqn: sigma - error null}) respectively) for the above four terms give that 
\begin{equation*}
    \mathcal{F}_{\varepsilon,\lambda}(v^{\tilde{t}}_\varepsilon) \leq \mathcal{F}_{\varepsilon,\lambda}(v^t_\varepsilon) + E(\varepsilon),
\end{equation*}
where the expression for $E(\varepsilon)$ is identical to that of Lemma \ref{lem: shifting path upper energy bounds} (where here are exploiting the fact that $t \in [-2d(N),2d(N)]$).

\bigskip

Recalling $v^0_\varepsilon = v_\varepsilon$, $v^{\pm t_0}_\varepsilon = v^{\pm t_0, 1}_\varepsilon$ and the bounds from Lemma \ref{lem: shifting path upper energy bounds} then completes the proof.
\end{proof}

\subsection{Paths to local energy minimisers}\label{subsec: local path}

Recall, $v_\varepsilon = \overline{\mathbb{H}}^\varepsilon 
\circ d^\pm_{\overline{M}}$ and that, by Lemma \ref{lem: shifting path upper energy bounds} and (\ref{eqn: properties of shifting functions}), for a fixed $\eta > 0$ and some $r_0 > 0$, as chosen in Subsection \ref{subsec: choosing radii}, there exists a $t_0 > 0$ and functions $v^{t,0}_\varepsilon \in W^{1,2}(N)$ for $t \in [-t_0,t_0]$, with the following properties:
\begin{itemize}
    \item $v^{t,0}_\varepsilon = v_\varepsilon$ in $B_{r_0}(p)$.

    \item $\mathcal{F}_{\varepsilon,\lambda}(v^{t,0}_\varepsilon) \leq \mathcal{F}_{\varepsilon,\lambda}(v_\varepsilon) + E(\varepsilon)$, where $E(\varepsilon) \rightarrow 0$ as $\varepsilon \rightarrow 0$.

    \item $\mathcal{F}_{\varepsilon,\lambda}(v^{\pm t_0,0}_\varepsilon) \leq \mathcal{F}_{\varepsilon,\lambda}(v_\varepsilon) - \eta$.
\end{itemize}
We now produce the local functions described in Step 3 of Subsection \ref{subsec: strategy}. These functions are constant outside of $B_{r_0}(p)$, providing a continuous path in $W^{1,2}(N)$ from $v_\varepsilon$ to $g_\varepsilon$ such that the maximum energy along this path is bounded from above by $\mathcal{F}_{\varepsilon,\lambda}(v^{t,0}_\varepsilon) + \frac{\eta}{2}$.

\begin{prop}\label{prop: local path}
For $\eta > 0$ as above there exists, for some $R \in (0,r_0)$ and $\varepsilon > 0$ sufficiently small, functions $g^{t,s}_\varepsilon \in W^{1,2}(N)$, for each $t \in [-t_0,t_0]$ and $s \in [-2,2]$, such that the following properties hold:
\begin{itemize}
    \item For each $t \in [-t_0,t_0]$ we have $g^{t,-2}_\varepsilon = v^{t,0}_\varepsilon$ on $N$ and $g^{t,s}_\varepsilon = v^{t,0}_\varepsilon$ on $N \setminus B_R(p)$ for all $s \in [-2,2]$.
    
    \item For each $t \in [-t_0,t_0]$ we have $g^{t,2}_\varepsilon = g_\varepsilon$ in $B_R(p)$ where $g_\varepsilon \in \mathcal{A}_{\varepsilon,\frac{R}{2}}(p)$ arises from Lemma \ref{lem: function minimisation} for $\rho = \frac{R}{2}$ and $q = p$.
 
    \item For each $s \in [-2,2]$, $t \in [-t_0,t_0] \rightarrow g^{t,s}_\varepsilon$ is a continuous path in $W^{1,2}(N)$.

    \item For each $t \in [-t_0,t_0]$, $s \in [-2,2] \rightarrow g^{t,s}_\varepsilon$ is a continuous path in $W^{1,2}(N)$.
    
    \item For each $t \in [-t_0,t_0]$ and $s \in [-2,2]$ we have 
    \begin{equation*}
        \mathcal{F}_{\varepsilon,\lambda}(g^{t,s}_\varepsilon) \leq \mathcal{F}_{\varepsilon,\lambda}(v^{t,0}_\varepsilon) + \frac{\eta}{2}.
    \end{equation*}
\end{itemize}
Furthermore, if the functions $v_\varepsilon$ are such that $\mathcal{F}_{\varepsilon,\lambda}(v_\varepsilon) \geq \mathcal{F}_{\varepsilon,\lambda}(g_\varepsilon) + \tau$ for some $\tau > 0$, we have for each $t \in [-t_0,t_0]$ that 
\begin{equation*}
    \mathcal{F}_{\varepsilon,\lambda}(g^{t,2}_\varepsilon) \leq \mathcal{F}_{\varepsilon,\lambda}(v^{t,0}_\varepsilon) - \tau.
\end{equation*}
\end{prop}

\begin{rem}
The paths exhibited in Proposition \ref{prop: local path} are “local" in the sense that they remain constant outside of $B_R(p)$, and additionally do not require the assumption of positive Ricci curvature of $N$. Specifically, in any ambient manifold (without any curvature assumption), given an $\eta > 0$ (not necessarily as fixed by Lemma \ref{lem: shifting path upper energy bounds}) and $f \in W^{1,2}(N)$ such that $|f| \leq 1$ and $f$ equal to $v_\varepsilon$ in a ball bi-Lipschitz diffeomorphic to the Euclidean ball of the same radius, a path as in Proposition \ref{prop: local path} may be constructed (replacing $v^{t,0}_\varepsilon$ by $f$ in the conclusions) for a sufficiently small $R > 0$.
\end{rem}

\begin{proof}
Given some $\eta > 0$ we choose $R > 0$ sufficiently small to ensure that
\begin{equation*}
     2^{n+2}R^n\omega_n + C_pR^n + 2\lambda \mathrm{Vol}(B_{\frac{R}{2}}(p)) < \frac{\eta}{4},
\end{equation*}
where $\omega_n$ is the volume of the $n$-dimensional unit ball. We then potentially re-choose a smaller $0 < R \leq \min\{r_0, r_p\}$, such that the closed ball $\overline{B}_R(p)$ is $2$-bi-Lipschitz diffeomorphic, via some smooth chart $\psi$ on $N$, to the Euclidean ball, $\overline{B}^{\mathbb{R}^{n+1}}_R(0)$ of radius $R$, with $\psi(p) = 0$. Consider a sweep-out of $\overline{B}^{\mathbb{R}^{n+1}}_R$ by planes $\Pi_l$ defined for $l \in [-1,1]$ by
\begin{equation*}
    \Pi_l = \left\{y \in \overline{B}^{\mathbb{R}^{n+1}}_R \: \Big| \: y = (y_1, \dots, y_n, lR) \right\},
\end{equation*}
which for each $l \in [-1,1]$ is the intersection of $\overline{B}^{\mathbb{R}^{n+1}}_R$ with the plane $\{ y_{n+1} = lR \}$. We consider the images, $P_l = \psi(\Pi_l)$ in $N$. The $P_l$ then sweep-out the closure of $B_R(p)$, $\overline{B}_R(p)$ in the sense that $\overline{B}_R(p) = \cup_{l \in [-1,1]} P_l$ and that for $l \neq s$ we have $P_l \cap P_s = \emptyset$. Each $\Pi_l$ divides $\overline{B}^{\mathbb{R}^{n+1}}_R$ into two connected regions $\{ y \in \overline{B}^{\mathbb{R}^{n+1}}_R \: | \: y_{n+1} > lR \}$ and $\{ y \in \overline{B}^{\mathbb{R}^{n+1}}_R \: | \: y_{n+1} < lR \}$; we denote the images of these sets under the diffeomorphism $\psi$ as $E_l = \psi(\{ y \in \overline{B}^{\mathbb{R}^{n+1}}_R \: | \: y_{n+1} > lR \})$ and $F_l = \psi(\{ y \in \overline{B}^{\mathbb{R}^{n+1}}_R \: | \: y_{n+1} < lR \})$ in $\overline{B}_R(p)$ respectively. Note then that $\overline{B}_R(p) = E_l \cup F_l \cup P_l$, where the unions are all mutually disjoint.

\bigskip

We now define for $l \in [-1,1]$ functions $p^l_\varepsilon \in W^{1,2}(B_R(p))$ given by
\begin{equation*}
    p^l_\varepsilon = \overline{\mathbb{H}}^\varepsilon(d^\pm_{P_l}),
\end{equation*}
where here, for $d_{P_l}$ the usual distance function on $N$ from the set $P_l$, we define the Lipschitz signed distance to $P_l$ by
\begin{equation*}
    d^\pm_{P_l} =
    \begin{cases}
    +d_{P_l}, \text{ if } x \in E_l\\
    0, \text{ if } x \in P_l\\
    -d_{P_l}, \text{ if } x \in F_l.
    \end{cases}
\end{equation*}
As the sets $P_l$ vary continuously in the Hausdorff distance as $l \in [-1,1]$ varies, the functions $p^l_\varepsilon$ vary continuously in $W^{1,2}(B_R(p))$ with respect to $l \in [-1,1]$. 

\bigskip

Apply Lemma \ref{lem: function minimisation} for the choice of $B_\frac{R}{2}(p)$ for each $0 < \varepsilon < \tilde{\varepsilon}$ as in the minimisation section. For $l \in [-1,1]$ we define on $B_R(p)$ the functions (see Figure \ref{fig: local path} for motivation of the definition)
\begin{equation*}
    \check{h}^l_\varepsilon(x) = \max\Big\{\min\{v_\varepsilon(x), p^l_\varepsilon(x)\},  \min\{g_\varepsilon(x),v_\varepsilon(x)\} \Big\},
\end{equation*}
and 
\begin{equation*}
    \hat{h}^l_\varepsilon(x) = \max\Big\{\min\{g_\varepsilon(x), p^l_\varepsilon(x)\}, \min\{g_\varepsilon(x),v_\varepsilon(x)\} \Big\}.
\end{equation*}
Note then that for $2\varepsilon\Lambda_\varepsilon < \frac{R}{2}$ we have $\check{h}^{-1}_\varepsilon = v_\varepsilon$, $\check{h}^1_\varepsilon = \hat{h}^1_\varepsilon = \min\{g_\varepsilon,v_\varepsilon\}$, $\hat{h}^{-1}_\varepsilon = g_\varepsilon$, and both $|\check{h}^l_\varepsilon| \leq 1$, $|\hat{h}^l_\varepsilon| \leq 1$.

\bigskip

For functions $f,g \in W^{1,2}(B_R(p))$ we may write $\max \{f,g\} = \max\{f - g, 0\} + g = \frac{1}{2}(|f-g| +(f-g)) + g$ and $\min \{f,g\} = \min\{f - g, 0\} + g = \frac{1}{2}(|f-g| - (f-g)) + g$; thus $\max \{f,g\}, \min \{f,g\} \in W^{1,2}(N)$. We also note that if $t \rightarrow f_t$ is a continuous path in $W^{1,2}(B_R(p))$, then the paths $t \rightarrow \max\{f_t,0\}$ and $t \rightarrow \min\{f_t,0\}$ are continuous in $W^{1,2}(B_R(p))$. To see this we write $\max\{f_t,0\} - \max\{f_s,0\} = \frac{1}{2}((|f_t|-|f_s|) + (f_t - f_s))$ which by reverse triangle inequality may be controlled by the $W^{1,2}(B_R(p))$ norm of $f_t - f_s$, showing continuity of the path in $W^{1,2}(B_R(p))$. The proof for the continuity of the path $t \rightarrow \min\{f_t,0\}$ is identical. With the above in mind, we have that $\check{h}_\varepsilon^l, \hat{h}^l_\varepsilon \in W^{1,2}(B_R(p))$ for each $l \in [-1,1]$ and the paths $l \in [-1,1] \rightarrow \check{h}^l_\varepsilon$ and $l \in [-1,1] \rightarrow \hat{h}^l_\varepsilon$ are continuous in $W^{1,2}(B_R(p))$.

\bigskip

We have that $v_\varepsilon = g_\varepsilon$ on $B_R(p) \setminus \overline{B_\frac{R}{2}(p)}$ (as $v_\varepsilon = g_\varepsilon$ on $N \setminus B_\frac{R}{2}(p)$) and hence $\hat{h}^l_\varepsilon = \check{h}^l_\varepsilon = v_\varepsilon$ on $B_R(p) \setminus \overline{B_\frac{R}{2}(p)}$ also. Thus we may extend to well defined functions $\check{g}^{t,l}_\varepsilon, \hat{g}^{t,l}_\varepsilon \in W^{1,2}(N)$ for $t \in [-t_0,t_0]$ and $l \in [-1,1]$ by defining
\begin{equation*}
    \check{g}^{t,l}_\varepsilon(x) = 
    \begin{cases}
    v^{t,0}_\varepsilon(x), \text{ if } x \in N \setminus B_R(p)\\
    \check{h}^l_\varepsilon(x), \text{ if } x \in B_R(p),
    \end{cases}
\end{equation*}
and
\begin{equation*}
    \hat{g}^{t,l}_\varepsilon(x) = 
    \begin{cases}
    v^{t,0}_\varepsilon(x), \text{ if } x \in N \setminus B_R(p)\\
    \hat{h}^l_\varepsilon(x), \text{ if } x \in B_R(p),
    \end{cases}
\end{equation*}
This is well defined as we have that $v^{t,0}_\varepsilon = v_\varepsilon$ on the set $B_R(p)$, hence in defining the functions above we only edit the functions $v_\varepsilon^{t,0}$ in $B_{\frac{R}{2}}(p)$ and keep them in $W^{1,2}(N)$. By the above arguments, as the paths $l \in [-1,1] \rightarrow \check{h}^l_\varepsilon$ and $l \in [-1,1] \rightarrow \hat{h}^l_\varepsilon$ are continuous in $W^{1,2}(B_R(p))$ and so we have that the paths $t \in [-t_0,t_0] \rightarrow \check{g}^{t,l}_\varepsilon$, $l \in [-1,1] \rightarrow \check{g}^{t,l}_\varepsilon$, $t \in [-t_0,t_0] \rightarrow \hat{g}^{t,l}_\varepsilon$ and $l \in [-1,1] \rightarrow \hat{g}^{t,l}_\varepsilon$ are continuous in $W^{1,2}(N)$. 

\bigskip

Note that then as $\check{h}^1_\varepsilon = \hat{h}^1_\varepsilon = \min\{g_\varepsilon,v_\varepsilon\}$ we have for each $t \in [-t_0,t_0]$ that $\check{g}^{t,1}_\varepsilon = \hat{g}^{t,-1}_\varepsilon = \min\{g_\varepsilon,v_\varepsilon\}$ in $B_R(p)$. We thus define, for $s \in [-2,2]$, the functions
\begin{equation*}
    g^{t,s}_\varepsilon = 
    \begin{cases}
    \check{g}^{t,s+1}, \text{ if } s \in [-2,0]\\
    \hat{g}^{t,s-1}_\varepsilon, \text{ if } s \in [0,2],
    \end{cases}
\end{equation*}
so that $g^{t,0}_\varepsilon = \check{g}^{t,1}_\varepsilon = \hat{g}^{t,-1}_\varepsilon = \min\{g_\varepsilon,v_\varepsilon\}$, $g^{t,-2}_\varepsilon = \check{g}^{0,-1}_\varepsilon = v^{t,0}_\varepsilon$ and  $g^{0,2}_\varepsilon = \hat{g}^{0,1}_\varepsilon = g_\varepsilon$ in $B_R(p)$. By the continuity of the paths mentioned above we have that the paths, $s \in [-2,2] \rightarrow g^{t,s}_\varepsilon$ and $t \in [-t_0,t_0] \rightarrow g^{t,s}_\varepsilon$, are continuous in $W^{1,2}(N)$.

\bigskip

We now show that we may bound the energy of all the functions $g^{t,s}_\varepsilon$ above by $\mathcal{F}_{\varepsilon,\lambda}(v^{t,0}_\varepsilon)$ plus errors depending only on the geometry of the ball $B_R(p)$. As $\{ g_\varepsilon \neq v_\varepsilon \} \subset B_{\frac{R}{2}}(p)$ we note that $g^{t,s}_\varepsilon = v^{t,0}_\varepsilon$ on $N \setminus B_{\frac{R}{2}}(p)$. We then compute that, as $\check{h}^{s+1}_\varepsilon$ is always equal to one of $g_\varepsilon,v_\varepsilon$ or $p_\varepsilon^{s+1}$ in $B_{\frac{R}{2}}(p)$, for $s \in [-2,0]$ we have
\begin{align*}
   \mathcal{F}_{\varepsilon,\lambda}(g^{t,s}_\varepsilon) &= \frac{1}{2\sigma}\int_N e_\varepsilon(g^{t,s}_\varepsilon) - \frac{\lambda}{2}\int_N g^{t,s}_\varepsilon = \frac{1}{2\sigma}\int_{N \setminus B_{\frac{R}{2}}(p)} e_\varepsilon(v^{t,0}_\varepsilon) + \frac{1}{2\sigma}\int_{B_{\frac{R}{2}}(p)} e_\varepsilon(\check{h}^{s+1}_\varepsilon)\\
   &- \frac{\lambda}{2}\int_{N \setminus B_{\frac{R}{2}}(p)} v^{t,0}_\varepsilon  - \frac{\lambda}{2}\int_{B_{\frac{R}{2}}(p)} \check{h}^{s+1}_\varepsilon\\
    &= \frac{1}{2\sigma}\int_{N \setminus B_{\frac{R}{2}}(p)} e_\varepsilon(v^{t,0}_\varepsilon) + \frac{1}{2\sigma}\int_{B_{\frac{R}{2}}(p) \cap \{\check{h}^{s+1}_\varepsilon = v_\varepsilon\}} e_\varepsilon(v_\varepsilon) + \frac{1}{2\sigma}\int_{B_{\frac{R}{2}}(p) \cap \{\check{h}^{s+1}_\varepsilon = g_\varepsilon\}} e_\varepsilon(g_\varepsilon)  \\
    &+ \frac{1}{2\sigma}\int_{B_{\frac{R}{2}}(p) \cap \{\check{h}^{s+1}_\varepsilon = p^{s+1}_\varepsilon\}} e_\varepsilon(p^{s+1}_\varepsilon) -\frac{\lambda}{2}\int_{N \setminus B_{\frac{R}{2}}(p)} v^{t,0}_\varepsilon  - \frac{\lambda}{2}\int_{B_{\frac{R}{2}}(p)} \check{h}^{s+1}_\varepsilon\\
    &\leq \frac{1}{2\sigma}\int_{N \setminus B_{\frac{R}{2}}(p)} e_\varepsilon(v^{t,0}_\varepsilon) + \frac{1}{2\sigma}\int_{B_{\frac{R}{2}}(p)} e_\varepsilon(v_\varepsilon)  + \frac{1}{2\sigma}\int_{B_{\frac{R}{2}}(p)} e_\varepsilon(g_\varepsilon) + \frac{1}{2\sigma}\int_{B_{\frac{R}{2}}(p)} e_\varepsilon(p^{s+1}_\varepsilon)\\
    &-\frac{\lambda}{2}\int_{N \setminus B_{\frac{R}{2}}(p)} v^{t,0}_\varepsilon  - \frac{\lambda}{2}\int_{B_{\frac{R}{2}}(p)} \check{h}^{s+1}_\varepsilon\\
    &= \mathcal{E}_\varepsilon(v^{t,0}_\varepsilon) + \frac{1}{2\sigma}\int_{B_{\frac{R}{2}}(p)} e_\varepsilon(g_\varepsilon) + \frac{1}{2\sigma}\int_{B_{\frac{R}{2}}(p)} e_\varepsilon(p^{s+1}_\varepsilon) -\frac{\lambda}{2}\int_{N \setminus B_{\frac{R}{2}}(p)} v^{t,0}_\varepsilon  - \frac{\lambda}{2}\int_{B_{\frac{R}{2}}(p)} \check{h}^{s+1}_\varepsilon\\
    &= \mathcal{F}_{\varepsilon,\lambda}(v^{t,0}_\varepsilon) + \frac{1}{2\sigma}\int_{B_{\frac{R}{2}}(p)} e_\varepsilon(g_\varepsilon) + \frac{1}{2\sigma}\int_{B_{\frac{R}{2}}(p)} e_\varepsilon(p^{s+1}_\varepsilon) +\frac{\lambda}{2}\int_{B_{\frac{R}{2}}(p)} (v^{t,0}_\varepsilon  - \check{h}^{s+1}_\varepsilon),
\end{align*}
where in the final line we have added and subtracted the term $\frac{\lambda}{2}\int_{B_{\frac{R}{2}}(p)}v^{t,0}_\varepsilon$. Similarly for $s \in [0,2]$ we have
\begin{equation*}
    \mathcal{F}_{\varepsilon,\lambda}(g^{t,s}_\varepsilon) \leq  \mathcal{F}_{\varepsilon,\lambda}(v^{t,0}_\varepsilon) + \frac{1}{2\sigma}\int_{B_{\frac{R}{2}}(p)} e_\varepsilon(g_\varepsilon) + \frac{1}{2\sigma}\int_{B_{\frac{R}{2}}(p)} e_\varepsilon(p^{s-1}_\varepsilon)+\frac{\lambda}{2}\int_{B_{\frac{R}{2}}(p)} (v^{t,0}_\varepsilon  - \hat{h}^{s-1}_\varepsilon).
\end{equation*}
Thus as for all $l \in [-1,1]$ we have $|\check{h}^l_\varepsilon| \leq 1$, $|\hat{h}^l_\varepsilon| \leq 1$ and $|v^{t,0}_\varepsilon| \leq 1$ we see that for each $t \in [-t_0,t_0]$ and $s \in [-2,2]$
\begin{equation}\label{eqn: local path bound 1}
    \mathcal{F}_{\varepsilon,\lambda}(g^{t,s}_\varepsilon) \leq \mathcal{F}_{\varepsilon,\lambda}(v^{t,0}_\varepsilon) + \frac{1}{2\sigma}\int_{B_{\frac{R}{2}}(p)} e_\varepsilon(g_\varepsilon) + \sup_{l \in (-1,1)}\left(\frac{1}{2\sigma}\int_{B_{\frac{R}{2}}(p)} e_\varepsilon(p^l_\varepsilon) \right) + \lambda \mathrm{Vol}(B_{\frac{R}{2}}(p)). 
\end{equation}
Similarly, we note that as $v_\varepsilon = g_\varepsilon$ on $N \setminus B_{\frac{R}{2}}(p)$, $|g_\varepsilon|, |v_\varepsilon| \leq 1$ on $N$ and $\mathcal{F}_{\varepsilon,\lambda}(g_\varepsilon) \leq \mathcal{F}_{\varepsilon,\lambda}(v_\varepsilon)$ by construction of $g_\varepsilon$ we have that
\begin{equation*}
    \frac{1}{2\sigma}\int_{B_{\frac{R}{2}}(p)} e_\varepsilon(g_\varepsilon) \leq \frac{1}{2\sigma}\int_{B_{\frac{R}{2}}(p)} e_\varepsilon(v_\varepsilon) + \lambda \mathrm{Vol}(B_{\frac{R}{2}}(p)).
\end{equation*}
and so by similar calculations to Subsection \ref{subsec: one dimensional profile} and (\ref{eqn: 1d solution errors in eps^2}) we see that 
\begin{align*}
    \frac{1}{2\sigma}\int_{B_{\frac{R}{2}}(p)} e_\varepsilon(v_\varepsilon) &\leq (1 + \beta \varepsilon^2){\esssup}_{a \in [-2\varepsilon\Lambda_\varepsilon,2\varepsilon\Lambda_\varepsilon]}\mathcal{H}^n(\{d^\pm_{\overline{M}} = a \} \cap B_R(p))\\
    &\rightarrow \mathcal{H}^n(M \cap B_R(p)) \text{ as } \varepsilon \rightarrow 0.
\end{align*}
By (\ref{eqn: monotonicity gives euclidean volume growth for M around p}) we conclude that, for sufficiently small $\varepsilon > 0$, we have
\begin{equation}\label{eqn: local path bound 2}
   \frac{1}{2\sigma}\int_{B_{\frac{R}{2}}(p)} e_\varepsilon(g_\varepsilon)
    \leq C_pR^n + \lambda \mathrm{Vol}(B_{\frac{R}{2}}(p)) + \frac{\eta}{8}.
\end{equation}
We now compute upper $\mathcal{E}_\varepsilon$ bounds on the term in (\ref{eqn: local path bound 1}) involving the functions $p^l_\varepsilon$ in terms of the geometry of $N$. Applying the co-area formula, slicing with $d^\pm_{P_l}$, and using (\ref{eqn: 1d solution errors in eps^2}) we see that for $l \in (-1,1)$ we have
\begin{align*}
    \frac{1}{2\sigma}\int_{B_{\frac{R}{2}}(p)} e_\varepsilon(p^l_\varepsilon) &= \frac{1}{2\sigma}\int_{B_{\frac{R}{2}}(p)} e_\varepsilon(\overline{\mathbb{H}}^\varepsilon(d^\pm_{P_l})) = \frac{1}{2\sigma}\int_\mathbb{R} \int_{\{ d^\pm_{P_l} = a \} \cap B_{\frac{R}{2}}(p)} e_\varepsilon(\overline{\mathbb{H}}^\varepsilon(a)) d\mathcal{H}^n \,da\\
    &= \frac{1}{2\sigma}\int_\mathbb{R} \mathcal{H}^n(\{ d^\pm_{P_l} = a \} \cap B_{\frac{R}{2}}(p))e_\varepsilon(\overline{\mathbb{H}}^\varepsilon(a)) \,da\\
    &\leq (1 + \beta \varepsilon^2){\esssup}_{a \in [-2\varepsilon\Lambda_\varepsilon, 2\varepsilon\Lambda_\varepsilon]} \mathcal{H}^n(\{ d^\pm_{P_l} = a \} \cap B_{\frac{R}{2}}(p)).
\end{align*}
We now bound ${\esssup}_{a \in [-2\varepsilon\Lambda_\varepsilon, 2\varepsilon\Lambda_\varepsilon]} \mathcal{H}^n(\{ d^\pm_{P_l} = a \} \cap B_{\frac{R}{2}}(p))$ from above independently of $\varepsilon$. Recall that $P_l = \psi(\Pi_l)$ where $\psi$ is a smooth 2-bi-Lipschitz map from $\overline{B}_R^{\mathbb{R}^{n+1}}(0)$ to $\overline{B}_R(p)$ and $\Pi_l = \left\{y \in \overline{B}^{\mathbb{R}^{n+1}}_R \: \Big| \: y = (y_1, \dots, y_n, lR) \right\}$. For $l \in (-1,1)$ we have that $\Pi_l$ is a smooth embedded $n$-dimensional sub-manifold of $\overline{B}_R^{\mathbb{R}^{n+1}}$ and hence its image, $P_l$, in $N$ is a smooth embedded $n$-dimensional sub-manifold of $\overline{B}_R(p) \subset N$.

\bigskip

We define the tubular hypersurface at distance $s$ from $P_l$ to be the set 
\begin{equation*}
P_l(a) = \{ x \in N \: | \: \text{ there exists a geodesic of length } a \text{ meeting } P_l \text{ orthogonally} \}.
\end{equation*} 
Here we choose $\varepsilon > 0$ possibly smaller to ensure that $2\varepsilon\Lambda_\varepsilon < \frac{R}{2}$, so that $\{|d^\pm_{P_l}| = a \} \cap B_{\frac{R}{2}}(p) \subset P_l(a)$ for all $l \in (-1,1)$ such that $B_{\frac{R}{2}}(p) \cap P_l \neq \emptyset$. We then apply \cite[Lemma 3.12/8.2]{tubes} to compute that, for each $l \in (-1,1)$ we have
\begin{equation*}
    \mathcal{H}^n(P_l(s)) \leq \int_{P_l} \int_{S^0} \theta_u(q,s) du \, d\mathcal{H}^n(q).
\end{equation*}
In the above we are adopting the notation that $S^0$ is the $0$ dimensional unit sphere, and for $u \in S^0$ we define $\theta_u(q,a)$ to be the Jacobian of the exponential map $\exp_q$ at the point $\exp_q(au)$. Note then that for each $q \in P_l$ we have that $\theta_u(q,0) = 1$ and $\theta_u(q,a) \rightarrow 1$ as $a \rightarrow 0$. For all $\varepsilon > 0$ sufficiently small we ensure that $\theta_u(q,a) \leq 2$ for any $q \in \overline{B}_R(p)$; this may be seen by noticing that as the exponential map is smooth, its Jacobian varies continuously in each of its variables and hence its maximum, for a fixed $a \in \mathbb{R}$ on $\overline{B}_R(p) \times \mathbb{R}$, is achieved and converges to $1$ as $a \rightarrow 0$. Thus, we have that
\begin{equation*}
    \mathcal{H}^n(P_l(a)) \leq 4 \mathcal{H}^n(P_l).
\end{equation*}
Note that as $\psi$ is $2$-bi-Lipschitz we may apply \cite[Theorem 2.8]{E-G} to see that
\begin{equation*}
    \mathcal{H}^n(P_l) \leq 2^n \mathcal{H}^n(\Pi_l) \leq 2^n R^n \omega_n,
\end{equation*} 
where $\omega_n$ is the volume of the $n$-dimensional unit ball. Hence, for all $l \in [-1,1]$ and $\varepsilon > 0$ sufficiently small we have that 
\begin{equation*}
    \mathcal{H}^n(P_l(s)) \leq 2^{n+2}R^n\omega_n.
\end{equation*}
We thus conclude that as $\{|d^\pm_{P_l}| = a \} \cap B_{\frac{R}{2}}(p) \subset P_l(a)$ for all $l \in (-1,1)$ such that $B_{\frac{R}{2}}(p) \cap P_l \neq \emptyset$ we have
\begin{equation*}
    {\esssup}_{a \in [-2\varepsilon\Lambda_\varepsilon, 2\varepsilon\Lambda_\varepsilon]}\mathcal{H}^n(\{d^\pm_{P_l} = a\} \cap B_{\frac{R}{2}}(p)) \leq 2^{n+2}R^n\omega_n,
\end{equation*}
and hence
\begin{equation}\label{eqn: local path bound 3}
    \frac{1}{2\sigma}\int_{{B_\frac{R}{2}}(p)} e_\varepsilon(p^l_\varepsilon) \leq  (1 + \beta \varepsilon^2)2^{n+2}R^n\omega_n.
\end{equation}
Choosing $\varepsilon > 0$ again possibly smaller we ensure that
\begin{equation*}
   2^{n+2}R^n\omega_n   \beta \varepsilon^2 < \frac{\eta}{8}.
\end{equation*}
Recall that, given some $\eta > 0$, we chose $R > 0$ sufficiently small to ensure that
\begin{equation*}
     2^{n+2}R^n\omega_n + C_pR^n + 2\lambda \mathrm{Vol}(B_{\frac{R}{2}}(p)) < \frac{\eta}{4}.
\end{equation*}
Combining the above two bounds with (\ref{eqn: local path bound 1}), (\ref{eqn: local path bound 2}) and (\ref{eqn: local path bound 3}) we conclude that for each $t \in [-t_0,t_0]$ and $s \in [-2,2]$ we have
\begin{equation*}
    \mathcal{F}_{\varepsilon,\lambda}(g^{t,s}_\varepsilon) \leq  \mathcal{F}_{\varepsilon,\lambda}(v^{t,0}_\varepsilon) + \frac{\eta}{2},
\end{equation*}
for $\varepsilon > 0$ sufficiently small. 

\bigskip

As $\{ g_\varepsilon \neq v_\varepsilon \} \subset B_{\frac{R}{2}}(p)$ we have $g^{t,s}_\varepsilon = v^{t,0}_\varepsilon$ on $N \setminus B_{\frac{R}{2}}(p)$, so using the fact that in $B_R(p)$ we have $v^{t,0}_\varepsilon = v_\varepsilon$ and $g^{t,2}_\varepsilon = g_\varepsilon$ we thus note that
\begin{equation*}
    \mathcal{F}_{\varepsilon,\lambda}(g^{t,2}_\varepsilon) = \mathcal{F}_{\varepsilon,\lambda}(v^{t,0}_\varepsilon) + (\mathcal{F}_{\varepsilon,\lambda}(g_\varepsilon) - \mathcal{F}_{\varepsilon,\lambda}(v_\varepsilon)).
\end{equation*}
Using this, if the functions $v_\varepsilon$ are such that $\mathcal{F}_{\varepsilon,\lambda}(v_\varepsilon) \geq \mathcal{F}_{\varepsilon,\lambda}(g_\varepsilon) + \tau$ for some $\tau > 0$, we therefore conclude that
\begin{equation*}
    \mathcal{F}_{\varepsilon,\lambda}(g^{t,2}_\varepsilon) \leq  \mathcal{F}_{\varepsilon,\lambda}(v^{t,0}_\varepsilon) -\tau,
\end{equation*}
as desired.
\end{proof}

\section{Proof of Theorems \ref{thm: generic regularity in dimension 8 in positive ricci},  \ref{thm: generic removability of isolated singularities with regular cones in positive ricci} \& \ref{thm: local minimisation in positive ricci}}\label{sec: conclusion of the proof}

We recall our setup. Let $(N,g)$ be a smooth compact Riemannian manifold of dimension $3$ or higher with positive Ricci curvature. We consider $M \subset N$ a closed embedded hypersurface of constant mean curvature $\lambda \in \mathbb{R}$, smooth away from a closed singular set of Hausdorff dimension at most $n-7$, as produced by the one-parameter Allen--Cahn min-max in \cite{bellettini-wickramasekera}, with constant prescribing function $\lambda$. 

\subsection{Proof of Theorem \ref{thm: local minimisation in positive ricci}}\label{subsec: proof of theorem 3}

\begin{proof}[Proof of Theorem \ref{thm: local minimisation in positive ricci}]
For each isolated singularity $p \in \mathrm{Sing}(M)$ we may choose, positive $r_0$, $r$ and $R$, as in Subsection \ref{subsec: choosing radii} and Proposition \ref{prop: local path} in order to define the various paths constructed in Subsections \ref{subsec: defining shifted functions}, \ref{subsec: sliding the one-dimensional profile on M} and \ref{subsec: local path}. Assuming that for $g_\varepsilon \in \mathcal{A}_{\varepsilon,\frac{R}{2}}(p)$, defined as in Lemma \ref{lem: function minimisation} (setting $\rho = \frac{R}{2}$ and $q = p$), there exists a $\tau > 0$ such that $\mathcal{F}_{\varepsilon,\lambda}(v_\varepsilon) \geq \mathcal{F}_{\varepsilon,\lambda}(g_\varepsilon) + \tau$, for all $\varepsilon > 0$ sufficiently small, then using the constructions in Section \ref{sec: paths} as mentioned above we define the following nine paths:

\begin{itemize}
    \item First, a path from $+1$ to the constant $b_\varepsilon$,
    \begin{equation*}
        s \in [1,b_\varepsilon] \rightarrow s,
    \end{equation*}
    through constant functions which, by the construction of $b_\varepsilon$ through negative gradient flow of $\mathcal{F}_{\varepsilon,\lambda}$ in Subsection \ref{subsec: Allen--Cahn minmax} (c.f.~\cite[Section 5]{bellettini-wickramasekera}), has $\mathcal{F}_{\varepsilon,\lambda}$ energy along the path $\leq \mathcal{F}_{\varepsilon,\lambda}(1)$.
    
    \item Second, a path from $+1$ to $v^{-t_0}_\varepsilon = v^{-t_0,1}_\varepsilon$,
    \begin{equation*}
    t \in [-2d(N), -t_0] \rightarrow v^{t}_\varepsilon,
    \end{equation*}
    which by Lemma \ref{lem: sliding path upper energy bounds} has $\mathcal{F}_{\varepsilon,\lambda}$ energy along the path $\leq \mathcal{F}_{\varepsilon,\lambda}(v^{-t_0,1}) \leq \mathcal{F}_\lambda(E) - \eta + E(\varepsilon)$. This path varies continuously by the reasoning in Subsection \ref{subsec: sliding the one-dimensional profile on M}.

    \item Third, a path from $v^{-t_0,1}_\varepsilon$ to $v^{-t_0,0}_\varepsilon = g^{-t_0,-2}_\varepsilon$,
    \begin{equation*}
    s \in [0, 1] \rightarrow v^{-t_0,s}_\varepsilon,
    \end{equation*}
    which by Lemma \ref{lem: shifting path upper energy bounds} has $\mathcal{F}_{\varepsilon,\lambda}$ energy along the path $\leq \mathcal{F}_{\varepsilon,\lambda}(v_\varepsilon) -  \eta$. This path varies continuously by (\ref{eqn: sliding functions cts path in s}).

    \item Fourth, a path from $g^{-t_0,-2}_\varepsilon$ to $g^{-t_0,2}_\varepsilon$,
    \begin{equation*}
    s \in [-2,2] \rightarrow g^{-t_0,s}_\varepsilon,
    \end{equation*}
    which by Proposition \ref{prop: local path} and Lemma \ref{lem: shifting path upper energy bounds} has $\mathcal{F}_{\varepsilon,\lambda}$ energy along the path $\leq \mathcal{F}_{\varepsilon,\lambda}(v_\varepsilon) -  \frac{\eta}{2}$. This path varies continuously by Proposition \ref{prop: local path}.

    \item Fifth, a path from $g^{-t_0,2}_\varepsilon$ to $g^{t_0,2}_\varepsilon$,
    \begin{equation*}
    t \in [-t_0,t_0] \rightarrow g^{t,2}_\varepsilon,
    \end{equation*}
    which by Proposition \ref{prop: local path} and Lemma \ref{lem: shifting path upper energy bounds} has $\mathcal{F}_{\varepsilon,\lambda}$ energy along the path $\leq \mathcal{F}_{\varepsilon,\lambda}(v_\varepsilon) - \tau + E(\varepsilon)$. This path varies continuously by Proposition \ref{prop: local path}.

    \item Sixth, a path from $g^{t_0,-2}_\varepsilon = v^{t_0,0}_\varepsilon$ to $g^{t_0,2}_\varepsilon$,
    \begin{equation*}
    s \in [-2,2] \rightarrow g^{t_0,s}_\varepsilon,
    \end{equation*}
    which by  by Proposition \ref{prop: local path} and Lemma \ref{lem: shifting path upper energy bounds} has $\mathcal{F}_{\varepsilon,\lambda}$ energy along the path $\leq \mathcal{F}_{\varepsilon,\lambda}(v_\varepsilon) - \frac{\eta}{2}$. This path varies continuously by Proposition \ref{prop: local path}.

    \item Seventh, a path from         $v^{t_0,0}_\varepsilon$ to $v^{t_0,1}_\varepsilon = v^{t_0}_
    \varepsilon$,
    \begin{equation*}
    s \in [0,1] \rightarrow v^{t_0,s}_\varepsilon
    \end{equation*}
    which by Lemma \ref{lem: shifting path upper energy bounds} has $\mathcal{F}_{\varepsilon,\lambda}$ energy along the path $\leq \mathcal{F}_{\varepsilon,\lambda}(v_\varepsilon) -  \eta$. This path varies continuously by (\ref{eqn: sliding functions cts path in s}).

    \item Eighth, a path from $v^{t_0}_\varepsilon$ to $-1$,
    \begin{equation*}
    t \in [t_0,2d(N)] \rightarrow v^t_\varepsilon
    \end{equation*}
    which by Lemma \ref{lem: sliding path upper energy bounds} has $\mathcal{F}_{\varepsilon,\lambda}$ energy along the path $\leq \mathcal{F}_{\varepsilon,\lambda}(v_\varepsilon) - \eta + E(\varepsilon)$. This path varies continuously by the reasoning in Subsection \ref{subsec: sliding the one-dimensional profile on M}.

    \item Ninth, a path from $-1$ to the constant $a_\varepsilon$,
    \begin{equation*}
        s \in [-1,a_\varepsilon] \rightarrow s,
    \end{equation*}
    through constant functions which, by the construction of $a_\varepsilon$ through negative gradient flow of $\mathcal{F}_{\varepsilon,\lambda}$ in Subsection \ref{subsec: Allen--Cahn minmax} (c.f.~\cite[Section 5]{bellettini-wickramasekera}), has $\mathcal{F}_{\varepsilon,\lambda}$ energy along the path $\leq \mathcal{F}_{\varepsilon,\lambda}(-1)$.
\end{itemize}
Consider the above paths in the following order: first (reversed), second, third (reversed), fourth, fifth, sixth (reversed), seventh, eighth and ninth; this is the path depicted in Figure \ref{fig: path} in Subsection \ref{subsec: strategy}. In the order just given, the endpoint of each partial path matches the starting point of
the next, therefore their composition in the same order provides a continuous path
in $W^{1,2}(N)$, for all $\varepsilon > 0$ sufficiently small, from the constant $a_\varepsilon$ to the constant $b_\varepsilon$ with
\begin{equation*}
    \text{$\mathcal{F}_{\varepsilon,\lambda}$ energy along the path} \leq \mathcal{F}_{\varepsilon,\lambda}(v_\varepsilon) - \min\left\{\frac{\eta}{2},\tau \right\} + E(\varepsilon),
\end{equation*}
by Subsection \ref{subsec: Allen--Cahn minmax}, Lemma \ref{lem: shifting path upper energy bounds}, Lemma \ref{lem: sliding path upper energy bounds} and Proposition \ref{prop: local path}. By (\ref{eqn: 1d profile approx cmc energy}) and the fact that $E(\varepsilon) \rightarrow 0$ by Lemma \ref{lem: shifting path upper energy bounds}, by choosing  $\varepsilon > 0$ sufficiently small we ensure that we have
\begin{equation*}
    \text{$\mathcal{F}_{\varepsilon,\lambda}$ energy along the path} \leq \mathcal{F}_\lambda(E) - \min\left\{\frac{\eta}{4},\frac{\tau}{2} \right\}.
\end{equation*}
Note that as (\ref{eqn: min-max critical points converge to surface energy}) holds, by (\ref{eqn: 1d profile approx cmc energy}) and the path provided in Lemma \ref{lem: sliding path upper energy bounds}, we ensure $\mathcal{F}_{\varepsilon,\lambda}(u_{\varepsilon_j}) \rightarrow \mathcal{F}_{\lambda}(E)$ as $\varepsilon_j \rightarrow 0$. Thus, as the above path is admissible in the min-max construction of \cite{bellettini-wickramasekera}, we contradict the assumption that $ \mathcal{F}_{\varepsilon,\lambda}(v_\varepsilon) \geq \mathcal{F}_{\varepsilon,\lambda}(g_\varepsilon) + \tau$ for some $\tau > 0$. We therefore conclude that for any such $M$ as produced by the Allen--Cahn min-max procedure in Ricci positive curvature must be such that 
\begin{equation*}
   \mathcal{F}_{\varepsilon,\lambda}(v_\varepsilon) \leq \mathcal{F}_{\varepsilon,\lambda}(g_\varepsilon) + \tau_\varepsilon \text{ for some sequence } \tau_\varepsilon \rightarrow 0 \text{ as } \varepsilon \rightarrow 0.
\end{equation*}
With the above in mind and applying Lemma \ref{lem: recovery functions}, by setting $r_1 = \frac{R}{4}$ and $r_2 = \frac{R}{2}$, we have that $E$ satisfies
\begin{equation*}
    \mathcal{F}_\lambda(E) = \inf_{G \in \mathcal{C}(N)} \{\mathcal{F}_\lambda(G) \: | \: G \setminus B_\frac{R}{4}(p) = E \setminus B_\frac{R}{4} (p)\}.
\end{equation*}
In particular, by Remark \ref{rem: almost min area min cones}, we note that every tangent cone at an isolated singularity of $M$ is thus area-minimising. Applying the above reasoning for each isolated singularity of $M$ then concludes the proof of Theorem \ref{thm: local minimisation in positive ricci}.
\end{proof}

\subsection{Proof of Theorems \ref{thm: generic regularity in dimension 8 in positive ricci} \& \ref{thm: generic removability of isolated singularities with regular cones in positive ricci}}

\begin{proof}[Proof of Theorem \ref{thm: generic removability of isolated singularities with regular cones in positive ricci}]
To prove Theorem \ref{thm: generic removability of isolated singularities with regular cones in positive ricci} we exploit the results of Section \ref{sec: cmc surgery}, using the fact that the proof of Theorem \ref{thm: local minimisation in positive ricci} ensures $M$ is minimising in the sense of Definition \ref{def: one-sided minimiser} by Remark \ref{rem: surgery works on min-max cmc}. Because the set of isolated singular points of $M$ with regular tangent cones is discrete, but not necessarily closed when $n \geq 8$, it suffices to index the isolated singularities with regular tangent cones and make a small change to the metric around each point so that the sum of the resulting perturbations is arbitrarily small.

\bigskip

With the above in hand, we now make an arbitrarily small change to the metric at each isolated singular point, $p \in \mathrm{Sing}(M)$, with regular tangent cone. By applying Proposition \ref{prop: CMC surgery} (after rescaling) in a ball, $B_{\rho}(p)$, for some $\rho > 0$ sufficiently small, there exists both:
\begin{itemize}
    \item A metric, $\tilde{g}$, arbitrarily close to $g$ in the $C^{k,\alpha}$ norm for each $k \geq 1$ and $\alpha \in (0,1)$, agreeing with $g$ on $N \setminus B_{\rho}(p)$.

    \item A closed embedded hypersurface, $\widetilde{M}$, of constant mean curvature $\lambda$, which is smooth in $B_{\rho}(p)$, and agrees with $M$ on $N \setminus B_{\rho}(p)$.
\end{itemize}
In this manner we are able to locally smooth $M$ up to an arbitrarily small perturbation of the metric $g$. Thus we have shown that for each $k \geq 1$ and $\alpha \in (0,1)$ there exists a dense set of metrics, $\mathcal{G}_k \subset \text{Met}^{k,\alpha}_{\mathrm{Ric}_g > 0}$, such that for each $h \in \mathcal{G}_k$, $(N,h)$ admits a closed embedded hypersurface of constant mean curvature $\lambda$, smooth away from a closed singular set of Hausdorff dimension at most $n-7$, containing no singularities with regular tangent cones. By \cite[Theorem 2.10]{White2015} there thus exists a dense set $\mathcal{G}$ of the smooth metrics with Ricci positive curvature such that for each $h \in \mathcal{G}$, $(N,h)$ admits a closed embedded hypersurface of constant mean curvature $\lambda$, smooth away from a closed singular set of Hausdorff dimension at most $n-7$, containing no singularities with regular tangent cones, concluding the proof of Theorem \ref{thm: generic removability of isolated singularities with regular cones in positive ricci}.
\end{proof}

\begin{proof}[Proof of Theorem \ref{thm: generic regularity in dimension 8 in positive ricci}]
In dimension $8$ all singularities are isolated with regular tangent cone. Hence, for each $g \in \mathcal{G}$, as produced in the proof of Theorem \ref{thm: generic removability of isolated singularities with regular cones in positive ricci}, there exists a smooth hypersurface of constant mean curvature. The fact that $\mathcal{G}$ as above is then open in dimension $8$ follows from the results of \cite[Section 7]{White1991}.
\end{proof}

\bibliographystyle{alpha} 
\bibliography{main} 
\hrule
\Addresses

\end{document}